\documentclass[usenames,dvipsnames]{article}
\usepackage{amssymb}
\usepackage{amsmath}
\usepackage{amsthm}
\usepackage{mathtools,amscd,url}
\usepackage{tikz}
\usepackage[utf8]{inputenc}
\usepackage{csquotes}
\usepackage{latexsym}
\usepackage{tikz}
\usetikzlibrary{decorations.pathreplacing,arrows,decorations.pathmorphing,positioning}
\usepackage{fullpage}
\usepackage[ruled, vlined]{algorithm2e}
\usepackage{graphicx}
\usepackage{ytableau}
\usepackage{multirow}

\newcommand{\SOdrei}{\mathrm{SO}(3)}
\newcommand{\SOodd}{\mathrm{SO}(2k+1)}
\newcommand{\SOn}{\mathrm{SO}(n)}
\newcommand{\LRcoef}{\mathrm{c}_{\lambda}^{\mu}(\mathfrak{d})}
\newcommand{\LRtabs}{\mathrm{LR}_\lambda^{\mu}(\mathfrak{d})}
\newcommand{\set}[2]{\{#1:#2\}}
\newcommand{\Vtensr}{V^{\otimes r}}
\newcommand{\GL}{\mathrm{GL}}
\newcommand{\SO}{\mathrm{SO}}
\newcommand{\Des}{\mathrm{Des}}
\newcommand{\Symr}{\mathfrak{S}_r}
\newcommand{\SYT}{\mathrm{SYT}}

\newtheorem{theorem}{Theorem}
\newtheorem{lemma}[theorem]{Lemma}
\newtheorem{corollary}[theorem]{Corollary}
\newtheorem{proposition}[theorem]{Proposition}
\theoremstyle{definition}
\newtheorem{definition}[theorem]{Definition}
\theoremstyle{remark}
\newtheorem{remark}[theorem]{Remark} 
\newtheorem{example}[theorem]{Example}
\newtheorem{conjecture}[theorem]{Conjecture}

\title{A Sundaram type bijection for $\SOdrei$:\\ vacillating tableaux and pairs of standard Young tableaux and orthogonal Littlewood-Richardson tableaux}
\author{Judith Jagenteufel\footnote{Institut für Diskrete Mathematik und Geometrie, Fakult\"at f\"ur Mathematik und Geoinformation, TU Wien, Austria,\newline
Supported by the European Research Council (ERC): 306445 and the Austrian science fund (FWF): P29275} }

\begin{document}

\maketitle

\abstract{Motivated by the direct-sum-decomposition of the $r^{\text{th}}$ tensor power of the defining representation of the special orthogonal group $\mathrm{SO}(2k+1)$, we present a bijection between vacillating tableaux and pairs consisting of a standard Young tableau and an orthogonal Littlewood-Richardson tableau for $\mathrm{SO}(3)$.

Our bijection preserves a suitably defined descent set. Using it we determine the quasi-symmetric expansion of the Frobenius characters of the isotypic components.

On the combinatorial side we obtain a bijection between Riordan paths and standard Young tableaux with at most 3 rows, all of even length or all of odd length.
}

\section{Introduction}

We solve a problem posed by Sundaram in her 1986 thesis~\cite{MR2941115} concerning the direct-sum-decomposition into irreducibles of tensor powers of the defining representation (also known as vector representation) of the special orthogonal group $\SO(3)$.
In particular, our main result is a bijection for $\SO(3)$ between so called vacillating tableaux and pairs consisting of an orthogonal Littlewood-Richardson tableau (recently introduced by Kwon~\cite{2015arXiv151201877K}) and a standard Young tableau.

The question of finding such a bijection arises when we consider the decomposition of a tensor power $V^{\otimes r}$ of the defining representation $V$ of $\SO(2k+1)$ as an $\SO(2k+1)\times\Symr$ representation,
\[\Vtensr=\bigoplus_{\mu} V(\mu)\otimes U(r,\mu)= \bigoplus_{\mu} V(\mu)\otimes \bigoplus_{\lambda}  c^{\mu}_{\lambda} S(\lambda),\]
where $V(\mu)$ is an irreducible representation of $\SO(2k+1)$ and $S(\lambda)$ is a Specht module. Considering the decomposition, we concentrate on $U(r,\mu)$. A basis of $U(r,\mu)$ can be indexed by vacillating tableaux. In dimension $3$ these can be regarded as Riordan paths: lattice paths with north-east, east and south-east steps, no steps below the $x$-axis and no east steps on the $x$-axis. The multiplicities $c^{\mu}_{\lambda}$ can be obtained by counting orthogonal Littlewood-Richardson tableaux, while a basis of $S(\lambda)$ is indexed by standard Young tableaux.

Our bijection is composed of two independent bijections. The first one, described by Algorithm~\ref{algoLR3} and its inverse Algorithm~\ref{UalgoLR3}, considers orthogonal Littlewood-Richardson tableaux and standard Young tableaux and reduces the problem to the case of vacillating tableaux of shape $\emptyset$. To formulate it we give an explicit combinatorial description of Kwon's orthogonal Littlewood-Richardson tableaux, which are defined in a very general way, although, too abstract for our purposes.
The second bijection, described by Algorithm~\ref{algo3dim} and its inverse Algorithm~\ref{Ualgo3dim}, is a bijection between vacillating tableaux of dimension $3$ of shape $\emptyset$ and standard Young tableaux with 3 rows, all of even length or all of odd length.

The problem of finding a bijection explaining the decomposition of $U(r,\mu)$ has been attacked several times since Sundaram's thesis; in particular by Sundaram~\cite{MR1041447}, Proctor~\cite{MR1043509} and recently by Okada~\cite{MR3604801}. Okada implicitly obtains the decomposition of $U(r,\mu)$ for the multiplicity free cases $\mu=\emptyset$ and $\mu=(1)$ using representation theoretic computations in~\cite[Theorem 5.3(3)]{MR3604801} and proves that the associated combinatorial objects have the same cardinality in~\cite[Corollary 5.5(3)]{MR3604801}. Our bijection provides a proof for these results as a special case, which are on their part special cases of Okada's work. In fact, Okada asks for bijective proofs of his results. He speculates that Fomin's machinery of growth diagrams could be employed to find such bijections. Indeed, for the symplectic group, special cases of his results were obtained by Roby~\cite{MR2716353}  and Krattenthaler~\cite{MR3534070} using this method. However, for the orthogonal group the situation appears to be quite different. In particular, at least a naive application of Fomin's ideas does not even yield the desired bijection between vacillating tableaux and the set of standard Young tableaux in question.

An advantage of our combinatorial, bijective proof is that we obtain several additional properties and consequences such as the following.

Our bijection preserves a suitably defined notion of descents: for vacillating tableaux the definition is new, whereas for standard Young tableaux we use the classical descent set introduced by Schützenberger. Thus, we obtain the quasi-symmetric expansion of the Frobenius character of the isotypic space $ U(r,\mu)$,
\[\mathrm{ch} \, U(r,\mu)=\sum_{\substack{w \text{ a vacillating tableau}}} F_{\Des(w)},\] 
where $F_D$ denotes a fundamental quasi-symmetric function.

Among others, this property justifies our bijection to be called \enquote{Sundaram-like}, as she described a similar bijection for the defining representation of the symplectic group in her thesis~\cite{MR2941115}. There exists a similar (but less complicated) definition for descents in oscillating tableaux, which are used in the symplectic case instead of vacillating tableaux, and which Sundaram's bijection preserves. Thus there also exists a similar quasi-symmetric expansion of the Frobenius character. Rubey, Sagan and Westbury obtained this result in~\cite{MR3226822}.

An essentially different description of our second bijection proves further properties.

Vacillating tableaux of shape $\emptyset$ can be regarded as paths that end at the $x$-axis, and can thus be concatenated naively. We define concatenation on standard Young tableaux by writing them side by side and adjusting entries to obtain the following: concatenation of vacillating tableaux corresponds to concatenation of standard Young tableaux (see Theorem~\ref{TheoConcat}). A similar property holds for Sundaram's bijection.

Ignoring the orthogonal Littlewood-Richardson tableaux, our bijection describes an insertion algorithm of vacillating tableaux to standard Young tableaux describing the decomposition of $U(r,\mu)$ (see Corollary~\ref{InsertionCoro}). So does Sundaram's bijection for the symplectic group, mapping oscillating tableaux to standard Young tableaux. Due to the exceptional isomorphism $\mathfrak{sl}_2\cong\mathfrak{so}_3$, our algorithm can also be regarded as an insertion algorithm for the decomposition of the isotypic components in tensor powers of the adjoint representation of $\mathrm{SL}_2$.
Thus, our algorithm is a first step towards a resolution of a closely
related problem of Stembridge.  He shows in~\cite[Theorem 6.2]{MR899903}
that such an algorithm must also exist for the adjoint representation of
$\mathrm{GL}_n$, but so far, no explicit description has been found.
\section{Background}
\label{SectionBackground}

Although our bijection considers only $k=1$ (thus $n=2k+1=3$) we provide the setup for general $k$, as this ensures a better understanding of the problem.

\subsection{Schur-Weyl duality} 
Starting with the general linear group we have \enquote{classical Schur-Weyl duality}
\begin{align*}
\Vtensr \cong \bigoplus_{\substack{\lambda \vdash r \\ \ell(\lambda) \leq n}} V^{\GL}(\lambda) \otimes S(\lambda)
\end{align*}
as $\GL(V)\times \Symr$ representations where $V$ is a complex vector space of dimension $n$ and
\begin{itemize}
\item $\GL(V)$ acts diagonally (and on each position by matrix multiplication) and $\Symr$ permutes tensor positions
\item $V^{\GL}(\lambda)$ is an irreducible representation of $\GL(V)$ and $S(\lambda)$ is a Specht module.
\end{itemize}
Consider now the restriction from $\GL(V)$ to $\SO(V)$ for a vector space $V$ of odd dimension $n=2k+1$
\begin{align*}
V(\lambda)\downarrow^{\GL(V)}_{\SO(V)} \cong \bigoplus_{\substack{\mu \text{ a partition} \\ \ell(\mu) \leq k}} c_{\lambda}^{\mu}(\mathfrak{d}) V^{SO}(\mu )
\end{align*}
where $\LRcoef$ is the multiplicity of the irreducible representation $V^{\SO}(\mu)$ of $\SO(V)$ in $V^{\GL}(\lambda)$. For $\ell(\lambda)\leq k$ this simplifies to the classical branching rule due to Littlewood. We present Kwon's definition of orthogonal Littlewood-Richardson tableaux, a set that is counted by $\LRcoef$, in Section~\ref{SubSecLRTabs}.

Combining Schur-Weyl duality and the branching rule we obtain an isomorphism of $\SO(V)\times \Symr$ representations

\begin{align*}
\Vtensr\cong \bigoplus_{\substack{\lambda \vdash r \\ \ell(\lambda) \leq n}} \Big( \bigoplus_{\substack{ \mu \text{ a partition}\\ \ell(\mu) \leq k}} c_{\lambda}^{\mu}(\mathfrak{d}) V^{SO}(\mu ) \Big) \otimes S(\lambda)
 = \bigoplus_{\substack{\mu \text{ a partition} \\ \ell(\mu) \leq k}} V^{SO}(\mu )\otimes U(r,\mu)
\end{align*}
 with isotypic components of weight $\mu$
\begin{align}
\label{eq_identity}
U(r,\mu)=\bigoplus_{\substack{\lambda \vdash r \\ \ell(\lambda) \leq n}}  \LRcoef S(\lambda).
\end{align}
 The isomorphism of $\SO(V)$ representations (e.g. Okada~\cite[Cor. 3.6]{MR3604801}),
\begin{align*}
V^{SO}(\mu )\otimes V\cong\bigoplus_{\substack{\ell(\lambda)\leq k\\ \lambda=\mu\pm \square\\ \text{or } \lambda=\mu \text{ and } \ell(\mu)=k}} V^{SO}(\lambda)
\end{align*}
implies that a basis of $U(r,\mu)$ can be indexed by so called vacillating tableaux of shape $\mu$, defined in Section~\ref{SubSecVacTab}.

Therefore, to explain the decomposition \eqref{eq_identity} combinatorially, we are interested in a bijection between vacillating tableaux and pairs that consist of a standard Young tableau and an orthogonal Littlewood-Richardson tableau, as a basis of $S(\lambda)$ can be indexed with standard Young tableaux.

Moreover, we go further and follow the approach taken by Rubey, Sagan and Westbury~\cite{MR3226822} for the symplectic group, by introducing descent sets for vacillating tableaux; see Section~\ref{SubSecDesc}. This enables us to describe the quasi-symmetric expansion of the Frobenius character. For more information on the Frobenius character see the text book by Stanley~\cite{MR1676282}.

Recall that the Frobenius character of a representation $\rho$ of the symmetric group is
\[ \mathrm{ch} \,\rho = \frac{1}{r!} \sum_{\pi \in \Symr} \mathrm{Tr}\, \rho(\pi) p_{\lambda(\pi)}\]
where $p_\lambda$ denotes the power sum symmetric function associated with $\lambda$.

For $D\subseteq\{1,2,\dots,r\}$ the \emph{fundamental quasi-symmetric function} $F_D$ is defined as 
\[F_D=\sum_{\substack{i_1\leq i_2\leq \dots\leq i_r\\ j\in D \Rightarrow i_j<i_{j+1}}} x_{i_1}x_{i_2}\dots x_{i_r}.\]
Recall that
\[\mathrm{ch}\, S(\lambda)=s_\lambda=\sum_{Q\in \SYT(\lambda)} F_{\Des(Q)}\]
where $s_\lambda$ denotes a Schur function and $\Des(Q)$ is the descent set of a standard Young tableau, see Section~\ref{SubSecDesc}. The Frobenius character can be also defined by this equation and the requirement that it be an isometry.

For $\SO(3)$ we define a descent-preserving bijection. Therefore we obtain the following theorem.
\begin{theorem}
\[\mathrm{ch} \,U(r,\mu)=\sum F_{\Des(w)},\]
where the sum runs over all vacillating tableaux $w$ of length $r$ and shape $\mu$ and $\Des(w)$ is the descent set of $w$.
\end{theorem}

 The extension to the general case, $\SO(2k+1)$, is work in progress and will be presented in another paper.

\subsection{Vacillating Tableaux}
\label{SubSecVacTab}

We define two equivalent families of objects: vacillating tableaux (as defined by Sundaram in~\cite[Def. 4.1]{MR1041447}) and highest weight words.
\begin{definition}
\begin{enumerate}
\item A ($(2k+1)$-\emph{orthogonal}) \emph{vacillating tableau} of length $r$ is a sequence of partitions $\emptyset=\mu^0, \mu^1,\dots,\mu^r=\mu$ each of at most $k$ parts, such that:
\begin{itemize}
\item $\mu^i$ and $\mu^{i+1}$ differ in at most one position,
\item $\mu^i=\mu^{i+1}$ only occurs if the $k^{\text{th}}$ part of both partitions is non-empty.
\end{itemize}
We call the final partition $\mu$ the \emph{shape} of the tableau.
\item A ($(2k+1)$-\emph{orthogonal}) \emph{highest weight word} is a word $w$ with letters in $\{\pm 1, \pm 2, \dots, \pm k, 0\}$ of length $r$ such that for every initial segment $s$ of $w$ the following holds (we write $\#i$ for the number of $i$'s in $s$):
\begin{itemize}
\item $\#i - \#(-i) \geq 0$,
\item $\#i-\#(-i)\geq \#(i+1)-\#(-i-1)$,
\item if the last letter is $0$ then $\#k-\#(-k)>0$.
\end{itemize}
We call the partition $(\#1-\#(-1),\#2-\#(-2),\dots, \#k - \#(-k))$ the \emph{weight} of a highest weight word. The vacillating tableau corresponding to a word $w$ is the sequence of weights of the initial segments of $w$.
\end{enumerate}

By abuse of terminology we refer to both objects as \emph{vacillating tableaux}.
\end{definition}

For $k=1$ we can regard vacillating tableaux as certain lattice paths (Riordan paths: Motzkin paths without horizontal steps on the $x$-axis). We therefore take a north-east step for every $1$, a south-east step for every $-1$ and an east step for every $0$. In this setting we want the path to be in the first quadrant and do not allow east steps on the $x$-axis.

\begin{example}
\label{ExamVacTab}
The same object once written as a vacillating tableau, once as a highest weight word and once as a Riordan path.

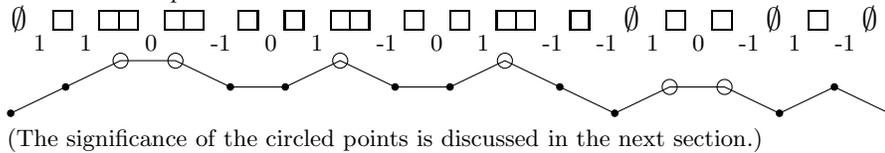
\ytableausetup{smalltableaux}
$
\begin{smallmatrix}
\text{\large{$\emptyset$}} && \ydiagram{1} && \ydiagram{2} && 
\ydiagram{2} && \ydiagram{1} && \ydiagram{1} &&
\ydiagram{2} && \ydiagram{1} && \ydiagram{1} &&
\ydiagram{2} && \ydiagram{1} && \text{\large{$\emptyset$}} && 
\ydiagram{1} && \ydiagram{1} && \text{\large{$\emptyset$}} && 
\ydiagram{1} && \text{\large{$\emptyset$}}\\
&\textstyle{1}& &\textstyle{1}& &\textstyle{0}& 
&\textstyle{\textrm{-}1}& &\textstyle{0}& &\textstyle{1}&
&\textstyle{\textrm{-}1}& &\textstyle{0}& &\textstyle{1}&
&\textstyle{\textrm{-}1}& &\textstyle{\textrm{-}1}& &\textstyle{1}&
&\textstyle{0}& &\textstyle{\textrm{-}1}& &\textstyle{1}& &\textstyle{\textrm{-}1}&
\end{smallmatrix}
$

\begin{tikzpicture}[yscale=0.35, xscale=0.73]
\draw (0,0) node[fill,circle,inner sep=0cm,minimum size=0.1cm] {}
-- (1,1) node[fill,circle,inner sep=0cm,minimum size=0.1cm] {}
-- (2,2) node[draw=black,circle,inner sep=0cm,minimum size=0.2cm] {}
-- (3,2) node[draw=black,circle,inner sep=0cm,minimum size=0.2cm] {}
-- (4,1) node[fill,circle,inner sep=0cm,minimum size=0.1cm] {}
-- (5,1) node[fill,circle,inner sep=0cm,minimum size=0.1cm] {}
-- (6,2) node[draw=black,circle,inner sep=0cm,minimum size=0.2cm] {}
-- (7,1) node[fill,circle,inner sep=0cm,minimum size=0.1cm] {}
-- (8,1) node[fill,circle,inner sep=0cm,minimum size=0.1cm] {}
-- (9,2) node[draw=black,circle,inner sep=0cm,minimum size=0.2cm] {}
-- (10,1) node[fill,circle,inner sep=0cm,minimum size=0.1cm] {}
-- (11,0) node[fill,circle,inner sep=0cm,minimum size=0.1cm] {}
-- (12,1) node[draw=black,circle,inner sep=0cm,minimum size=0.2cm] {}
-- (13,1) node[draw=black,circle,inner sep=0cm,minimum size=0.2cm] {}
-- (14,0) node[fill,circle,inner sep=0cm,minimum size=0.1cm] {}
-- (15,1) node[fill,circle,inner sep=0cm,minimum size=0.1cm] {}
-- (16,0) node[fill,circle,inner sep=0cm,minimum size=0.1cm] {};
\end{tikzpicture}

(The significance of the circled points is discussed in the next section.)
\end{example}

\subsection{Descents}
\label{SubSecDesc}

Following Rubey, Sagan and Westbury~\cite{MR3226822} we define the descent set for standard Young tableaux and vacillating tableaux.

\begin{definition}
Let $\SYT(\lambda)$ be the set of standard Young tableaux of shape $\lambda$. For a standard Young tableau $Q\in \SYT(\lambda)$ we call an entry $j$ a \emph{descent} if $j+1$ is in a row below $j$. We define the \emph{descent set} of $Q$ as:
\[\Des(Q)=\set{j}{j \text{ is a descent of } Q}.\]
\end{definition}
\begin{example}
The standard Young tableau
\[\begin{tikzpicture}[scale=0.35]
  \draw (0,4) -- (4,4);
  \draw (0,3) -- (4,3);
  \draw (0,2) -- (3,2);
  \draw (0,1) -- (2,1);
  \draw (0,4) -- (0,1);
  \draw (1,4) -- (1,1);
  \draw (2,4) -- (2,1);
  \draw (3,4) -- (3,2);
  \draw (4,4) -- (4,3);
  \draw (0.5,3.5) node {1};
  \draw (1.5,3.5) node {2};
  \draw (0.5,2.5) node {3};
  \draw (0.5,1.5) node {4};
  \draw (1.5,2.5) node {5};
  \draw (2.5,3.5) node {6};
  \draw (1.5,1.5) node {7};
  \draw (2.5,2.5) node {8};
  \draw (3.5,3.5) node {9};
\end{tikzpicture}\]
has descent set $\{2,3,6\}$.
\end{example}

\begin{definition}[Rubey, personal conversation]
We define descents for vacillating tableaux using highest weight words. We call a letter $j$ of $w$ a \emph{descent} if there exists a directed path from $w_j$ to $w_{j+1}$ in the crystal graph for the defining representation of $\SOodd$
\[  1\to 2\to \dots\to k \to 0 \to -k \to \dots \to -1\]
and $w_jw_{j+1} \neq j(-j)$ if for the initial segment $w_1,\dots,w_{j-1}$ holds $\#j-\#(-j)=0$.

 For $k=1$ this means that $j$ is a descent if the $j^{\text{th}}$ position is a $1$ followed by a $0$, or a $0$ followed by a $-1$, or a $1$ followed by a $-1$ if the numbers of $1$'s and $-1$'s before the $j^{\text{th}}$ position are different.

We define the \emph{descent set} of $w$ as
\[ \Des(w)= \set{j}{ j \text{ is a descent of } w}. \]
\end{definition}

The vacillating tableau in Example~\ref{ExamVacTab} has descent set $\{2,3,6,9,12,13\}$. The corresponding positions are circled.


\subsection{Crystal Graphs}
\label{CrystalGraphs}

In this section we summarize some properties of crystal graphs. In particular, we describe a certain crystal graph, that we need for defining orthogonal Littlewood-Richardson tableaux. For more information on crystals see the text book by Hong and Kang~\cite{MR1881971}.

Crystal graphs are certain acyclic directed graphs where vertices have finite in and out degree and each edge is labeled by a natural number. We only use crystal graphs whose vertices are labeled with certain tableaux.

For each vertex $C$ there is at most one outgoing edge labeled with $i$. If such an edge exists we denote its target by $f_i(C)$. Otherwise $f_i(C)$ is defined to be the distinguished symbol $\varnothing$.
Analogously there is at most one incoming edge labeled with $i$ and we define $e_i(C)$ as the tableau obtained by following an incoming edge labeled with $i$. We denote by $\varphi_i(C)$ (respectively $\varepsilon_i(C)$) the number of times one can apply $f_i$ (respectively $e_i$) to $C$.

We consider infinite crystal graphs. However, for the crystal graphs we consider, it holds that if we fix a natural number $\ell$ and delete all edges labeled with $\ell$ or larger, as well as all vertices that have incoming edges labeled with $\ell$ or larger, we obtain a finite crystal graph. Thus a lot of properties proven for finite ones hold also for our infinite crystal graphs.

\begin{definition}
The crystal graph of one-column tableaux is defined as follows:
\begin{enumerate}
\item The vertices are column strict tableaux with a single column and positive integers as entries.
\item Suppose that $i\in \mathbb{N}, i>0$ is an entry in a tableau $C$ but $i+1$ is not. Then $f_i(C)$ is the tableau one obtains by replacing $i$ by $i+1$.
Otherwise $f_i(C)=\varnothing$.
\item Suppose that neither $1$ nor $2$ is an entry in a tableau $C$. Then $f_0(C)$ is the tableau one obtains by adding a domino  \begin{ytableau}1\\ 2\end{ytableau} on top of $C$. Otherwise $f_0(C)=\varnothing$.
\end{enumerate}
See Figure~\ref{figCrystal1} for an example.
\end{definition}

\begin{figure}
\centering
\begin{tikzpicture}[scale=0.35]
\draw (-6.1,-3.4) node{$\vdots$};
\draw (-2.5,-3.4) node{$\vdots$};
\draw (1.5,-3.4) node{$\vdots$};
\draw (3.5,-4.4) node{$\vdots$};
\draw (5.5,-5.4) node{$\vdots$};
\draw [<-] (5.5,-4.9) -- (3.7,-2.1) node[midway, above]{\footnotesize 0};
\draw [<-] (1.7,-2.9) -- (3.3,-2.1) node[midway, above]{\footnotesize 5};
\draw [<-] (3.5,-3.9) -- (3.5,-2.1) node[midway, left]{\footnotesize 3};
\draw [<-] (1.3,-2.9) -- (-0.3,-2.1) node[midway, above]{\footnotesize 2};
\draw [<-] (-2.5,-2.9) -- (-0.7,-2.1) node[midway, above]{\footnotesize 6};
\draw [<-] (-2.7,-2.9) -- (-4.3,-2.1) node[midway, above]{\footnotesize 1};
\draw [<-] (-6.3,-2.9) -- (-4.7,-2.1) node[midway, above]{\footnotesize 7};
\draw [<-] (5.7,-4.9) -- (7.5,-4.1) node[midway, above]{\footnotesize 4};
\draw (7,-4) -- (7,0);
\draw (8,-4) -- (8,0);
\draw (7,-4) -- (8,-4);
\draw (7,-3) -- (8,-3);
\draw (7,-2) -- (8,-2);
\draw (7,-1) -- (8,-1);
\draw (7,0) -- (8,0);
\draw (7.5,-3.5) node{4};
\draw (7.5,-2.5) node{3};
\draw (7.5,-1.5) node{2};
\draw (7.5,-0.5) node{1};
\draw (3,-2) -- (3,0);
\draw (4,-2) -- (4,0);
\draw (3,-2) -- (4,-2);
\draw (3,-1) -- (4,-1);
\draw (3,0) -- (4,0);
\draw (3.5,-1.5) node{5};
\draw (3.5,-0.5) node{3};
\draw (-1,-2) -- (-1,0);
\draw (0,-2) -- (0,0);
\draw (-1,-2) -- (0,-2);
\draw (-1,-1) -- (0,-1);
\draw (-1,0) -- (0,0);
\draw (-0.5,-1.5) node{6};
\draw (-0.5,-0.5) node{2};
\draw (-5,-2) -- (-5,0);
\draw (-4,-2) -- (-4,0);
\draw (-5,-2) -- (-4,-2);
\draw (-5,-1) -- (-4,-1);
\draw (-5,0) -- (-4,0);
\draw (-4.5,-1.5) node{7};
\draw (-4.5,-0.5) node{1};
\draw [<-] (3.5,0.1) -- (1.7,0.9) node[midway, above]{\footnotesize 2};
\draw [<-] (-0.3,0.1) -- (1.3,0.9) node[midway, above]{\footnotesize 5};
\draw [<-] (-0.7,0.1) -- (-2.3,0.9) node[midway, above]{\footnotesize 1};
\draw [<-] (-4.5,0.1) -- (-2.7,0.9) node[midway, above]{\footnotesize 6};
\draw [<-] (3.7,0.1) -- (5.3,0.9) node[midway, above]{\footnotesize 4};
\draw [<-] (7.5,0.1) -- (5.7,0.9) node[midway, above]{\footnotesize 0};
\draw (5,1) -- (5,3);
\draw (6,1) -- (6,3);
\draw (5,1) -- (6,1);
\draw (5,2) -- (6,2);
\draw (5,3) -- (6,3);
\draw (5.5,1.5) node{4};
\draw (5.5,2.5) node{3};
\draw (1,1) -- (1,3);
\draw (2,1) -- (2,3);
\draw (1,1) -- (2,1);
\draw (1,2) -- (2,2);
\draw (1,3) -- (2,3);
\draw (1.5,1.5) node{5};
\draw (1.5,2.5) node{2};
\draw (-3,1) -- (-3,3);
\draw (-2,1) -- (-2,3);
\draw (-3,1) -- (-2,1);
\draw (-3,2) -- (-2,2);
\draw (-3,3) -- (-2,3);
\draw (-2.5,1.5) node{6};
\draw (-2.5,2.5) node{1};
\draw [<-] (5.5,3.1) -- (3.7,3.9) node[midway, above]{\footnotesize 2};
\draw [<-] (1.7,3.1) -- (3.3,3.9) node[midway, above]{\footnotesize 4};
\draw [<-] (1.3,3.1) -- (-0.3,3.9) node[midway, above]{\footnotesize 1};
\draw [<-] (-2.5,3.1) -- (-0.7,3.9) node[midway, above]{\footnotesize 5};
\draw (3,4) -- (3,6);
\draw (4,4) -- (4,6);
\draw (3,4) -- (4,4);
\draw (3,5) -- (4,5);
\draw (3,6) -- (4,6);
\draw (3.5,4.5) node{4};
\draw (3.5,5.5) node{2};
\draw (-1,4) -- (-1,6);
\draw (0,4) -- (0,6);
\draw (-1,4) -- (0,4);
\draw (-1,5) -- (0,5);
\draw (-1,6) -- (0,6);
\draw (-0.5,4.5) node{5};
\draw (-0.5,5.5) node{1};
\draw [<-] (3.3,6.1) -- (1.7,6.9) node[midway, below]{\footnotesize 1};
\draw [<-] (3.7,6.1) -- (5.5,6.9) node[midway, below]{\footnotesize 3};
\draw [<-] (-0.5,6.1) -- (1.3,6.9) node[midway, below]{\footnotesize 4};
\draw (1,7) -- (1,9);
\draw (2,7) -- (2,9);
\draw (1,7) -- (2,7);
\draw (1,8) -- (2,8);
\draw (1,9) -- (2,9);
\draw (1.5,7.5) node{4};
\draw (1.5,8.5) node{1};
\draw (5,7) -- (5,9);
\draw (6,7) -- (6,9);
\draw (5,7) -- (6,7);
\draw (5,8) -- (6,8);
\draw (5,9) -- (6,9);
\draw (5.5,7.5) node{3};
\draw (5.5,8.5) node{2};
\draw [->] (3.3,9.9) -- (1.5,9.1) node[midway, above]{\footnotesize 3};
\draw [->] (3.8,9.9) -- (5.5,9.1) node[midway, above]{\footnotesize 1};
\draw (3,10) -- (3,12);
\draw (4,10) -- (4,12);
\draw (3,10) -- (4,10);
\draw (3,11) -- (4,11);
\draw (3,12) -- (4,12);
\draw (3.5,10.5) node{3};
\draw (3.5,11.5) node{1};
\draw [<-] (3.5,12.1) -- (3.5,13.9) node[midway, right]{\footnotesize 2};
\draw (3,14) -- (3,16);
\draw (4,14) -- (4,16);
\draw (3,14) -- (4,14);
\draw (3,15) -- (4,15);
\draw (3,16) -- (4,16);
\draw (3.5,14.5) node{2};
\draw (3.5,15.5) node{1};
\end{tikzpicture}
\caption{The component with even-length tableaux of a crystal consisting of one-column tableaux.}
\label{figCrystal1}
\end{figure}
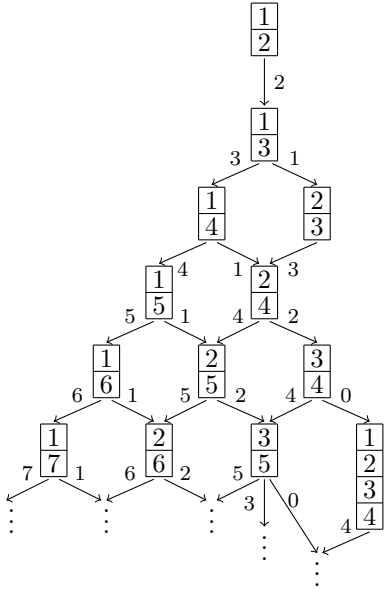

Next we introduce the tensor product of crystal graphs, using the left-to-right opposite of most of the literature.
The tensor product is associative.

\begin{definition}
The tensor product  $B_1\otimes B_2$  of two crystal graphs $B_1$ and $B_2$ is a crystal graph with vertex set $B_1 \times B_2$ and edges satisfying:
\begin{align*}
f_i(b\otimes b')&=\left\{ \begin{array}{ll} b\otimes f_i(b') & \text{if } \varepsilon_i(b)<\varphi_i(b')\\ f_i(b)\otimes b' & \text{otherwise} \end{array} \right.\\
e_i(b\otimes b')&=\left\{ \begin{array}{ll} e_i(b)\otimes b'  & \text{if } \varepsilon_i(b)>\varphi_i(b')\\ b\otimes e_i(b') & \text{otherwise} \end{array} \right.
\end{align*}
and
\begin{align*}
\varphi_i(b\otimes b')&= \varphi_i(b) + \mathrm{max}(0,\varphi_i(b')-\varepsilon_i(b))\\
\varepsilon_i(b\otimes b')&= \varepsilon_i(b') + \mathrm{max}(0,\varepsilon_i(b)-\varphi_i(b')).\\
\end{align*}
\end{definition} 


\begin{lemma}
\label{TensorLemma}
Consider a tensor $b=b_1\otimes b_2, \otimes \dots \otimes b_n$ and an index $i$.

We form a sequence of plus and minus signs according to $\varphi_i$ and $\varepsilon_i$ in the following way:

\begin{tabular}{c cc cc c cc c}
 $(\underbrace{+,+,\dots, +,}$ &$\underbrace{-,-,\dots, -,}$ &  $\underbrace{+,+,\dots, +,}$ &$\underbrace{-,-,\dots, -,}$ &$\dots $ &  $\underbrace{+,+,\dots, +,}$ &$\underbrace{-,-,\dots, -})$ \\
 $\varphi_i(b_1)$ & $\varepsilon_i(b_1)$ &$\varphi_i(b_2)$ & $\varepsilon_i(b_2)$ & $\dots$ & $\varphi_i(b_n)$ & $\varepsilon_i(b_n)$ 
\end{tabular}

Now we cancel out adjacent pairs of \enquote{$-,+$} until we obtain a new sequence $(+,+,\dots,+,-,-,\dots,-)$.

Then $f_i$ acts on the $b_j$ corresponding to the rightmost \enquote{$+$} in the new sequence and $e_i$ on the leftmost  \enquote{$-$}. Moreover, the number of  \enquote{$+$}s in the new sequence is $\varphi_i(b)$ and the number of  \enquote{$-$}s is $\varepsilon_i(b)$.
\end{lemma}

\begin{definition}
The crystal graph $B^{\mathfrak{d}}(\mu)$ where $\mu$ is a partition with at most $k=1$ part, is defined as follows.
The vertices are pairs of tableaux $(T,C)$ where $T$ is a two-column skew semistandard Young tableau and $C$ is a column strict single column such that:
\begin{itemize}
\item The \emph{tail} of a two-column skew semistandard tableau is the part of the tableau where only the first column exists. We require that the tail of $T$ have length $\mu_1$.
\item The \emph{residuum} of a two-column skew semistandard tableau is the number of positions we can shift the second column down such that the result is still a skew semistandard tableau. We require that the residuum of $T$ be less than or equal to 1.
\end{itemize}
To obtain the edges, we consider this as a subgraph of the tensor product of $3=2k+1$ one-column crystal graphs. Note that taking this point of view $\varphi_i(b_\ell)\in\{1,0\}$ for $\ell=1,2,3$.
\end{definition}

\begin{proposition}[Defining Proposition]
\label{UniqMax}
If the edges of a crystal graph are considered as a partial order, every connected component of the crystal graph has a unique maximum (see Kashiwara~\cite{ONCRYSTALBASES}). We call those maxima the \emph{highest weight elements}.
\end{proposition}

\subsection{Orthogonal Littlewood-Richardson Tableaux}
\label{SubSecLRTabs}

In order to define orthogonal Littlewood-Richardson tableaux we define a crystal graph, of which they are a part. A proof of the following lemma can be found in Kwon~\cite[Sec. 3, pages 8, 9 and 13]{2015arXiv151201877K}.

\begin{lemma}[Defining Lemma, Kwon]
\label{LRDefLem}
For $n=3$ the crystal graph $T^\mathfrak{d}(\mu)$ consists of those elements of $B^{\mathfrak{d}}(\mu)$ described in Subsection~\ref{CrystalGraphs} that are in the same component as one of the following highest weight elements:
\begin{itemize}
\item The two-column skew tableau consists only of one column filled with $1,2,\dots, \mu_1$. The one-column tableau is empty. (This belongs to Case 1 in the next section.)
\item The two-column skew tableau consists only of one column filled with $1,2,\dots, \mu_1>0$. The one-column tableau contains exactly one cell filled with $1$. (This belongs to Case 2 or 3 in the next section.)
\item The two-column skew tableau consists of two cells in a row, both containing $1$. The one-column tableau contains exactly one cell containing $1$. (This belongs to Case 4 in the next section.)
\end{itemize}
\end{lemma}
See Figure~\ref{figCrystalHighWeights} for images of the highest weights elements.

\begin{figure}
\centering
\begin{tikzpicture}[scale=0.35]
\draw (0,1) -- (0,4);
\draw (1,1) -- (1,4);
\draw (0,1) -- (1,1);
\draw (0,2) -- (1,2);
\draw (0,3) -- (1,3);
\draw (0,4) -- (1,4);
\draw (0.5,1.5) node{$\mu_1$};
\draw (0.5,3.5) node{1};
\draw (0.5,2.78) node{$\vdots$};
\draw (1.5,4.5) node{$\emptyset$};
\draw (3, 4.5) node{$\emptyset$};
\draw (0,1) node{};
\end{tikzpicture}
\raisebox{0.3cm}{$\mu_1 \geq 0$}
\hspace{2cm}
\begin{tikzpicture}[scale=0.35]
\draw (0,1) -- (0,4);
\draw (1,1) -- (1,4);
\draw (0,1) -- (1,1);
\draw (0,2) -- (1,2);
\draw (0,3) -- (1,3);
\draw (0,4) -- (1,4);
\draw (0.5,1.5) node{$\mu_1$};
\draw (0.5,3.5) node{1};
\draw (0.5,2.78) node{$\vdots$};
\draw (1.5,4.5) node{$\emptyset$};
\draw (2.5,4) -- (2.5,5);
\draw (2.5,5) -- (3.5,5);
\draw (3.5,5) -- (3.5,4);
\draw (2.5,4) -- (3.5,4);
\draw (3, 4.5) node{1};
\draw (0,1);
\end{tikzpicture}
\raisebox{0.3cm}{$\mu_1 > 0$}
\hspace{2cm}
\begin{tikzpicture}[scale=0.35]
\draw (0,4) -- (0,5);
\draw (0,5) -- (2,5);
\draw (1,5) -- (1,4);
\draw (0,4) -- (2,4);
\draw (2,5) -- (2,4);
\draw (0.5, 4.5) node{1};
\draw (1.5, 4.5) node{1};
\draw (2.5,4) -- (2.5,5);
\draw (2.5,5) -- (3.5,5);
\draw (3.5,5) -- (3.5,4);
\draw (2.5,4) -- (3.5,4);
\draw (3, 4.5) node{1};
\draw (0,1);
\end{tikzpicture}
\raisebox{0.3cm}{$\mu_1 = 0$}
\caption{Highest weight elements of $T^\mathfrak{d}(\mu)$}
\label{figCrystalHighWeights}
\end{figure}
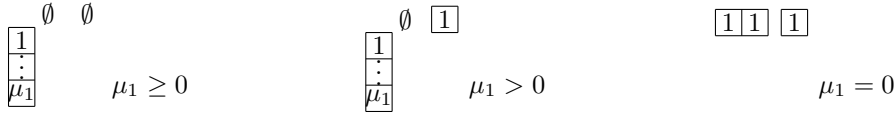

The crystal graph $T^{\mathfrak{d}}(\mu)$ is defined in Kwon~\cite[Eq. (4.7)]{2015arXiv151201877K} and~\cite[Def. 3.7]{MR3482440} as the set of (for $k=1$) pairs of a two-column skew semistandard Young tableaux where the tail has length $\mu_1$ and a column strict single column with certain dependencies characterized by a list of inequalities. This list was originally described by means of crystal graphs in the way it is stated in Lemma~\ref{LRDefLem}.

\begin{definition}
The set of \emph{orthogonal Littlewood-Richardson tableaux} is \[\LRtabs =\set{T\in T^{\mathfrak{d}}(\mu)}{\text{ $i$ occurs in $T$ exactly $\lambda'_i$ times and } \varepsilon_i(T)=0 \text{ for } i \neq 0}\]
with $\ell(\lambda)\leq n=2k+1$ and $\ell(\mu)\leq k$. 
\end{definition}

For an example see Figure~\ref{figCrystal2}.

We give another equivalent definition in the next section that we use for our algorithms.

As announced before, the set of orthogonal Littlewood-Richardson tableaux is counted by $\LRcoef$. See~\cite[Theorem 5.3]{2015arXiv151201877K}. This is one of the main results of~\cite{2015arXiv151201877K}.

\begin{theorem}[Kwon]
 $\LRcoef=|\LRtabs|$
\end{theorem}

Although Kwon considers $\mathrm{O}(n)$, for odd $n$ ($n=2k+1$) we get $\SO(n)$ as a special case. In this case $V(\lambda)\downarrow^{\mathrm{O}(2k+1)}_{\SO(2k+1)}$ is an irreducible $\SO(2k+1)$ representation and every such representation is isomorphic to such a restriction (see for example Okada~\cite[Sect. 2.4]{MR3604801}).

\section{The Bijection for $\SO(3)$}
\subsection{Orthogonal Littlewood-Richardson Tableaux reformulated}
\label{LRreformulated}
For our bijection we need an explicit description of Kwon's orthogonal Littlewood-Richardson tableaux for $\SO(3)$, provided by the following Theorem.

\begin{theorem}
\label{LemmaLR}
Let $\lambda$ and $\mu$ be partitions such that $\lambda\vdash r$, $ \ell(\lambda)\leq 3$ and $\ell(\mu)\leq 1$.
Then the orthogonal Littlewood-Richardson tableaux in $\LRtabs$ for $\SO(3)$ are exactly the following objects.
 (Along the $i^{\text{th}}$ line from right to left, diagrams are filled with $\{1,2,\dots, \lambda_i\}$. In particular, the number of entries $i$ occurring in this pair of tableaux is $\lambda_i'$.)\\

There are 4 cases:\\

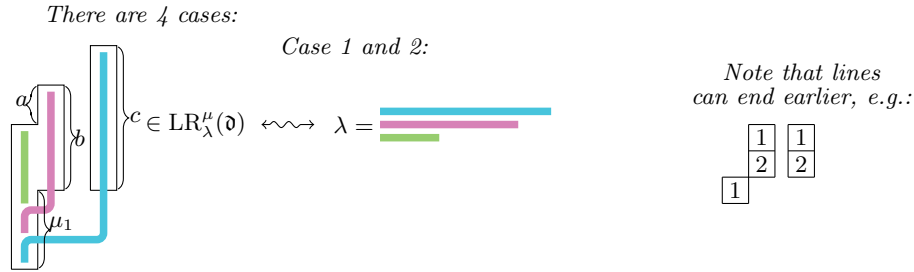
\begin{itemize}
\item[$-$] Case 1: $a$, $b$ and $c$ are even, $b\leq c$, the residuum is at most $1$.
\item[$-$] Case 2: $a$ and $b$ are even, $c$ is odd, $b<c$, $\mu\neq\emptyset$, $\lambda_1<c+\mu_1$, the residuum is at most $1$.
\end{itemize}
%
\\
$a$ and $b$ are even, $c$ is odd, the residuum is $1$, $\lambda_2=b-1  \leq c$, $\mu\neq\emptyset$, $a>0$. The rightmost line touches the middle column in exactly one cell and starts in the leftmost column two rows below.\\
%
\\
$a$ is even, $b$ and $c$ are odd, $\mu=\emptyset$, the residuum is $0$.

We can regard orthogonal Littlewood-Richardson tableaux as skew semistandard Young tableaux, whose shape is as indicated.
\end{theorem}

\begin{figure}
%
\vspace*{-0.3cm}
\caption{Some parts of crystal graphs. Orthogonal Littlewood-Richardson tableaux are emphasized.}
\label{figCrystal2}
\end{figure}

\begin{corollary}
The residuum is $1$ if and only if the leftmost line has one position in the tail and the middle column is not empty. In this case $a>0$. The leftmost line cannot have more than one position in the tail as this would give a residuum higher than $1$.
\end{corollary}

\begin{example}
For $\lambda=(8,5,3)$ and thus $r=16$ we have the following orthogonal Littlewood-Richardson tableaux:

\begin{tabular}{llllll}
$\mu=(1)$&$\mu=(2)$&$\mu=(3)$&$\mu=(3)$&$\mu=(4)$&$\mu=(5)$\\
\end{tabular}
\end{example}

In Table~\ref{TabListExample} in the appendix we provide a list of all pairs of $\lambda$ and $\mu$ and their associated orthogonal Littlewood-Richardson tableaux for $r\leq 3$.

For the proof of Theorem~\ref{LemmaLR} we need another observation in order to concentrate on the right paths in a crystal graph.

\begin{proposition}
Every element in  $T^\mathfrak{d}(\mu)$ can be obtained by applying a sequence of operators $f_{i_j} \circ f_{i_{j-1}}\circ \cdots \circ f_{i_2}\circ f_{i_1}$ to one of the top elements $T$ such that $i_\ell=0$ if and only if $f_{i_\ell}\circ \dots \circ f_{i_1}(T)$ is an orthogonal Littlewood-Richardson tableau.

We call the corresponding path in a crystal graph a \emph{good path}. The sequence of orthogonal Littlewood-Richardson tableaux on a good path is uniquely determined.
\end{proposition}

\begin{proof} This follows from Proposition~\ref{UniqMax}:
as $T^\mathfrak{d}(\mu)$ is a crystal graph, every component has a unique maximal element (its highest weight element). Each component of the graph obtained from $T^\mathfrak{d}(\mu)$ by erasing the edges labeled with $0$ is also a crystal graph and therefore has also a unique maximal element.
\end{proof}

The following lemma describes orthogonal Littlewood-Richardson tableaux in $T^\mathfrak{d}(\mu)$ and proves Theorem~\ref{LemmaLR}.

\begin{lemma}
\label{lemToProofTheorem}
Let $L$ be an orthogonal Littlewood-Richardson tableau in $T^\mathfrak{d}(\mu)$. Define a \emph{gap} to be an entry $j>1$ in $L$ such that $j-1$ is not in the same column.
\begin{enumerate}
\item \label{lemPoint1Gap}
$L$ has no gap in the rightmost column.

\item \label{lemPoint2Shape} $L$ is a skew semistandard Young tableau of the outer shape indicated in Theorem~\ref{LemmaLR}. Moreover the parity constraints are as indicated.

\item \label{lemPoint3Middle1} A gap in the middle column $b_2$ of $L$ can only occur in Case 3 and can only be the largest element of $b_2$. Moreover in Case 3 the largest element of $b_2$ is $c+1$.

\item \label{lemPoint4Middle2} In Case 3 the tail has entries $b-a+1$, $c+2$, $c+3$, $\dots$, $c+\mu_1$ with $a\geq 2$. In particular, the residuum is $1$.

\item \label{lemPoint5Gaps12} In Cases 1 and 2 gaps can only occur in the tail. If $j$ is a gap, then $j-1$ is the largest entry in a column to the right. 

\item \label{lemPoint6Uneq2} In Case 2, the smallest entry of the tail is at most $c$.
\end{enumerate}
\end{lemma}

\begin{proof}
Before proving the lemma we recall that:
\begin{itemize}
\item $\LRtabs$ is the set of tableaux $L$ in $T^\mathfrak{d}(\mu)$ such that $i$ occurs in $L$ exactly $\lambda'_i$ times and $\varepsilon_i(L)=0$ for $ i \neq 0 $.
\item The length of the tail in $T^\mathfrak{d}(\mu)$ is always $\mu_1$, thus constant.
\end{itemize}
\begin{enumerate}
\item Suppose that the rightmost column contains a gap $j$. Then the sequence in Lemma~\ref{TensorLemma} ends with a \enquote{$-$} because $\varepsilon_{j-1}(b_3)>0$ and thus $\varepsilon_{j-1}(b_1\otimes b_2 \otimes b_3)=\varepsilon_{j-1}(L)>0$.

\item  We use induction on the number of Littlewood-Richardson tableaux on a good path. Let us consider the highest weight elements of $T^\mathfrak{d}(\mu)$

\hspace{1.8cm}\begin{tikzpicture}[scale=0.35]
\draw (-2.5,3.5) node{$L_1:$};
\draw (-2,1.5) node{($\mu_1\geq 0$)};
\draw (0,1) -- (0,4);
\draw (1,1) -- (1,4);
\draw (0,1) -- (1,1);
\draw (0,2) -- (1,2);
\draw (0,3) -- (1,3);
\draw (0,4) -- (1,4);
\draw (0.5,1.5) node{$\mu_1$};
\draw (0.5,3.5) node{1};
\draw (0.5,2.78) node{$\vdots$};
\draw (3,5) node{};
\end{tikzpicture}\raisebox{0.35cm},
\hspace{0.7cm}
\begin{tikzpicture}[scale=0.35]
\draw (-2.5,3.5) node{$L_{23}:$};
\draw (-2,1.5) node{($\mu_1 > 0$)};
\draw (0,1) -- (0,4);
\draw (1,1) -- (1,4);
\draw (0,1) -- (1,1);
\draw (0,2) -- (1,2);
\draw (0,3) -- (1,3);
\draw (0,4) -- (1,4);
\draw (0.5,1.5) node{$\mu_1$};
\draw (0.5,3.5) node{1};
\draw (0.5,2.78) node{$\vdots$};
\draw (2,4) -- (2,5);
\draw (2,5) -- (3,5);
\draw (3,5) -- (3,4);
\draw (2,4) -- (3,4);
\draw (2.5, 4.5) node{1};
\end{tikzpicture}\raisebox{0.35cm},
\hspace{0.7cm}
\begin{tikzpicture}[scale=0.35]
\draw (-2.5,3.5) node{$L_4:$};
\draw (-2,1.5) node{($\mu_1 = 0$)};
\draw (0,4) -- (0,5);
\draw (0,5) -- (2,5);
\draw (1,5) -- (1,4);
\draw (0,4) -- (2,4);
\draw (2,5) -- (2,4);
\draw (0.5, 4.5) node{1};
\draw (1.5, 4.5) node{1};
\draw (2,5) -- (3,5);
\draw (3,5) -- (3,4);
\draw (2,4) -- (3,4);
\draw (2.5, 4.5) node{1};
\draw (0,1);
\end{tikzpicture}

which belong to Case 1, Case 2 or 3 and Case 4 respectively. As illustrated for $T^\mathfrak{d}(1)$ in Figure~\ref{figCrystal2}, the tableaux belonging to Case 2 and Case 3 occur in the same component of the crystal graph. Moreover, following any good path beginning in $L_{23}$ the next orthogonal Littlewood-Richardson tableaux encountered is one of the following four:

\hspace{1.2cm}\begin{tikzpicture}[scale=0.35]
\draw (0,1) -- (0,4);
\draw (1,1) -- (1,4);
\draw (0,1) -- (1,1);
\draw (0,2) -- (1,2);
\draw (0,3) -- (1,3);
\draw (0,4) -- (1,4);
\draw (0.5,1.5) node{$\mu_1$};
\draw (0.5,3.5) node{1};
\draw (0.5,2.78) node{$\vdots$};
\draw (2,4) -- (2,7);
\draw (3,4) -- (3,7);
\draw (2,4) -- (3,4);
\draw (2,5) -- (3,5);
\draw (2,6) -- (3,6);
\draw (2,7) -- (3,7);
\draw (2.5, 6.5) node{1};
\draw (2.5, 5.5) node{2};
\draw (2.5, 4.5) node{3};
\draw (0,-1);
\draw (3,1.5) node{($\mu_1\geq 1$)};
\end{tikzpicture}\raisebox{1cm},
\hspace{0.7cm}
\begin{tikzpicture}[scale=0.35]
\draw (-0.4,0) -- (-0.4,4);
\draw (1.4,0) -- (1.4,4);
\draw (-0.4,0) -- (1.4,0);
\draw (-0.4,1) -- (1.4,1);
\draw (-0.4,2) -- (1.4,2);
\draw (-0.4,3) -- (1.4,3);
\draw (-0.4,4) -- (1.4,4);
\draw (0.5,0.5) node{$\scriptstyle \mu_1+2$};
\draw (0.5,3.5) node{1};
\draw (0.5,2.5) node{4};
\draw (0.5,1.78) node{$\vdots$};
\draw (2.4,4) -- (2.4,7);
\draw (3.4,4) -- (3.4,7);
\draw (2.4,4) -- (3.4,4);
\draw (2.4,5) -- (3.4,5);
\draw (2.4,6) -- (3.4,6);
\draw (2.4,7) -- (3.4,7);
\draw (2.9, 6.5) node{1};
\draw (2.9, 5.5) node{2};
\draw (2.9, 4.5) node{3};
\draw (0,-1);
\draw (3.5,0.5) node{($\mu_1\geq 2$)};
\end{tikzpicture}\raisebox{1cm},
\hspace{0.7cm}
\begin{tikzpicture}[scale=0.35]
\draw (-0.4,-1) -- (-0.4,4);
\draw (1.4,-1) -- (1.4,4);
\draw (-0.4,-1) -- (1.4,-1);
\draw (-0.4,0) -- (1.4,0);
\draw (-0.4,1) -- (1.4,1);
\draw (-0.4,2) -- (1.4,2);
\draw (-0.4,3) -- (1.4,3);
\draw (-0.4,4) -- (1.4,4);
\draw (0.5,-0.5) node{$\scriptstyle \mu_1+1$};
\draw (0.5,3.5) node{1};
\draw (0.5,2.5) node{2};
\draw (0.5,1.5) node{4};
\draw (0.5,0.78) node{$\vdots$};
\draw (2.4,4) -- (2.4,7);
\draw (3.4,4) -- (3.4,7);
\draw (2.4,4) -- (3.4,4);
\draw (2.4,5) -- (3.4,5);
\draw (2.4,6) -- (3.4,6);
\draw (2.4,7) -- (3.4,7);
\draw (2.9, 6.5) node{1};
\draw (2.9, 5.5) node{2};
\draw (2.9, 4.5) node{3};
\draw (3.5,-0.5) node{($\mu_1\geq 3$)};
\end{tikzpicture}
\hspace{0.1cm} \raisebox{1cm}{and} \hspace{0.2cm}
\begin{tikzpicture}[scale=0.35]
\draw (-0.4,0) -- (-0.4,4);
\draw (1.4,0) -- (1.4,4);
\draw (-0.4,0) -- (1.4,0);
\draw (-0.4,1) -- (1.4,1);
\draw (-0.4,2) -- (1.4,2);
\draw (-0.4,3) -- (1.4,3);
\draw (-0.4,4) -- (1.4,4);
\draw (0.5,0.5) node{$\scriptstyle \mu_1+1$};
\draw (0.5,3.5) node{1};
\draw (0.5,2.5) node{3};
\draw (0.5,1.78) node{$\vdots$};
\draw (1.4,4) -- (1.4,6);
\draw (2.4,4) -- (2.4,6);
\draw (3.4,5) -- (3.4,6);
\draw (1.4,4) -- (2.4,4);
\draw (1.4,5) -- (3.4,5);
\draw (1.4,6) -- (3.4,6);
\draw (1.9, 4.5) node{2};
\draw (1.9, 5.5) node{1};
\draw (2.9, 5.5) node{1};
\draw (0,-1);
\draw (3.5,0.5) node{($\mu_1\geq 1$)};
\end{tikzpicture}

The first three of these belong to Case 2, the last belongs to Case 3. Together with $L_1$, $L_{23}$ and $L_4$, these tableaux constitute the base case of the induction. Clearly, they are all skew semistandard of the required shape.

In the inductive step we show additionally that the Littlewood-Richardson tableaux on a good path all belong to the same case, with the exception of $L_{23}$.

For $i>0$ the generators $f_i$ map skew semistandard Young tableaux to skew semistandard Young tableaux of the same shape (see~\cite{ONCRYSTALBASES}). Thus it remains to consider $f_0$, which adds a domino to one of the three columns and yields an orthogonal Littlewood-Richardson tableau.

If the domino is added on top of the rightmost column there is nothing to check.

Thus, suppose now that the domino is added to the middle column. We have to show that the result is a skew semistandard Young tableau of the same outer shape as the preimage. Item~\ref{lemPoint1Gap} shows that there are no gaps in the rightmost column. If the middle column is non-empty, its smallest entry has a right neighbor by induction, and is at least $3$. Otherwise, the rightmost column contains at least two entries, because if it were empty, the domino could not be added to the middle column by Lemma~\ref{TensorLemma}. In both cases $f_0$ preserves the outer shape and semistandardness. Moreover the resulting tableau belongs to the same case as the preceding orthogonal Littlewood-Richardson tableau on the good path.

If the domino is added to the leftmost column, there is nothing to check because the first two columns are skew semistandard by definition of $B^{\mathfrak{d}}(\mu)$.

Finally, in Case 2 we have $b<c$ because $c$ is odd.

\item If $j$ is a gap in the middle column, $\varepsilon_{j-1}(b_2)>0$. Because of $\varepsilon_{j-1}(L)=\varepsilon_{j-1}(b_1\otimes b_2 \otimes b_3)=0$, we have $\varphi_{j-1}(b_3)>0$. By item~\ref{lemPoint1Gap} there cannot be a gap in $b_3$, so $j-1$ must be $b_3$'s largest entry. In particular $j=c+1$. Because of semistandardness, this implies that the cell containing the gap of $b_2$ has no right neighbor. The only shape compatible with this requirement is that of Case 3.

Suppose now that $b_2$ contains no gap. By induction, $b_2$ contains entries $1$ and $2$. Therefore, $\varepsilon_0(b_2)=1$, and Lemma~\ref{TensorLemma} implies that the domino was not added to the rightmost column to obtain $L$. Let $L'$ be the orthogonal Littlewood-Richardson tableau preceding $L$ on the good path. Then the rightmost columns of $L$ and $L'$ are the same. By induction, the largest entry in the middle column of $L'$ is $c+1$. Because the generators $f_i$ only increase entries, the largest entry of the middle column of $L$, is at least $c+1$. The claim follows, since the length of the middle column is at most $c+1$, and it contains no gaps.

\item Let $t$ be the smallest entry of the tail and $s=b-1$ the second largest element of the middle column. We show $t \leq s$ for all tableaux on the good path. Suppose that $t=s$, and consider the action of $f_t$. If there is a gap in the middle column, $\varphi_t(b_2)>0$ and the smallest entry in the tail is not modified. Otherwise, the rightmost column also contains $t=c$, so $\varphi_t(b_3)>0$ and $\varepsilon_t(b_2)=0$, so the entry in the tail is not modified either.

Because the residuum is at most $1$, the second entry of the tail, if present, is at least $c+2$. Because $c+2-(b-a-1)\geq 2$ it is a gap. However, an entry $j>c+2$ cannot be a gap, because $\varepsilon_{j-1}(L)=0$ but $\varphi_{j-1}(b_2)=\varphi_{j-1}(b_3)=0$ for $j>c+2$.

\item Suppose that $j$ is a gap in $b_1$. Because $\varepsilon_{j-1}(L)=0$, either $\varphi_{j-1}(b_2)>0$ or $\varphi_{j-1}(b_3)>0$. Since there are no gaps in $b_2$ and $b_3$, in both cases $j-1$ is the largest entry in its column. Semistandardness implies that $j$ must be in the tail.

\item Let $t=b-a+1$ be the smallest entry of the tail. We show that $t\leq c$ for all tableaux on the good path. Suppose that $t=c$ and consider the action of $f_t$. Since there is no gap in the middle column, $\varepsilon_t(b_2)=0$, so $f_t$ modifies the largest element of $b_3$. Note that this argument does not hold for Case 1 due to the base case $L_1$.\qedhere
\end{enumerate}
\end{proof}

\subsection{A descent-preserving bijection}
\label{subsectionBijection}

In this section we describe our bijection for $\SO(3)$. It maps pairs consisting of a standard Young tableaux with at most 3 rows and an orthogonal Littlewood-Richardson tableau for $\SO(3)$ to vacillating tableaux with at most one row. Moreover, it preserves descents.

The bijection is composed of two parts. First we use the information of the orthogonal Littlewood-Richardson tableau to obtain a bigger standard Young tableau, all of whose row lengths have the same parity. In the second step we map this new standard Young tableau to a vacillating tableau of shape $\emptyset$. To obtain a vacillating tableau of the desired shape, we delete positions at the end of it such that the number of deleted positions corresponds to the number of added positions in the first part.

We formulate now our bijection using Algorithms~\ref{algoLR3} and~\ref{algo3dim}, give a complete example as well as formulate the reversed bijection using the reversed Algorithms~\ref{UalgoLR3} and~\ref{Ualgo3dim}. In Table~\ref{TabNotation} we present notation used in Algorithms~\ref{algo3dim} and~\ref{Ualgo3dim}.

In Section~\ref{Sub1stAlgo} we examine the details of the algorithms describing the first part of the bijection. In Section~\ref{Sub2ndAlgo} we give details for the second part.

\subsubsection{Formulation of the bijection}
\label{bijection}

Starting with $(Q,L)$ where $Q$ is a standard Young tableau in $\SYT(\lambda)$ and $L$ an orthogonal Littlewood-Richardson tableau in $\LRtabs$ we use Algorithm~\ref{algoLR3} which adds $\mu_1$ cells to $Q$ to obtain a standard Young tableau $\tilde{Q}$ whose row lengths all have the same parity.

Now we distinguish two cases. If our resulting tableau $\tilde{Q}$ consists of rows of even length this is the tableau we use in Algorithm~\ref{algo3dim}. Otherwise, thus when $\tilde{Q}$ consists of three rows of odd length, we look for the largest entry $e$ in $\tilde{Q}$ and attach $e+1$ to the first row, $e+2$ to the second one and $e+3$ to the third one. We obtain a standard Young tableau with rows of even length that we use in in Algorithm~\ref{algo3dim}.

We continue using Algorithm~\ref{algo3dim} to obtain a vacillating tableau $\tilde{V}$ of shape $\emptyset$ from our modified standard Young tableau $\tilde{Q}$.

Once again we distinguish the two cases from before. If we added entries to $\tilde{Q}$ to obtain a tableau with rows of even length from one with rows of odd length, we delete now the last three positions of $\tilde{V}$. Those are always $1,0,-1$, (see Lemma~\ref{Lem10-1}). Otherwise we do not change $\tilde{V}$.
Therefore we obtain once again a vacillating tableau $\tilde{V}$ of shape $\emptyset$.

We finish our algorithm by deleting the last $\mu_1$ entries of $\tilde{V}$ to obtain a vacillating tableau $V$ of shape $\mu$ and length $r=|\lambda|$. Those are always $-1$ steps (see Lemma~\ref{LemHorizonal1} and Lemma~\ref{LemHorizonal2}).

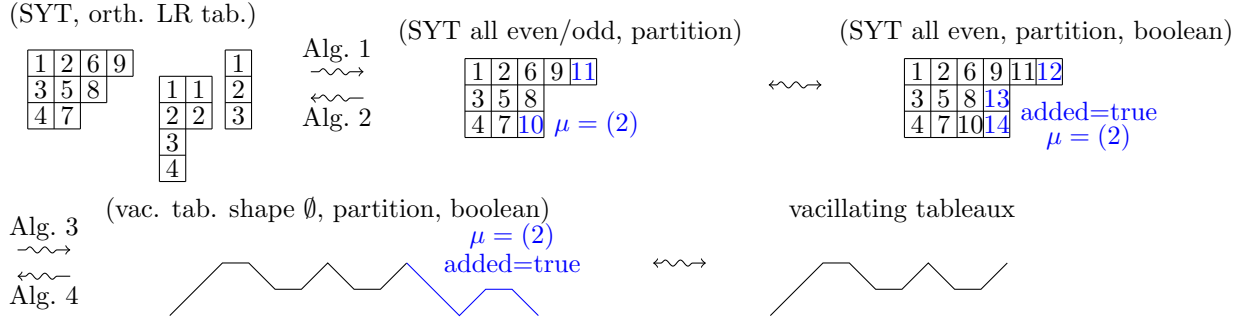
\begin{figure}
\begin{tikzpicture}[scale=0.35]
  \draw (0,4) -- (4,4);
  \draw (0,3) -- (4,3);
  \draw (0,2) -- (3,2);
  \draw (0,1) -- (2,1);
  \draw (0,4) -- (0,1);
  \draw (1,4) -- (1,1);
  \draw (2,4) -- (2,1);
  \draw (3,4) -- (3,2);
  \draw (4,4) -- (4,3);
  \draw (0.5,3.5) node {1};
  \draw (1.5,3.5) node {2};
  \draw (0.5,2.5) node {3};
  \draw (0.5,1.5) node {4};
  \draw (1.5,2.5) node {5};
  \draw (2.5,3.5) node {6};
  \draw (1.5,1.5) node {7};
  \draw (2.5,2.5) node {8};
  \draw (3.5,3.5) node {9};
  \draw (5,3) -- (7,3);
  \draw (5,2) -- (7,2);
  \draw (5,1) -- (7,1);
  \draw (5,0) -- (6,0);
  \draw (5,-1) -- (6,-1);
  \draw (5,3) -- (5,-1);
  \draw (6,3) -- (6,-1);
  \draw (7,3) -- (7,1);
  \draw (5.5,2.5) node {1};
  \draw (6.5,2.5) node {1};
  \draw (5.5,1.5) node {2};
  \draw (6.5,1.5) node {2};
  \draw (5.5,0.5) node {3};
  \draw (5.5,-0.5) node {4};
  \draw (7.5,4) -- (8.5,4);
  \draw (7.5,3) -- (8.5,3);
  \draw (7.5,2) -- (8.5,2);
  \draw (7.5,1) -- (8.5,1);
  \draw (7.5,4) -- (7.5,1);
  \draw (8.5,4) -- (8.5,1);
  \draw (8,3.5) node {1};
  \draw (8,2.5) node {2};
  \draw (8,1.5) node {3};
  \draw (0,-0.5) node {};
  \draw (4,5.3) node {(SYT, orth. LR tab.)};
\end{tikzpicture}
\hspace{0.15cm}
\begin{tikzpicture}[scale=0.35]
 \draw (0,-0.5) node{};
 \draw (0,5) node{};
 \draw [->,decorate,
decoration={snake,amplitude=.4mm,segment length=2mm,post length=1mm, pre length=1mm}] (0,3.5) -- (2,3.5) node[midway, above] {Alg.~\ref{algoLR3}};
 \draw [<-,decorate,
decoration={snake,amplitude=.4mm,segment length=2mm,post length=1mm, pre length=1mm}] (0,2.5) -- (2,2.5)node[midway, below] {Alg.~\ref{UalgoLR3}};
\end{tikzpicture}
\begin{tikzpicture}[scale=0.35]
  \draw (0,4) -- (5,4);
  \draw (0,3) -- (5,3);
  \draw (0,2) -- (3,2);
  \draw (0,1) -- (3,1);
  \draw (0,4) -- (0,1);
  \draw (1,4) -- (1,1);
  \draw (2,4) -- (2,1);
  \draw (3,4) -- (3,1);
  \draw (4,4) -- (4,3);
  \draw (5,4) -- (5,3);
  \draw (0.5,3.5) node {1};
  \draw (1.5,3.5) node {2};
  \draw (0.5,2.5) node {3};
  \draw (0.5,1.5) node {4};
  \draw (1.5,2.5) node {5};
  \draw (2.5,3.5) node {6};
  \draw (1.5,1.5) node {7};
  \draw (2.5,2.5) node {8};
  \draw (3.5,3.5) node {9};
  \draw (2.5,1.5) node[color=blue] {10};
  \draw (4.5,3.5) node[color=blue] {11};
  \draw (5,1.5) node[color=blue] {$\mu=(2)$};
  \draw (0,-0.5) node {};
  \draw (4,5) node{(SYT all even/odd, partition)};
\end{tikzpicture}
\begin{tikzpicture}[scale=0.35]
 \draw (0,-0.5) node{};
 \draw (0,5) node{};
 \draw [<->,decorate,
decoration={snake,amplitude=.4mm,segment length=2mm,post length=1mm, pre length=1mm}] (0,3) -- (2,3);
\end{tikzpicture}
\begin{tikzpicture}[scale=0.35]
  \draw (9,4) -- (15,4);
  \draw (9,3) -- (15,3);
  \draw (9,2) -- (13,2);
  \draw (9,1) -- (13,1);
  \draw (9,4) -- (9,1);
  \draw (10,4) -- (10,1);
  \draw (11,4) -- (11,1);
  \draw (12,4) -- (12,1);
  \draw (13,4) -- (13,1);
  \draw (14,4) -- (14,3);
  \draw (15,4) -- (15,3);
  \draw (9.5,3.5) node {1};
  \draw (10.5,3.5) node {2};
  \draw (9.5,2.5) node {3};
  \draw (9.5,1.5) node {4};
  \draw (10.5,2.5) node {5};
  \draw (11.5,3.5) node {6};
  \draw (10.5,1.5) node {7};
  \draw (11.5,2.5) node {8};
  \draw (12.5,3.5) node {9};
  \draw (11.5,1.5) node {10};
  \draw (13.5,3.5) node {11};
  \draw (14.5,3.5) node[color=blue] {12};
  \draw (12.5,2.5) node[color=blue] {13};
  \draw (12.5,1.5) node[color=blue] {14};
  \draw (11,-0.5) node{};
  \draw (16,1) node[color=blue] {$\mu=(2)$};
  \draw (16,2) node[color=blue] {added=true};
  \draw (14,5) node{(SYT all even, partition, boolean)};
\end{tikzpicture}

\begin{tikzpicture}[scale=0.35]
 \draw (0,-0.5) node{};
 \draw (0,5) node{};
 \draw [->,decorate,
decoration={snake,amplitude=.4mm,segment length=2mm,post length=1mm, pre length=1mm}] (0,3.5) -- (2,3.5) node[midway, above] {Alg.~\ref{algo3dim}};
 \draw [<-,decorate,
decoration={snake,amplitude=.4mm,segment length=2mm,post length=1mm, pre length=1mm}] (0,2.5) -- (2,2.5)node[midway, below] {Alg.~\ref{Ualgo3dim}};
\end{tikzpicture}
\begin{tikzpicture}[scale=0.35]
\draw[white](0,0)--(0,3);
\draw (0,0)  -- (1,1)--(2,2)--(3,2)--(4,1)--(5,1)--(6,2)--(7,1)--(8,1)--(9,2);
\draw[blue](9,2)--(10,1)--(11,0)--(12,1)--(13,1)--(14,0);
\draw (6,4) node {(vac. tab. shape $\emptyset$, partition, boolean)};
\draw (0,-1.5) node{};
  \draw (13,3) node[color=blue] {$\mu=(2)$};
    \draw (13,2) node[color=blue] {added=true};
\end{tikzpicture}
\hspace{0.5cm}
\begin{tikzpicture}[scale=0.35]
 \draw (0,-0.5) node{};
 \draw (0,5) node{};
 \draw [<->,decorate,
decoration={snake,amplitude=.4mm,segment length=2mm,post length=1mm, pre length=1mm}] (0,3) -- (2,3);
\end{tikzpicture}
\hspace{0.5cm}
\begin{tikzpicture}[scale=0.35]
\draw[white](0,0)--(0,3);
\draw (0,0)  -- (1,1)--(2,2)--(3,2)--(4,1)--(5,1)--(6,2)--(7,1)--(8,1)--(9,2);
\draw (5,4) node {vacillating tableaux};
\draw (0,-1.5) node{};
\end{tikzpicture}
\vspace{-0.5cm}
\caption{The strategy of our bijection outlined using our running example}
\label{FigStrategy}
\end{figure}

In Figure~\ref{FigStrategy} we illustrate the strategy of our bijection using our running example.

In Table~\ref{TabListExample} in the appendix we provide a list of all pairs $(Q,L)$, their associated $\tilde{Q}$ and their associated $V$ for $r\leq 3$.

As Algorithm~\ref{algoLR3} and~\ref{UalgoLR3} are inverse and Algorithm~\ref{algo3dim} and~\ref{Ualgo3dim} are inverse, the reversed bijection is easily defined.

We start with a vacillating tableau $V$ of shape $\mu=(\mu_1)$ and length $r$, add $\mu_1$ $(-1)$'s to obtain a vacillating tableau $V$ of shape $\emptyset$. If this is of odd length we add $1,0,-1$. Next we apply Algorithm~\ref{Ualgo3dim} to obtain a standard Young tableau $\tilde{Q}$. If we added ${1,0,-1}$ we delete the biggest 3 entries. Those are at the end of each row in increasing order (see Lemma~\ref{Lem10-1}). Then we apply Algorithm~\ref{UalgoLR3} to obtain an orthogonal Littlewood-Richardson tableau in $\LRtabs$ and a standard Young tableau in $\SYT(\lambda)$ where $\lambda$ is a partition such that $\lambda\vdash r$ and $\mu\leq \lambda$.

\subsection{The first Bijection}
\label{Sub1stAlgo}

In this section we consider the first bijection consisting of Algorithm~\ref{algoLR3} and its inverse, Algorithm~\ref{UalgoLR3}. The former maps a pair consisting of an orthogonal Littlewood-Richardson tableau in $\LRtabs$ and a standard Young tableau in $\SYT(\lambda)$ to a bigger standard Young tableau having $\mu_1$ additional cells that form a horizontal strip with at most one cell in the first row, whose row lengths have all the same parity.
We start by giving two examples, the former is of Case 3, the latter, our running example, is of Case 2.

\begin{algorithm}[p]
\label{algoLR3}
\SetKwInOut{Input}{input}\SetKwInOut{Output}{output}
 \Input{orthogonal Littlewood-Richardson tableau $L\in \LRtabs$, standard Young tableau $Q\in \SYT(\lambda)$}
 \Output{standard Young tableau $\tilde{Q}$ with 3 (possibly empty) rows, all of even or all of odd length, contains the original standard Young tableau and $\mu_1$ additional positions that form a horizontal strip with at most one element in the first row, that is filled increasingly from left to right\\ the partition $\mu$, unmodified}
rotate $L$ by $\frac{\pi}{2}$ and shift cells such that it becomes the row-strict tableau $\tilde{L}$, row $i$ being filled with $1,2,\dots,\lambda_i$\;
add cell labeled $x$ to $\tilde{L}$ in row $i+1$ for each position in row $i$ that was not in column $i$ from the right\;
add cells labeled $x$ to $\tilde{L}$ in row $3$ such that the number of cells labeled $x$ is $\mu_1$\;
\If(\tcc*[f]{adjust parity}){not all rows of $\tilde{L}$ have the same parity}
{
 \If{total number of cells is odd}
 {move last cell of the bottommost even row to the other even row\;}
  \If{total number of cells is even}
 {move last cell of the bottommost odd row to the other odd row\;}
}
add numbers $|\lambda|+1, |\lambda|+2, \dots$ increasingly from left to right to $Q$ at the positions of the cells in $\tilde{L}$ labeled $x$ to obtain $\tilde{Q}$\;
\Return ($\tilde{Q}$, $\mu$)\;
 \caption{(orth. LR tab. + SYT) to (SYT + suitable partition)}
\end{algorithm}

\begin{algorithm}[p]
\label{UalgoLR3}
\SetKwInOut{Input}{input}\SetKwInOut{Output}{output}
 \Input{standard Young tableau $\tilde{Q}\in\SYT(\tilde{\lambda})$ with 3 rows, all of even or all of odd length, a partition $\mu=(\mu_1)$, such that the skew-diagram containing the $\mu_1$ largest entries in $\tilde{Q}$ form a horizontal strip with at most one position in the first row, filled from left to right increasingly}
  \Output{orthogonal Littlewood-Richardson tableau $L\in\LRtabs$, standard Young tableau $Q\in\SYT(\lambda)$}
let $L$ be the row-strict tableau of the same shape as  $\tilde{Q}$, with row $i$ filled with $1,2,\dots,\tilde{\lambda_i}$\;
delete the cells in  $\tilde{Q}$ with the $\mu_1$ largest entries to obtain $Q$\;
replace the entries in the corresponding cells in $L$ with $x$'s\;
rotate $L$ by $-\frac{\pi}{2}$\;
\For(\tcc*[f]{mark elements for tail}){each column of $L$}{
\If{there is an odd number of cells labeled $x$ in this column}{mark the bottommost cell in this column that is not marked yet and contains a number, if there exists none, mark the bottommost unmarked cell in the column to the right\;
delete the bottommost cell labeled $x$\;}
\If{there are still cells labeled $x$ in this column}{
beginning at the bottom, mark as many cells, that are not marked yet and contain numbers, in the column to the right, as there are cells labeled $x$ in the current column\;
delete the cells labeled $x$ in the current column\;
}
  }
shift all cells that are marked to the leftmost column\;
separate the rightmost column from the other ones such that we obtain two tableaux\;
shift the leftmost column down such that the tail has length $\mu_1$\;
\If(\tcc*[f]{fix parity}){the second column has odd length and $\mu_1\neq0$}
{
find the smallest entry in the two-column tableau, that can change columns, such that the resulting tableau, after shifting columns to obtain a tail of length $\mu_1$, is an orthogonal Littlewood-Richardson tableau; move this entry to the other column in the two-column tableau and shift the columns such that the tail has length $\mu_1$.
}
\Return{($L$, $Q$)}
 \caption{(SYT + suitable partition) to (orth. LR tab. + SYT) }
\end{algorithm}

\begin{example}
In this example our orthogonal Littlewood-Richardson tableau is of Case 3 with $\mu=(2)$, $\lambda=(5,1,1)$ and $r=7$:\vspace*{-0.3cm}

\begin{tikzpicture}[scale=0.35]
  \draw (0,4) -- (5,4);
  \draw (0,3) -- (5,3);
  \draw (0,2) -- (1,2);
  \draw (0,1) -- (1,1);
  \draw (0,4) -- (0,1);
  \draw (1,4) -- (1,1);
  \draw (2,4) -- (2,3);
  \draw (3,4) -- (3,3);
  \draw (4,4) -- (4,3);
  \draw (5,4) -- (5,3);
  \draw (0.5,3.5) node {1};
  \draw (1.5,3.5) node {4};
  \draw (0.5,2.5) node {2};
  \draw (0.5,1.5) node {3};
  \draw (2.5,3.5) node {5};
  \draw (4.5,3.5) node {7};
  \draw (3.5,3.5) node {6};
  \draw (7,2) -- (8,2);
  \draw (7,1) -- (8,1);
  \draw (6,0) -- (8,0);
  \draw (6,-1) -- (7,-1);
  \draw (6,-2) -- (7,-2);
  \draw (6,0) -- (6,-2);
  \draw (7,2) -- (7,-2);
  \draw (8,2) -- (8,0);
  \draw (7.5,1.5) node {1};
      \draw[->] (7.5,2) -- (7.5,3);
    \draw (7.5,0.5) node[draw, circle, scale=1.2] {};
  \draw (7.5,0.5) node {4};
    \draw[->] (7.5,0) -- (8,-1);
  \draw (6.5,-0.5) node {1};
    \draw[->] (6.5,0) -- (6.5,1);
    \draw (6.5,-1.5) node[draw, circle, scale=1.2] {};
  \draw (6.5,-1.5) node {5};
    \draw[->] (7,-1.5) -- (7.5,-1.5);
  \draw (8.5,4) -- (9.5,4);
  \draw (8.5,3) -- (9.5,3);
  \draw (8.5,2) -- (9.5,2);
  \draw (8.5,1) -- (9.5,1);
  \draw (8.5,4) -- (8.5,1);
  \draw (9.5,4) -- (9.5,1);
  \draw (9,3.5) node {1};
  \draw (9,2.5) node {2};
  \draw (9,1.5) node {3};
  \draw (0,-2) node {};
\end{tikzpicture}
\begin{tikzpicture}[scale=0.35]
 \draw (0,-2) node{};
 \draw (0,5) node{};
 \draw [->,decorate,
decoration={snake,amplitude=.4mm,segment length=2mm,post length=1mm, pre length=1mm}] (0,3) -- (1,3);
\end{tikzpicture}
\begin{tikzpicture}[scale=0.35]
  \draw (0,4) -- (5,4);
  \draw (0,3) -- (5,3);
  \draw (0,2) -- (3,2);
  \draw (0,1) -- (1,1);
  \draw (0,4) -- (0,1);
  \draw (1,4) -- (1,1);
  \draw (2,4) -- (2,2);
  \draw (3,4) -- (3,2);
  \draw (4,4) -- (4,3);
  \draw (5,4) -- (5,3);
  \draw (0.5,3.5) node {1};
  \draw (1.5,3.5) node {4};
  \draw (0.5,2.5) node {2};
  \draw (0.5,1.5) node {3};
  \draw (2.5,3.5) node {5};
  \draw (4.5,3.5) node {7};
  \draw (3.5,3.5) node {6};
      \draw (1.5,2.5) node[draw, circle, scale=1.2] {};
  \draw (1.5,2.5) node {8};
      \draw (2.5,2.5) node[draw, circle, scale=1.2] {};
  \draw (2.5,2.5) node {9};
  \draw (6,4) -- (11,4);
  \draw (6,3) -- (11,3);
  \draw (6,4) -- (6,1);
  \draw (7,4) -- (7,1);
  \draw (8,4) -- (8,3);
  \draw (6.5,3.5) node {1};
  \draw (6.5,2.5) node {1};
  \draw (6,2) -- (9,2);
  \draw (6,1) -- (7,1);
  \draw (8,4) -- (8,2);
  \draw (9,4) -- (9,2);
  \draw (10,4) -- (10,3);
  \draw (11,4) -- (11,3);
  \draw (6.5,1.5) node {1};
  \draw (7.5,3.5) node {2};
  \draw (8.5,3.5) node {3};
      \draw (7.5,2.5) node[draw, circle, scale=1.2] {};
  \draw (7.5,2.5) node {x};
      \draw (8.5,2.5) node[draw, circle, scale=1.2] {};
  \draw (8.5,2.5) node {x};
  \draw (9.5,3.5) node {4};
  \draw (10.5,3.5) node {5};
        \draw (9.5,3.5) node[draw, circle, scale=1.2] {};
              \draw (10.5,3.5) node[draw, circle, scale=1.2] {};
  \draw (0,-2) node {};
\end{tikzpicture}
\begin{tikzpicture}[scale=0.35]
 \draw (0,-2) node{};
 \draw (0,5) node{};
 \draw [->,decorate,
decoration={snake,amplitude=.4mm,segment length=2mm,post length=1mm, pre length=1mm}] (0,3) -- (1,3);
\end{tikzpicture}
\begin{tikzpicture}[scale=0.35]
  \draw (0,4) -- (5,4);
  \draw (0,3) -- (5,3);
  \draw (0,2) -- (3,2);
  \draw (0,1) -- (1,1);
  \draw (0,4) -- (0,1);
  \draw (1,4) -- (1,1);
  \draw (2,4) -- (2,2);
  \draw (3,4) -- (3,2);
  \draw (4,4) -- (4,3);
  \draw (5,4) -- (5,3);
  \draw (0.5,3.5) node {1};
  \draw (1.5,3.5) node {4};
  \draw (0.5,2.5) node {2};
  \draw (0.5,1.5) node {3};
  \draw (2.5,3.5) node {5};
  \draw (4.5,3.5) node {7};
  \draw (3.5,3.5) node {6};
  \draw (1.5,2.5) node {8};
  \draw (2.5,2.5) node {9};
  \draw (0,-2) node {};
\end{tikzpicture}
\end{example}

\begin{example}
\label{ExRunStart}
In our running example we start with a standard Young tableau and an orthogonal Littlewood-Richardson tableau of Case 2 with  $\mu=(2)$, $\lambda=(4,3,2)$ and $r=9$:\vspace*{-0.3cm}\\
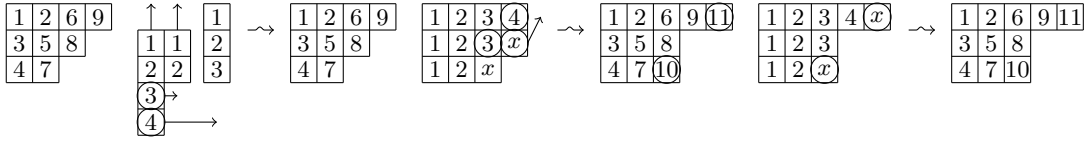
\begin{tikzpicture}[scale=0.35]
  \draw (0,4) -- (4,4);
  \draw (0,3) -- (4,3);
  \draw (0,2) -- (3,2);
  \draw (0,1) -- (2,1);
  \draw (0,4) -- (0,1);
  \draw (1,4) -- (1,1);
  \draw (2,4) -- (2,1);
  \draw (3,4) -- (3,2);
  \draw (4,4) -- (4,3);
  \draw (0.5,3.5) node {1};
  \draw (1.5,3.5) node {2};
  \draw (0.5,2.5) node {3};
  \draw (0.5,1.5) node {4};
  \draw (1.5,2.5) node {5};
  \draw (2.5,3.5) node {6};
  \draw (1.5,1.5) node {7};
  \draw (2.5,2.5) node {8};
  \draw (3.5,3.5) node {9};
  \draw (5,3) -- (7,3);
  \draw (5,2) -- (7,2);
  \draw (5,1) -- (7,1);
  \draw (5,0) -- (6,0);
  \draw (5,-1) -- (6,-1);
  \draw (5,3) -- (5,-1);
  \draw (6,3) -- (6,-1);
  \draw (7,3) -- (7,1);
  \draw (5.5,2.5) node {1};
        \draw[->] (5.5,3) -- (5.5,4);
  \draw (6.5,2.5) node {1};
        \draw[->] (6.5,3) -- (6.5,4);
  \draw (5.5,1.5) node {2};
  \draw (6.5,1.5) node {2};
  \draw (5.5,0.5) node[draw, circle, scale=1.2] {};
  \draw (5.5,0.5) node {3};
  \draw[->] (6,0.5) -- (6.5,0.5);
  \draw (5.5,-0.5) node[draw, circle, scale=1.2] {};
  \draw (5.5,-0.5) node {4};
  \draw[->] (6,-0.5) -- (8,-0.5);
  \draw (7.5,4) -- (8.5,4);
  \draw (7.5,3) -- (8.5,3);
  \draw (7.5,2) -- (8.5,2);
  \draw (7.5,1) -- (8.5,1);
  \draw (7.5,4) -- (7.5,1);
  \draw (8.5,4) -- (8.5,1);
  \draw (8,3.5) node {1};
  \draw (8,2.5) node {2};
  \draw (8,1.5) node {3};
  \draw (0,-1) node {};
\end{tikzpicture}
\begin{tikzpicture}[scale=0.35]
 \draw (0,-1) node{};
 \draw (0,5) node{};
 \draw [->,decorate,
decoration={snake,amplitude=.4mm,segment length=2mm,post length=1mm, pre length=1mm}] (0,3) -- (1,3);
\end{tikzpicture}
\begin{tikzpicture}[scale=0.35]
  \draw (0,4) -- (4,4);
  \draw (0,3) -- (4,3);
  \draw (0,2) -- (3,2);
  \draw (0,1) -- (2,1);
  \draw (0,4) -- (0,1);
  \draw (1,4) -- (1,1);
  \draw (2,4) -- (2,1);
  \draw (3,4) -- (3,2);
  \draw (4,4) -- (4,3);
  \draw (0.5,3.5) node {1};
  \draw (1.5,3.5) node {2};
  \draw (0.5,2.5) node {3};
  \draw (0.5,1.5) node {4};
  \draw (1.5,2.5) node {5};
  \draw (2.5,3.5) node {6};
  \draw (1.5,1.5) node {7};
  \draw (2.5,2.5) node {8};
  \draw (3.5,3.5) node {9};
  \draw (5,4) -- (9,4);
  \draw (5,3) -- (9,3);
  \draw (5,2) -- (9,2);
  \draw (5,1) -- (8,1);
  \draw (5,4) -- (5,1);
  \draw (6,4) -- (6,1);
  \draw (7,4) -- (7,1);
  \draw (8,4) -- (8,1);
  \draw (9,4) -- (9,2);
  \draw (5.5,3.5) node {1};
  \draw (6.5,3.5) node {2};
  \draw (7.5,3.5) node {3};
  \draw (8.5,3.5) node {4};
      \draw (8.5,3.5) node[draw, circle, scale=1.2] {};
  \draw (8.5,2.5) node {$x$};
    \draw[->] (9,2.5) -- (9.5,3.5);
  \draw (5.5,2.5) node {1};
  \draw (6.5,2.5) node {2};
  \draw (7.5,2.5) node {3};
      \draw (7.5,2.5) node[draw, circle, scale=1.2] {};
  \draw (5.5,1.5) node {1};
  \draw (6.5,1.5) node {2};
    \draw (8.5,2.5) node[draw, circle, scale=1.2] {};
  \draw (7.5,1.5) node {$x$};
  \draw (0,-1) node {};
\end{tikzpicture}
\begin{tikzpicture}[scale=0.35]
 \draw (0,-1) node{};
 \draw (0,5) node{};
 \draw [->,decorate,
decoration={snake,amplitude=.4mm,segment length=2mm,post length=1mm, pre length=1mm}] (0,3) -- (1,3);
\end{tikzpicture}
\begin{tikzpicture}[scale=0.35]
  \draw (0,4) -- (5,4);
  \draw (0,3) -- (5,3);
  \draw (0,2) -- (3,2);
  \draw (0,1) -- (3,1);
  \draw (0,4) -- (0,1);
  \draw (1,4) -- (1,1);
  \draw (2,4) -- (2,1);
  \draw (3,4) -- (3,1);
  \draw (4,4) -- (4,3);
  \draw (5,4) -- (5,3);
  \draw (0.5,3.5) node {1};
  \draw (1.5,3.5) node {2};
  \draw (0.5,2.5) node {3};
  \draw (0.5,1.5) node {4};
  \draw (1.5,2.5) node {5};
  \draw (2.5,3.5) node {6};
  \draw (1.5,1.5) node {7};
  \draw (2.5,2.5) node {8};
  \draw (3.5,3.5) node {9};
    \draw (2.5,1.5) node[draw, circle, scale=1.2] {};
  \draw (2.5,1.5) node {10};
    \draw (4.5,3.5) node[draw, circle, scale=1.2] {};
  \draw (4.5,3.5) node {11};
  \draw (6,4) -- (11,4);
  \draw (6,3) -- (11,3);
  \draw (6,2) -- (9,2);
  \draw (6,1) -- (9,1);
  \draw (6,4) -- (6,1);
  \draw (7,4) -- (7,1);
  \draw (8,4) -- (8,1);
  \draw (9,4) -- (9,1);
  \draw (10,4) -- (10,3);
  \draw (11,4) -- (11,3);
  \draw (6.5,3.5) node {1};
  \draw (7.5,3.5) node {2};
  \draw (8.5,3.5) node {3};
  \draw (9.5,3.5) node {4};
    \draw (10.5,3.5) node[draw, circle, scale=1.2] {};
  \draw (10.5,3.5) node {$x$};
  \draw (6.5,2.5) node {1};
  \draw (7.5,2.5) node {2};
  \draw (8.5,2.5) node {3};
  \draw (6.5,1.5) node {1};
  \draw (7.5,1.5) node {2};
  \draw (8.5,1.5) node[draw, circle, scale=1.2] {};
  \draw (8.5,1.5) node {$x$};
  \draw (0,-1) node {};
\end{tikzpicture}
\begin{tikzpicture}[scale=0.35]
 \draw (0,-1) node{};
 \draw (0,5) node{};
 \draw [->,decorate,
decoration={snake,amplitude=.4mm,segment length=2mm,post length=1mm, pre length=1mm}] (0,3) -- (1,3);
\end{tikzpicture}
\begin{tikzpicture}[scale=0.35]
  \draw (0,4) -- (5,4);
  \draw (0,3) -- (5,3);
  \draw (0,2) -- (3,2);
  \draw (0,1) -- (3,1);
  \draw (0,4) -- (0,1);
  \draw (1,4) -- (1,1);
  \draw (2,4) -- (2,1);
  \draw (3,4) -- (3,1);
  \draw (4,4) -- (4,3);
  \draw (5,4) -- (5,3);
  \draw (0.5,3.5) node {1};
  \draw (1.5,3.5) node {2};
  \draw (0.5,2.5) node {3};
  \draw (0.5,1.5) node {4};
  \draw (1.5,2.5) node {5};
  \draw (2.5,3.5) node {6};
  \draw (1.5,1.5) node {7};
  \draw (2.5,2.5) node {8};
  \draw (3.5,3.5) node {9};
  \draw (2.5,1.5) node {10};
  \draw (4.5,3.5) node {11};
  \draw (0,-1) node {};
\end{tikzpicture}
\end{example}

We see that in both cases the added cells form a horizontal strip with at most one position in the first row that is filled from left to right increasingly.
Moreover, we point out that same standard Young tableau $Q$ and partition $\mu$ still can lead to different standard Young tableau $\tilde{Q}$. Thus there is essential information in the orthogonal Littlewood-Richardson tableau:
\begin{example}
$\lambda=(4,2)$, $\mu=(2)$\vspace*{-0.3cm}

\begin{tikzpicture}[scale=0.35]
  \draw (0,4) -- (4,4);
  \draw (0,3) -- (4,3);
  \draw (0,2) -- (2,2);
  \draw (0,4) -- (0,2);
  \draw (1,4) -- (1,2);
  \draw (2,4) -- (2,2);
  \draw (3,4) -- (3,3);
  \draw (4,4) -- (4,3);
  \draw (0.5,3.5) node {1};
  \draw (1.5,3.5) node {2};
  \draw (0.5,2.5) node {5};
  \draw (3.5,3.5) node {4};
  \draw (1.5,2.5) node {6};
  \draw (2.5,3.5) node {3};
  \draw (5,0) -- (6,0);
  \draw (5,-1) -- (6,-1);
  \draw (5,-2) -- (6,-2);
  \draw (5,0) -- (5,-2);
  \draw (6,0) -- (6,-2);
  \draw (5.5,-0.5) node {1};
  \draw (5.5,-1.5) node {2};
  \draw (7,4) -- (8,4);
  \draw (7,3) -- (8,3);
  \draw (7,2) -- (8,2);
  \draw (7,1) -- (8,1);
  \draw (7,0) -- (8,0);
  \draw (7,4) -- (7,0);
  \draw (8,4) -- (8,0);
  \draw (7.5,3.5) node {1};
  \draw (7.5,2.5) node {2};
  \draw (7.5,1.5) node {3};
  \draw (7.5,0.5) node {4};
  \draw (0,-2) node{};
\end{tikzpicture}
\begin{tikzpicture}[scale=0.35]
  \draw (0,-2) node{};
 \draw (0,0) node{};
 \draw (0,5) node{};
 \draw [->,decorate,
decoration={snake,amplitude=.4mm,segment length=2mm,post length=1mm, pre length=1mm}] (0,3) -- (1.1,3);
\end{tikzpicture}
\begin{tikzpicture}[scale=0.35]
  \draw (0,-2) node{};
\draw (0,0.2) node{};
  \draw (0,-2) node{};
  \draw (0,4) -- (4,4);
  \draw (0,3) -- (4,3);
  \draw (0,2) -- (2,2);
  \draw (0,1) -- (2,1);
  \draw (0,4) -- (0,1);
  \draw (1,4) -- (1,1);
  \draw (2,4) -- (2,1);
  \draw (3,4) -- (3,3);
  \draw (4,4) -- (4,3);
  \draw (0.5,3.5) node {1};
  \draw (1.5,3.5) node {2};
  \draw (0.5,2.5) node {5};
  \draw (3.5,3.5) node {4};
  \draw (1.5,2.5) node {6};
  \draw (2.5,3.5) node {3};
  \draw (0.5,1.5) node {7};
  \draw (1.5,1.5) node {8};
\end{tikzpicture}
\raisebox{1.3cm}{whereas}
\begin{tikzpicture}[scale=0.35]
  \draw (0,-2) node{};
  \draw (0,4) -- (4,4);
  \draw (0,3) -- (4,3);
  \draw (0,2) -- (2,2);
  \draw (0,4) -- (0,2);
  \draw (1,4) -- (1,2);
  \draw (2,4) -- (2,2);
  \draw (3,4) -- (3,3);
  \draw (4,4) -- (4,3);
  \draw (0.5,3.5) node {1};
  \draw (1.5,3.5) node {2};
  \draw (0.5,2.5) node {5};
  \draw (3.5,3.5) node {4};
  \draw (1.5,2.5) node {6};
  \draw (2.5,3.5) node {3};
  \draw (6,4) -- (7,4);
  \draw (6,3) -- (7,3);
  \draw (6,2) -- (7,2);
  \draw (6,4) -- (6,2);
  \draw (7,4) -- (7,2);
  \draw (6.5,3.5) node {1};
  \draw (6.5,2.5) node {2};
  \draw (5,2) -- (6,2);
  \draw (5,1) -- (6,1);
  \draw (5,0) -- (6,0);
  \draw (5,2) -- (5,0);
  \draw (6,2) -- (6,0);
  \draw (5.5,1.5) node {3};
  \draw (5.5,0.5) node {4};
  \draw (7.5,4) -- (8.5,4);
  \draw (7.5,3) -- (8.5,3);
  \draw (7.5,2) -- (8.5,2);
  \draw (7.5,4) -- (7.5,2);
  \draw (8.5,4) -- (8.5,2);
  \draw (8,3.5) node {1};
  \draw (8,2.5) node {2};
\end{tikzpicture}
\begin{tikzpicture}[scale=0.35]
  \draw (0,-2) node{};
 \draw (0,0) node{};
 \draw (0,5) node{};
 \draw [->,decorate,
decoration={snake,amplitude=.4mm,segment length=2mm,post length=1mm, pre length=1mm}] (0,3) -- (1.1,3);
\end{tikzpicture}
\begin{tikzpicture}[scale=0.35]
  \draw (0,-2) node{};
\draw (0,0.2) node{};
  \draw (0,4) -- (4,4);
  \draw (0,3) -- (4,3);
  \draw (0,2) -- (4,2);
  \draw (0,4) -- (0,2);
  \draw (1,4) -- (1,2);
  \draw (2,4) -- (2,2);
  \draw (3,4) -- (3,2);
  \draw (4,4) -- (4,2);
  \draw (0.5,3.5) node {1};
  \draw (1.5,3.5) node {2};
  \draw (0.5,2.5) node {5};
  \draw (3.5,3.5) node {4};
  \draw (1.5,2.5) node {6};
  \draw (2.5,3.5) node {3};
  \draw (2.5,2.5) node {7};
  \draw (3.5,2.5) node {8};
\end{tikzpicture}.
\end{example}

We will see that every step in Algorithm~\ref{algoLR3} is reversible, and Algorithm~\ref{UalgoLR3} is its inverse (see Lemma~\ref{LemfirstInverse}). To illustrate the latter we use our running example and give a full iteration Algorithm~\ref{UalgoLR3}.
\begin{example}
In our running example we have $\mu=(2)$ and the following standard Young tableau $\tilde{Q}$:

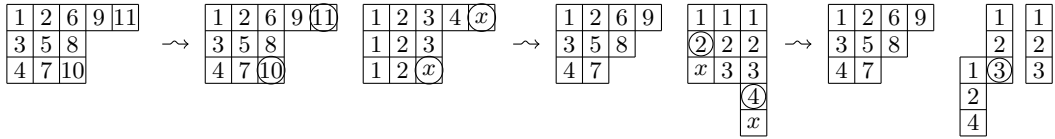
\begin{tikzpicture}[scale=0.35]
  \draw (0,4) -- (5,4);
  \draw (0,3) -- (5,3);
  \draw (0,2) -- (3,2);
  \draw (0,1) -- (3,1);
  \draw (0,4) -- (0,1);
  \draw (1,4) -- (1,1);
  \draw (2,4) -- (2,1);
  \draw (3,4) -- (3,1);
  \draw (4,4) -- (4,3);
  \draw (5,4) -- (5,3);
  \draw (0.5,3.5) node {1};
  \draw (1.5,3.5) node {2};
  \draw (0.5,2.5) node {3};
  \draw (0.5,1.5) node {4};
  \draw (1.5,2.5) node {5};
  \draw (2.5,3.5) node {6};
  \draw (1.5,1.5) node {7};
  \draw (2.5,2.5) node {8};
  \draw (3.5,3.5) node {9};
  \draw (2.5,1.5) node {10};
  \draw (4.5,3.5) node {11};
  \draw (0,-1) node {};
\end{tikzpicture}
\begin{tikzpicture}[scale=0.35]
 \draw (0,-0.5) node{};
 \draw (0,5) node{};
 \draw [->,decorate,
decoration={snake,amplitude=.4mm,segment length=2mm,post length=1mm, pre length=1mm}] (0,3) -- (1,3);
\end{tikzpicture}
\begin{tikzpicture}[scale=0.35]
  \draw (0,4) -- (5,4);
  \draw (0,3) -- (5,3);
  \draw (0,2) -- (3,2);
  \draw (0,1) -- (3,1);
  \draw (0,4) -- (0,1);
  \draw (1,4) -- (1,1);
  \draw (2,4) -- (2,1);
  \draw (3,4) -- (3,1);
  \draw (4,4) -- (4,3);
  \draw (5,4) -- (5,3);
  \draw (0.5,3.5) node {1};
  \draw (1.5,3.5) node {2};
  \draw (0.5,2.5) node {3};
  \draw (0.5,1.5) node {4};
  \draw (1.5,2.5) node {5};
  \draw (2.5,3.5) node {6};
  \draw (1.5,1.5) node {7};
  \draw (2.5,2.5) node {8};
  \draw (3.5,3.5) node {9};
    \draw (2.5,1.5) node[draw, circle, scale=1.2] {};
  \draw (2.5,1.5) node {10};
    \draw (4.5,3.5) node[draw, circle, scale=1.2] {};
  \draw (4.5,3.5) node {11};
  \draw (6,4) -- (11,4);
  \draw (6,3) -- (11,3);
  \draw (6,2) -- (9,2);
  \draw (6,1) -- (9,1);
  \draw (6,4) -- (6,1);
  \draw (7,4) -- (7,1);
  \draw (8,4) -- (8,1);
  \draw (9,4) -- (9,1);
  \draw (10,4) -- (10,3);
  \draw (11,4) -- (11,3);
  \draw (6.5,3.5) node {1};
  \draw (7.5,3.5) node {2};
  \draw (8.5,3.5) node {3};
  \draw (9.5,3.5) node {4};
    \draw (10.5,3.5) node[draw, circle, scale=1.2] {};
  \draw (10.5,3.5) node {$x$};
  \draw (6.5,2.5) node {1};
  \draw (7.5,2.5) node {2};
  \draw (8.5,2.5) node {3};
  \draw (6.5,1.5) node {1};
  \draw (7.5,1.5) node {2};
  \draw (8.5,1.5) node[draw, circle, scale=1.2] {};
  \draw (8.5,1.5) node {$x$};
  \draw (0,-1) node {};
\end{tikzpicture}
\begin{tikzpicture}[scale=0.35]
 \draw (0,-0.5) node{};
 \draw (0,5) node{};
 \draw [->,decorate,
decoration={snake,amplitude=.4mm,segment length=2mm,post length=1mm, pre length=1mm}] (0,3) -- (1,3);
\end{tikzpicture}
\begin{tikzpicture}[scale=0.35]
  \draw (0,4) -- (4,4);
  \draw (0,3) -- (4,3);
  \draw (0,2) -- (3,2);
  \draw (0,1) -- (2,1);
  \draw (0,4) -- (0,1);
  \draw (1,4) -- (1,1);
  \draw (2,4) -- (2,1);
  \draw (3,4) -- (3,2);
  \draw (4,4) -- (4,3);
  \draw (0.5,3.5) node {1};
  \draw (1.5,3.5) node {2};
  \draw (0.5,2.5) node {3};
  \draw (0.5,1.5) node {4};
  \draw (1.5,2.5) node {5};
  \draw (2.5,3.5) node {6};
  \draw (1.5,1.5) node {7};
  \draw (2.5,2.5) node {8};
  \draw (3.5,3.5) node {9};
  \draw (5,4) -- (8,4);
  \draw (5,3) -- (8,3);
  \draw (5,2) -- (8,2);
  \draw (5,1) -- (8,1);
  \draw (7,0) -- (8,0);
  \draw (7,-1) -- (8,-1);
  \draw (5,4) -- (5,1);
  \draw (6,4) -- (6,1);
  \draw (7,4) -- (7,-1);
  \draw (8,4) -- (8,-1);
  \draw (5.5,3.5) node {1};
  \draw (6.5,3.5) node {1};
  \draw (7.5,3.5) node {1};
  \draw (5.5,2.5) circle (13pt) node {2};
  \draw (6.5,2.5) node {2};
  \draw (7.5,2.5) node {2};
  \draw (5.5,1.5) node {$x$};
  \draw (6.5,1.5) node {3};
  \draw (7.5,1.5) node {3};
  \draw (7.5,-0.5) node {$x$};
  \draw (7.5,0.5) circle (13pt) node {4};
  \draw (0,-1) node {};
\end{tikzpicture}
\begin{tikzpicture}[scale=0.35]
 \draw (0,-0.5) node{};
 \draw (0,5) node{};
 \draw [->,decorate,
decoration={snake,amplitude=.4mm,segment length=2mm,post length=1mm, pre length=1mm}] (0,3) -- (1,3);
\end{tikzpicture}
\begin{tikzpicture}[scale=0.35]
  \draw (0,4) -- (4,4);
  \draw (0,3) -- (4,3);
  \draw (0,2) -- (3,2);
  \draw (0,1) -- (2,1);
  \draw (0,4) -- (0,1);
  \draw (1,4) -- (1,1);
  \draw (2,4) -- (2,1);
  \draw (3,4) -- (3,2);
  \draw (4,4) -- (4,3);
  \draw (0.5,3.5) node {1};
  \draw (1.5,3.5) node {2};
  \draw (0.5,2.5) node {3};
  \draw (0.5,1.5) node {4};
  \draw (1.5,2.5) node {5};
  \draw (2.5,3.5) node {6};
  \draw (1.5,1.5) node {7};
  \draw (2.5,2.5) node {8};
  \draw (3.5,3.5) node {9};
  \draw (6,4) -- (7,4);
  \draw (6,3) -- (7,3);
  \draw (5,2) -- (7,2);
  \draw (5,1) -- (7,1);
  \draw (5,0) -- (6,0);
  \draw (5,-1) -- (6,-1);
  \draw (5,2) -- (5,-1);
  \draw (6,4) -- (6,-1);
  \draw (7,4) -- (7,1);
  \draw (5.5,1.5) node {1};
  \draw (6.5,3.5) node {1};
  \draw (5.5,0.5) node {2};
  \draw (6.5,2.5) node {2};
  \draw (6.5,1.5) circle (13pt) node {3};
  \draw (5.5,-0.5) node {4};
  \draw (7.5,4) -- (8.5,4);
  \draw (7.5,3) -- (8.5,3);
  \draw (7.5,2) -- (8.5,2);
  \draw (7.5,1) -- (8.5,1);
  \draw (7.5,4) -- (7.5,1);
  \draw (8.5,4) -- (8.5,1);
  \draw (8,3.5) node {1};
  \draw (8,2.5) node {2};
  \draw (8,1.5) node {3};
  \draw (0,-1) node {};
\end{tikzpicture}

Now we have the case that the two left columns are odd. Thus one element has to change column. In this case it is the $3$ in the second column. What we obtain are our original tableaux:
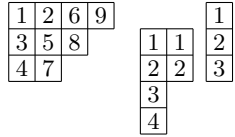
\[
\begin{tikzpicture}[scale=0.35]
  \draw (0,4) -- (4,4);
  \draw (0,3) -- (4,3);
  \draw (0,2) -- (3,2);
  \draw (0,1) -- (2,1);
  \draw (0,4) -- (0,1);
  \draw (1,4) -- (1,1);
  \draw (2,4) -- (2,1);
  \draw (3,4) -- (3,2);
  \draw (4,4) -- (4,3);
  \draw (0.5,3.5) node {1};
  \draw (1.5,3.5) node {2};
  \draw (0.5,2.5) node {3};
  \draw (0.5,1.5) node {4};
  \draw (1.5,2.5) node {5};
  \draw (2.5,3.5) node {6};
  \draw (1.5,1.5) node {7};
  \draw (2.5,2.5) node {8};
  \draw (3.5,3.5) node {9};
  \draw (5,3) -- (7,3);
  \draw (5,2) -- (7,2);
  \draw (5,1) -- (7,1);
  \draw (5,0) -- (6,0);
  \draw (5,-1) -- (6,-1);
  \draw (5,3) -- (5,-1);
  \draw (6,3) -- (6,-1);
  \draw (7,3) -- (7,1);
  \draw (5.5,2.5) node {1};
  \draw (6.5,2.5) node {1};
  \draw (5.5,1.5) node {2};
  \draw (6.5,1.5) node {2};
  \draw (5.5,0.5) node {3};
  \draw (5.5,-0.5) node {4};
  \draw (7.5,4) -- (8.5,4);
  \draw (7.5,3) -- (8.5,3);
  \draw (7.5,2) -- (8.5,2);
  \draw (7.5,1) -- (8.5,1);
  \draw (7.5,4) -- (7.5,1);
  \draw (8.5,4) -- (8.5,1);
  \draw (8,3.5) node {1};
  \draw (8,2.5) node {2};
  \draw (8,1.5) node {3};
  \draw (0,-1) node {};
\end{tikzpicture}\]
\end{example}

We proceed by formulating and proving properties of Algorithm~\ref{algoLR3} and~\ref{UalgoLR3}.

Since the filling is not mentioned in the algorithm, the following proposition is clear.

\begin{proposition}
Algorithm~\ref{algoLR3} does not depend on the filling of the standard Young tableau. Only the shape matters.
\end{proposition}

\begin{lemma}
\label{LemHorizonal1}
Algorithm~\ref{algoLR3} is well-defined and returns a standard Young tableau, all of whose row lengths have the same parity, and a partition $\mu$ with a single part. The added cells in the standard Young tableau in Algorithm~\ref{algoLR3} form a horizontal strip with at most one position in the first row. They are filled from left to right increasingly.
\end{lemma}

\begin{proof}
Considering the classification in Theorem~\ref{LemmaLR}, we note that either $\mu_1$ cells or $\mu_1-1$ cells labeled $x$ are added for those cells being shifted from one row to another. An additional cell labeled $x$ is added to the third row if and only if $L$ is of Case 1 or 2 and the residuum is 1. Thus Algorithm~\ref{algoLR3} is well-defined.

We now show that, when \enquote{adjusting the parity}, the third row of $\tilde{L}$ is non-empty if the total number of cells in $\tilde{L}$ is odd. In Case 1 $\tilde{L}$ has an even number of cells ($|\tilde{Q}|=\mu_1+|\lambda|=\mu_1+\mu_1+2b-a+c$). In Case 2 and Case 3 the topmost entry of the tail cannot end up in the first row of $\tilde{L}$ (compare again with the classification in Theorem~\ref{LemmaLR}). So it either ends up in the third row or in the second row creating a cell labeled $x$ in the third row.

By construction all rows of $\tilde{Q}$ have the same parity. It remains to show that $\tilde{Q}$ is indeed a standard tableau and that the added cells behave as described. This will follow from the fact that the added cells labeled $x$ form a horizontal strip, which is later filled with the numbers $|\lambda|+1,|\lambda|+2,\dots,|\lambda|+\mu_1$ from left to right.

First we observe that the cells shifted into another row (from the tail or, in Case 3, from the bottom of the middle column) form a horizontal strip essentially because they all have different entries - compare with Theorem~\ref{LemmaLR}.

Thus, when the residuum is $0$ or $L$ is of Case 3, cells labeled $x$ also form a horizontal strip, because they are added one row below. This strip only spans two rows. Therefore, after moving a cell to a higher row when \enquote{adjusting the parity}, the cells labeled $x$ still form a horizontal strip.
 
When the residuum is $1$ (and thus $a>0$ and as $a$ is even $a\geq 2$) and $L$ is of Case 1 or 2, an additional cell labeled $x$ is added to the third row. However in this case inspection of the classification of Theorem~\ref{LemmaLR} shows that the third row including all the added cells (including the additional one) is at most as long as the middle line, that is, $\lambda_2$. Now we can argue as above.

There cannot be more than one cell shifted to the first row, as we do not shift more than one cell.
\end{proof}

\begin{lemma}
Algorithm~\ref{UalgoLR3} is well-defined.
\end{lemma}

\begin{proof}
We have to show that there are cells to mark and that there is a cell to shift in the case that the middle and leftmost column have odd length.

As the cells labeled $x$ form a horizontal strip before rotating, they form a vertical strip afterwards. Therefore there are always enough cells to mark in the column to the right. Note that in the rightmost column there is at most one cell labeled $x$. Thus, if something is marked, it is marked in the same column. Thus we have to consider the parts that mark cells in the same column, namely the columns with an odd number of cells labeled $x$.

If there is an odd number of cells labeled $x$ in the leftmost column, either this marks something from the leftmost column or also from the middle one.

If there is an odd number of cells labeled $x$ in the middle column we need to show that the middle column is longer than the number of cells labeled $x$ in the leftmost one. Therefore we consider the original parity of the lines. As the number of cells labeled $x$ in the middle column is odd, the number of unmarked cells has opposite parity from the original column. Thus the leftmost column contains at least one less box.

The same argument holds for the rightmost column.

It remains to show that in \enquote{fix parity} there is a cell to shift if necessary. In particular we have to show that the middle column and the leftmost one are not the same if $\mu_1\neq 0$. Suppose therefore that they are the same.

We consider cells that might been shifted to the tail, thus to the leftmost column. Cells shifted from the rightmost column need to have bigger labels than any cells originally in the middle column as we shift the bottommost cells and there where at least as many cells labeled $x$ in the middle column. Thus there cannot be such cells in the tail.

There cannot be cells in the tail that are shifted from the middle column to the leftmost one as those columns are the same.

Finally there can be only one cell in the tail coming from the leftmost column. This happens if there is an odd number of $x$ in the leftmost column. However as the leftmost column and the middle one have to be the same, this is not possible due to the parity of the lengths.

Therefore the tail is $\emptyset$.
\end{proof}

\begin{lemma}
Algorithm~\ref{UalgoLR3} produces an orthogonal Littlewood-Richardson tableau $L$ in $\LRtabs$ and a standard Young tableau $Q$ in $\SYT(\lambda)$.
\end{lemma}

\begin{proof}
Algorithm~\ref{UalgoLR3} produces a standard Young tableau $Q$ because its input $\tilde{Q}$ is a standard Young tableau from which the largest $\mu_1$ entries are deleted to obtain $Q$.

Now we consider $L$. We point out that $L$ has lines filled with $1,2,\dots, \lambda_\ell$. Moreover, there are gaps $j$ only in the tail, or possibly in the middle column at the bottommost position. Those are always in such a way, that $j-1$ is the last position of a column to the right.  Furthermore, due to \enquote{fix parity} the parity of $a$ and $b$ is always even.
The residuum is always less or equal 1 as there is at most one number from the original leftmost column in the tail.

Now if the rightmost column has even length, the middle one also has even length because label cells with $x$ and mark and shift cells can only change parity if the leftmost line consists of exactly one position and this is labeled $x$ and this does not change the parity of the rightmost line. Thus if the rightmost column has even length and if the tail has length $0$, $L$ is an orthogonal Littlewood-Richardson tableau.

If the rightmost column has odd length, there is no gap in the middle column and the middle column is shorter than the rightmost one, we show that this is Case 2, thus $c+\mu_1<\lambda_1$.
Suppose the rightmost column has odd length and $c+\mu_1=\lambda_1$. Then there would have been no cell labeled $x$ in the leftmost column. Thus the leftmost column with out the tail is has odd length. Thus the middle column has odd length. Thus \enquote{fix parity} is done. This has to move an element from the leftmost column to the middle one as otherwise the tail would get longer, and this has to be the topmost position of the tail. Thus we get either a gap in the middle column or a middle column that is longer than the right one.

It there is a gap in the middle column or the middle column is longer than the rightmost one, we have to show that this is Case 3, thus the residuum is 1, the middle line is only in the middle column and $a>0$ and the rightmost column has odd length. If this occurs, there was some change from the leftmost column to the middle one in \enquote{fix parity}. Therefore the rightmost column has odd length. Now the residuum is 1 and $a>0$ as we shifted the leftmost column down by one. The middle line is only in the middle column as middle line elements in the leftmost column would have been smaller candidates to change column.
\end{proof}

Now we show the (perhaps surprising) fact that Algorithm~\ref{algoLR3} and Algorithm~\ref{UalgoLR3} are inverse.

\begin{lemma}
\label{LemfirstInverse}
Algorithm~\ref{algoLR3} and Algorithm~\ref{UalgoLR3} are inverse.
\end{lemma}

\begin{proof}
We show first, that applying Algorithm~\ref{algoLR3} on a pair consisting of an orthogonal Littlewood-Richardson tableau $L$ and a standard Young tableau $Q$ with $\mu\neq \emptyset$ and Algorithm~\ref{UalgoLR3} on the outcome ($\tilde{Q}$, $\mu$) returns the original pair. This is obvious for the standard Young tableau $Q$ as exactly the added cells are removed.
For $L$ we consider Cases 1 to 4 separately.
\begin{enumerate}

\item We start with Case 1. In this case the rows of $\tilde{Q}$ have even length (because $|\tilde{Q}|=\mu_1+|\lambda|=\mu_1+\mu_1+2b-a+c$). We show that Algorithm~\ref{UalgoLR3} marks precisely the cells coming from the tail. The following sketch illustrates the situation just before \enquote{adjust parity} in Algorithm~\ref{algoLR3}. We display the three rows of the row-strict tableau obtained from $L$ as three separate columns and distinguish the two possible values of the residuum.

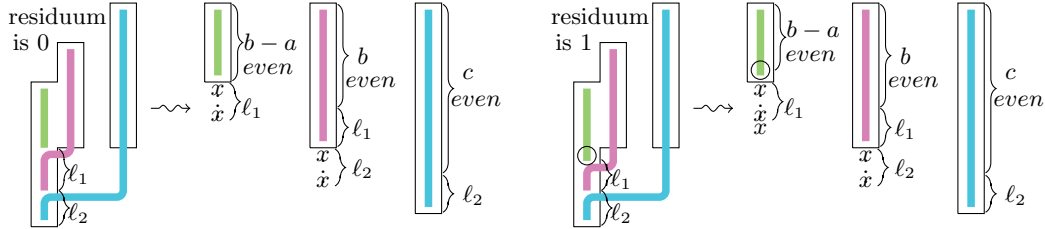
\begin{tikzpicture}[scale=0.35]
  \draw (0,-1.5) -- (0,4) -- (1,4) -- (1, 5.5) -- (2,5.5) -- (2,1.5) -- (1,1.5) -- (1,-1.5)  -- (0,-1.5);
  \draw (3,1.5) -- (3,7) -- (4,7) -- (4, 1.5) -- (3,1.5);
  \draw[line width=1mm, color=YellowGreen] (0.5,1.5) -- (0.5,3.75);
  \draw[line width=1mm, color=Thistle] (0.5,-0.125) %
  -- (0.5,1) arc[start angle=180, end angle=90, radius=0.25] %
  -- (1.25,1.25)  arc[start angle=270, end angle=360, radius=0.25] %
  -- (1.5,5.25);
  \draw[line width=1mm, color=SkyBlue] (0.5,-1.25) %
  -- (0.5,-0.625) arc[start angle=180, end angle=90, radius=0.25] %
  -- (3.25,-0.375) arc[start angle=270, end angle=360, radius=0.25] %
  -- (3.5,6.75);
    \draw (0,-1.5) node{};
       \draw[black,decorate,decoration={brace,amplitude=4pt}] (1.05,1.45) -- (1.05,-0.12);
       \draw[black,decorate,decoration={brace,amplitude=4pt}] (1.05,-0.13) -- (1.05,-1.45);
    \draw (1.8,0.5) node{$\ell_1$};
    \draw (1.8,-1) node{$\ell_2$};
    \draw (1,6.5) node{residuum};
    \draw (0,5.5) node{is $0$};
\end{tikzpicture}
\begin{tikzpicture}[scale=0.35]
 \draw (0,-1.5) node{};
 \draw (0,5) node{};
 \draw [->,decorate,
decoration={snake,amplitude=.4mm,segment length=2mm,post length=1mm, pre length=1mm}] (0,3) -- (1.5,3);
\end{tikzpicture}
\begin{tikzpicture}[scale=0.35]
  \draw (0,4) -- (0,7) -- (1,7) --  (1,4) -- (0,4);
  \draw (4,1.5) -- (4, 7) -- (5,7) -- (5,1.5) -- (4,1.5);
  \draw (8,-1) -- (8,7) -- (9,7) -- (9, -1) -- (8,-1);
  \draw[line width=1mm, color=YellowGreen] (0.5,4.25) -- (0.5,6.75);
  \draw[line width=1mm, color=Thistle] (4.5,1.75)  -- (4.5,6.75);
  \draw[line width=1mm, color=SkyBlue] (8.5,-0.75)  -- (8.5,6.75);
   \draw[black,decorate,decoration={brace,amplitude=4pt}] (1.05,6.95) -- (1.05,4.05);
      \draw[black,decorate,decoration={brace,amplitude=4pt}] (0.95,3.95) -- (0.95,2.55);
   \draw[black,decorate,decoration={brace,amplitude=4pt}] (5.05,6.95) -- (5.05,3.05);
   \draw[black,decorate,decoration={brace,amplitude=4pt}] (5.05,3) -- (5.05,1.55);
      \draw[black,decorate,decoration={brace,amplitude=4pt}] (4.95,1.45) -- (4.95,0.05);
   \draw[black,decorate,decoration={brace,amplitude=4pt}] (9.05,6.95) -- (9.05,0.55);
   \draw[black,decorate,decoration={brace,amplitude=4pt}] (9.05,0.45) -- (9.05,-0.95);
  \draw (6,5) node {$b$};
  \draw (6.3,4) node {$even$};
  \draw (10,4.3) node {$c$};
  \draw (10.3,3.3) node {$even$};
  \draw (2.5,5.5) node {$b-a$};
  \draw (2.4,4.5) node {$even$};
  \draw (1.8,3) node {$\ell_1$};
  \draw (6,2.2) node {$\ell_1$};
  \draw (6,0.7) node {$\ell_2$};
  \draw (10.2,-0.3) node {$\ell_2$};
  \draw (0.5,3.7) node{$x$};
  \draw (0.5,3.4) node{$\vdots$};
  \draw (0.5,2.7) node{$x$};
  \draw (4.5,1.2) node{$x$};
  \draw (4.5,0.9) node{$\vdots$};
  \draw (4.5,0.2) node{$x$};
  \draw (0,-1.5) node{};
\end{tikzpicture}
\hspace{0.3cm}
\begin{tikzpicture}[scale=0.35]
  \draw (0,-1.5) -- (0,4) -- (1,4) -- (1, 5.5) -- (2,5.5) -- (2,1.5) -- (1,1.5) -- (1,-1.5)  -- (0,-1.5);
  \draw (3,1.5) -- (3,7) -- (4,7) -- (4, 1.5) -- (3,1.5);
  \draw[line width=1mm, color=YellowGreen] (0.5,1) -- (0.5,3.75);
  \draw[line width=1mm, color=Thistle] (0.5,-0.125) %
  -- (0.5,0.5) arc[start angle=180, end angle=90, radius=0.25] %
  -- (1.25,0.75)  arc[start angle=270, end angle=360, radius=0.25] %
  -- (1.5,5.25);
  \draw[line width=1mm, color=SkyBlue] (0.5,-1.25) %
  -- (0.5,-0.625) arc[start angle=180, end angle=90, radius=0.25] %
  -- (3.25,-0.375) arc[start angle=270, end angle=360, radius=0.25] %
  -- (3.5,6.75);
    \draw (0,-1.5) node{};
       \draw[black,decorate,decoration={brace,amplitude=4pt}] (1.05,1.05) -- (1.05,-0.12);
       \draw[black,decorate,decoration={brace,amplitude=4pt}] (1.05,-0.13) -- (1.05,-1.45);
        \draw (0.5,1.2) node[draw, circle, scale=0.8] {};
    \draw (1.8,0.3) node{$\ell_1$};
    \draw (1.8,-1) node{$\ell_2$};
    \draw (1,6.5) node{residuum};
    \draw (0,5.5) node{is $1$};
\end{tikzpicture}
\begin{tikzpicture}[scale=0.35]
 \draw (0,-1.5) node{};
 \draw (0,5) node{};
 \draw [->,decorate,
decoration={snake,amplitude=.4mm,segment length=2mm,post length=1mm, pre length=1mm}] (0,3) -- (1.5,3);
\end{tikzpicture}
\begin{tikzpicture}[scale=0.35]
  \draw (0,4) -- (0,7) -- (1,7) --  (1,4) -- (0,4);
  \draw (4,1.5) -- (4, 7) -- (5,7) -- (5,1.5) -- (4,1.5);
  \draw (8,-1) -- (8,7) -- (9,7) -- (9, -1) -- (8,-1);
  \draw[line width=1mm, color=YellowGreen] (0.5,4.25) -- (0.5,6.75);
  \draw[line width=1mm, color=Thistle] (4.5,1.75)  -- (4.5,6.75);
  \draw[line width=1mm, color=SkyBlue] (8.5,-0.75)  -- (8.5,6.75);
   \draw[black,decorate,decoration={brace,amplitude=4pt}] (1.05,6.95) -- (1.05,4.45);
      \draw[black,decorate,decoration={brace,amplitude=4pt}] (0.95,3.95) -- (0.95,2.55);
   \draw[black,decorate,decoration={brace,amplitude=4pt}] (5.05,6.95) -- (5.05,3.05);
   \draw[black,decorate,decoration={brace,amplitude=4pt}] (5.05,3) -- (5.05,1.55);
      \draw[black,decorate,decoration={brace,amplitude=4pt}] (4.95,1.45) -- (4.95,0.05);
   \draw[black,decorate,decoration={brace,amplitude=4pt}] (9.05,6.95) -- (9.05,0.55);
   \draw[black,decorate,decoration={brace,amplitude=4pt}] (9.05,0.45) -- (9.05,-0.95);
          \draw (0.5,4.45) node[draw, circle, scale=0.8] {};
  \draw (6,5) node {$b$};
  \draw (6.3,4) node {$even$};
  \draw (10,4.3) node {$c$};
  \draw (10.3,3.3) node {$even$};
  \draw (2.5,6) node {$b-a$};
  \draw (2.4,5) node {$even$};
  \draw (1.8,3) node {$\ell_1$};
  \draw (6,2.2) node {$\ell_1$};
  \draw (6,0.7) node {$\ell_2$};
  \draw (10.2,-0.3) node {$\ell_2$};
  \draw (0.5,3.7) node{$x$};
  \draw (0.5,3.4) node{$\vdots$};
  \draw (0.5,2.7) node{$x$};
  \draw (0.5,2.2) node{$x$};
  \draw (4.5,1.2) node{$x$};
  \draw (4.5,0.9) node{$\vdots$};
  \draw (4.5,0.2) node{$x$};
  \draw (0,-1.5) node{};
\end{tikzpicture}

Now we have eight cases as $\ell_1$ and $\ell_2$ could be even or odd. We first consider the case of residuum $0$ and $\ell_1$ and $\ell_2$ both even.
In this case we mark $\ell_1$ position in the middle column and $\ell_2$ in the rightmost one, thus exactly those that came from the tail. If for example $\ell_1$ is odd and $\ell_2$ is even, we see that one cell labeled $x$ will move to the middle column, and an even number of such positions will remain in the leftmost one. The moved position will mark one position in the middle column (there are now $\ell_2+1$ cells labeled $x$), the other ones (as there are even many left) will also mark one position each in the middle column, thus exactly those coming from the tail. 

The other cases are analogous so we only consider one more example: residuum 1 and $\ell_1$ and $\ell_2$ both odd. Now our leftmost column has odd length ($(b-a)+1+\ell_1+1$), our middle one has even length ($b+\ell_1+\ell_2$) and out rightmost one has odd length ($c+\ell_2$). Thus a cell labeled $x$ will move from the leftmost column to the rightmost one. Now there is an odd number of cells labeled $x$ in the leftmost column ($\ell_1$), an odd number of cells labeled $x$ in the middle column ($\ell_2$) and one thus an odd number of cells labeled $x$ in the rightmost column. Thus one cell is marked in the leftmost column, $\ell_1-1+1$ in the middle column and $\ell_2-1+1$ in the rightmost one. Those are exactly those coming from the tail.

There is no \enquote{fix parity} in Algorithm~\ref{UalgoLR3} in Case 1.

\item Case 2 is a little bit more complicated but similar. In this case the rows of $\tilde{Q}$ have odd length as $c$ is odd. To illustrate the situation we get the same sketch as above, except that $c$ is odd.

Again, we provide the details only for one of the eight different cases and leave the remaining to the reader. Let us consider the case of residuum $1$, and $\ell_1$ and $\ell_2$ both even.

In Algorithm~\ref{algoLR3} we move a cell labeled $x$ from the leftmost column to the middle one to obtain odd length. Now the number of cells labeled $x$ in the middle row  is $\ell_2+1$. Thus in Algorithm~\ref{UalgoLR3} there are $\ell_2$ cells of the rightmost column and one cell of the middle column marked. The number of cells labeled $x$ in the leftmost column is $\ell_1$. Thus $\ell_1$ further cells are marked in the middle column.

We now have to show that after \enquote{fix parity}, exactly the $\ell_1$ largest of the $\ell_1+1$ marked cells from the middle column end up in the tail.

Because $b-1$, the number of remaining non-marked cells in the middle column is odd and $\mu_1\neq 0$ we have to \enquote{fix the parity} in Algorithm~\ref{UalgoLR3}. Thus we consider the possible candidates for such a shift. The largest entry in the leftmost column not in the tail is smaller or equal to the largest entry in the middle column because this column is longer as the residuum is $1$ and thus $a>0$. Thus this cannot change. The largest entry in the middle column cannot change either, as this would make the tail longer by $1$. Thus the smallest entry that was shifted from the middle column to the leftmost one has to change to the middle column. This is exactly what we wanted to show. Note that by moving it, we also get residuum $1$.

\item Case 3 is also similar but has a slightly different illustration. In this case the rows of $\tilde{Q}$ have also odd length as $c$ is odd and the residuum is always $1$.

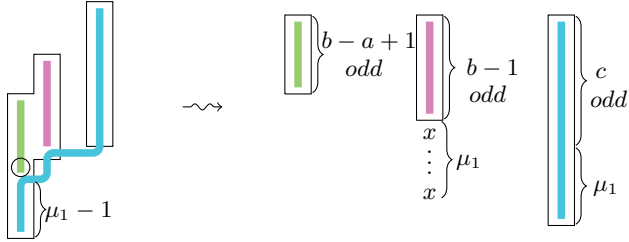
\begin{tikzpicture}[scale=0.35]
  \draw (0,-2) -- (0,3.5) -- (1,3.5) -- (1, 5) -- (2,5) -- (2,1) -- (1,1)  -- (1,-2)  -- (0,-2);
  \draw (3,1.5) -- (3,7) -- (4,7) -- (4, 1.5) -- (3,1.5);
  \draw[line width=1mm, color=YellowGreen] (0.5,0.5) -- (0.5,3.25);
    \draw[line width=1mm, color=Thistle] (1.5,1.5) -- (1.5,4.75);
  \draw[line width=1mm, color=SkyBlue] (0.5,-1.75) %
  -- (0.5,0) arc[start angle=180, end angle=90, radius=0.25] %
  -- (1.25,0.25) arc[start angle=270, end angle=360, radius=0.25] %
  -- (1.5,1) arc[start angle=180, end angle=90, radius=0.25] %
  -- (3.25,1.25) arc[start angle=270, end angle=360, radius=0.25] %
  -- (3.5,6.75);
  \draw (0,-2) node{};
          \draw (0.5,0.7) node[draw, circle, scale=0.8] {};
 \draw[black,decorate,decoration={brace,amplitude=4pt}] (1.05,0.15) -- (1.05,-1.95);
 \draw (2.7,-1) node {$\mu_1-1$};
\end{tikzpicture}
\hspace{0.5cm}
\begin{tikzpicture}[scale=0.35]
  \draw (0,-2) node{};
 \draw (0,5) node{};
 \draw [->,decorate,
decoration={snake,amplitude=.4mm,segment length=2mm,post length=1mm, pre length=1mm}] (0,3) -- (1.5,3);
\end{tikzpicture}
\hspace{0.5cm}
\begin{tikzpicture}[scale=0.35]
  \draw (0,4) -- (0,7) -- (1,7) --  (1,4) -- (0,4);
  \draw (5,3) -- (5, 7) -- (6,7) -- (6,3) -- (5,3);
  \draw (10,-1) -- (10,7) -- (11,7) -- (11, -1) -- (10,-1);
  \draw[line width=1mm, color=YellowGreen] (0.5,4.25) -- (0.5,6.75);
  \draw[line width=1mm, color=Thistle] (5.5,3.25)  -- (5.5,6.75);
  \draw[line width=1mm, color=SkyBlue] (10.5,-0.75)  -- (10.5,6.75);
   \draw[black,decorate,decoration={brace,amplitude=4pt}] (1.05,6.95) -- (1.05,4.05);
   \draw[black,decorate,decoration={brace,amplitude=4pt}] (6.05,6.95) -- (6.05,3.05);
      \draw[black,decorate,decoration={brace,amplitude=4pt}] (5.95,2.95) -- (5.95,0.05);
   \draw[black,decorate,decoration={brace,amplitude=4pt}] (11.05,6.95) -- (11.05,2.05);
   \draw[black,decorate,decoration={brace,amplitude=4pt}] (11.05,1.95) -- (11.05,-0.95);
  \draw (7.9,5) node {$b-1$};
  \draw (7.7,4) node {$odd$};
  \draw (12,4.8) node {$c$};
  \draw (12.3,3.8) node {$odd$};
  \draw (3.2,6) node {$b-a+1$};
  \draw (3,5) node {$odd$};
  \draw (7,1.4) node {$\mu_1$};
  \draw (12.2,0.4) node {$\mu_1$};
  \draw (5.5,2.5) node{$x$};
  \draw (5.5,1.7) node{$\vdots$};
  \draw (5.5,0.2) node{$x$};
  \draw (0,-1.5) node{};
\end{tikzpicture}

In this case $\mu_1$ elements of the rightmost column are marked. We have to show, that only the largest $\mu_1-1$ remain in the rightmost column.

As the middle column has odd length $b-1$ an element has to change column. The largest entry from the leftmost column not in the tail is smaller than $b-1$ which is the largest entry in the middle column because due to the residuum $a>0$. Thus this cannot change column. The largest entry in the middle column cannot change column either, because this would give a tail, that is longer than $\mu_1$. Thus the smallest element in the tail, that is the smallest one that was shifted from the middle column to the left one, changes column. This is what we wanted to show.

\item In Case 4, $L$ is not changed at all. 
\end{enumerate}

Now we show that applying first Algorithm~\ref{UalgoLR3} on a pair $(\tilde{Q},\mu)$ and then Algorithm~\ref{algoLR3} on the outcome returns the original pair. Therefore we have to show that we add the same positions to $Q$ as we deleted before to obtain $\tilde{Q}$. As they are filled from left to right increasingly and $\mu$ is encoded in the length of the tail, this is sufficient.

Throughout this proof we arrange $L$ like we did in the first part of the proof, namely as three columns.

The cells that are shifted to the tail in Algorithm~\ref{UalgoLR3} are exactly those that change column in Algorithm~\ref{algoLR3}, except for the case that we have to \enquote{fix parity} in Algorithm~\ref{UalgoLR3}.

In the former case exactly the same cells labeled $x$ occur after Algorithm~\ref{algoLR3}. One can prove this by the same case study as above. Once again we consider one case exemplarily: $\tilde{Q}$ has rows of even length, there is one cell labeled $x$ in the rightmost column, there is an even number $e$ of cells labeled $x$ in the middle column and there is an odd number $o$ of cells labeled $x$ in the leftmost column. Thus $e+1$ cells are marked in the rightmost column, $o-1$ cells are marked in the middle column and one position is marked in the leftmost column in Algorithm~\ref{UalgoLR3}. Therefore, in Algorithm~\ref{algoLR3} there are $e+1$ cells labeled $x$ inserted in the middle column, and $o-1+1$ inserted in the rightmost one. Thus when \enquote{adjust the parity} in Algorithm~\ref{algoLR3} one cell labeled $x$ is shifted to the rightmost column.

In the latter case, that there is a \enquote{fix parity} in Algorithm~\ref{UalgoLR3} we have two cases.

If an element changes column that belonged to the leftmost column or the middle one, this does not change where the cells labeled $x$ are added as those are added in both cases into the leftmost column.

If an element from the leftmost column is moved to the middle one that belonged to the rightmost column, this does not change anything, as this is still an element that is in the wrong column. Thus cells labeled $x$ are put into the middle column and one of them is shifted right, if the parity demands it.
\end{proof}

\subsection{The second Bijection}
\label{Sub2ndAlgo}

In the second part of our bijection, we map three-rowed standard Young tableaux, all of whose row lengths have the same parity, to vacillating tableaux of shape $\emptyset$. Moreover, this bijection should preserve descents. This property restricts our choice, however does not define it by far. In Figure~\ref{FigAllTabs} in the appendix we see standard Young tableaux and the corresponding vacillating tableaux up to a size of $r=6$.

\begin{table}
\centering
\caption{Notation for Algorithms~\ref{algo3dim} and~\ref{Ualgo3dim}}
\label{TabNotation}
\begin{tabular}{|p{6.5cm}|p{9cm}|}
\hline
Labeled word in $\{\pm 1,0\}$ & A word in $\{\pm 1,0\}$. Each letter of this word is labeled by an integer $1\leq i \leq r$ strictly increasing from left to right. Thus each position consists of a label and an entry.\\
\hline
A position in a labeled word is \enquote{on level $\ell$} & The minimum of the sum over all entries strictly to the left and the one over all entries to the left including the one of this position is $\ell$.

Illustration of positions on level $\ell$:
\begin{tikzpicture}[scale=0.35]
\draw[dotted] (1,0)--(6,0);
\draw (0.5,0) node{$\ell$};
\draw (1.5,1)--(2.5,0);
\draw (3,0)--(4,1);
\draw (4.5,0) -- (5.5,0);
\end{tikzpicture}\\
\hline
A $0$ is in \enquote{3-row-position} & This $0$ is either on level $2$ or higher or it is the rightmost $0$ of an odd sequence of $0$'s on level $1$.\\
\hline
\multicolumn{2}{p{15.5cm}}{We refer to positions by their labels and take instructions literally:}\\
\hline
 \enquote{Insert $e$ at $m$} into a given labeled word, which does not contain the label $m$ & We insert a new position with entry $e$ and label $m$ such that the labels are still sorted.\\
\hline
\enquote{The position to the right  (or left) of $m$} in a labeled word, where $m$ does not occur & The position of the right (or left) of the spot where we would insert $m$.\\
\hline
\multicolumn{2}{p{15.5cm}}{We refer to processes in the algorithms in the following way:}\\
\hline
\enquote{Connect / Separate} & \enquote{Connect} means that a sequence $-1,1$ on level $0$ is changed into a sequence $0,0$ on level $1$:
\begin{tikzpicture}[scale=0.35]
\draw[dotted] (0.5,1)--(1,1);
\draw (1,1)--(2,0)--(3,1);
\draw[dotted] (3,1) -- (3.5,1);
\draw (1,0.25) node {};
\end{tikzpicture}
\begin{tikzpicture}[scale=0.35]
 \draw [->,decorate,
decoration={snake,amplitude=.4mm,segment length=2mm,post length=1mm, pre length=1mm}] (0,3) -- (2,3);
\end{tikzpicture}
\begin{tikzpicture}[scale=0.35]
\draw[dotted] (0.5,1)--(1,1);
\draw (1,1)--(2,1)--(3,1);
\draw[dotted] (3,1) -- (3.5,1);
\draw (1,0.25) node {};
\end{tikzpicture}

\enquote{Separate} means that a sequence $0,0$ on level $1$  is changed into a sequence $-1,1$ on level $0$.

This nomenclature is justified by Theorem~\ref{TheoConcat}. \\
\hline
\end{tabular}
\end{table}

\begin{algorithm}
\label{algo3dim}
\SetKwInOut{Input}{input}\SetKwInOut{Output}{output}
 \Input{standard Young tableau $Q$ with at most $3$ rows, all of them of even length}
 \Output{vacillating tableau $V$ of shape $\emptyset$ of even length}
 \tcc{Inserting $1^{\text{st}}$ row}
 construct word $V=(1,-1,1,-1,\dots,1,-1)$ with same length as the first row of $Q$\;
 label the letters with the numbers of the first row of $Q$\;
  \tcc{Inserting $2^{\text{nd}}$ row}
 \For{pairs $a,b$ in the second row of $Q$, starting with the rightmost pair, going left}{
  insert $-1$ at $b$\tcc*{insert $b$}
  \eIf(\tcc*[f]{insert $a$ case 1}){position to the right of $a$ is $-1$ not labeled with $b$}
  {
   insert $0$ at $a$; change $-1$ to the right of $a$ into $0$\;
   }(\tcc*[f]{insert $a$ case 2})
   {
   insert $-1$ at $a$; change next $-1$ to the left of $a$ into $1$\;
  }
  change pairs of $1,-1$ between $a$ and $b$ into $0,0$\;
 }
   \tcc{Inserting $3^{\text{rd}}$ row}
  \For{pairs $a,b$ in the third row of $Q$, starting with the rightmost pair, going left}{
    insert $-1$ at $b$; change next $-1$ to the left of $b$ into $0$, let $\tilde{b}$ be its label\tcc*{insert $b$}
    start at $b$ and let $c$ be the current position, let $C$ be an empty list\;
    \While{$\tilde{a}$ undefined}
    {
    \uIf(\tcc*[f]{connect}){$c$ is $1$ on level $0$}
    {  
    change $c$ and the position to the left into $0$, put their labels into $C$\;
    }
    \ElseIf(\tcc*[f]{separate}){the number of positions strictly to the left is even, $c$ and the position to the left are $0$'s on level $1$ and $a$ is not inserted}
    {
    change $c$ into $1$ and the position to the left into $-1$\;
    }
    \uIf(\tcc*[f]{insert $a$}){$a$ is directly to the left of $c$}
    {
    insert $-1$ at $a$\;
    }
    \ElseIf{$a$ is inserted}
    {
    \lIf(\tcc*[f]{case 1}){ $c$ is $-1$ not labeled with $a$}
    {change $c$ into $0$, let $\tilde{a}$ be its label}
    \lElseIf(\tcc*[f]{case 2}){$c$ is $0$, not in $C$, left of $\tilde{b}$}
    {change $c$ into $1$, let $\tilde{a}$ be its label}
    }
    go one position to the left\;
    }
  }
  forget the labeling of $V$ and \Return $V$\;
 \caption{standard Young tableau to vacillating tableau}
\end{algorithm}

\begin{algorithm}
\label{Ualgo3dim}
\SetKwInOut{Input}{input}\SetKwInOut{Output}{output}
 \Input{vacillating tableau $V$ of shape $\emptyset$ of even length $r$}
 \Output{standard Young tableau $Q$ with at most 3 rows, all of them of even length}
 attach labels $\{1,2,\dots, r\}$ to $V$\;
   \tcc{Extracting $3^{\text{rd}}$ row}
  \While{there are $0$'s in 3-row-position}
  {
  let $p$ be the leftmost $0$ in 3-row-position and $p'$ be the position to its right\tcc*{extract $a$}
  \lIf(\tcc*[f]{case 2}){$p$ is on level $1$ and $p'$ is a $1$}
    {change $p'$ into $0$, let $\tilde{a}$ be $p'$}
  \lElse(\tcc*[f]{case 1}){change $p$ into $-1$, let $\tilde{a}$ be $p$}
  let $a$ be the next $-1$ to the right be $a$, delete $a$\;
  start at $\tilde{a}$, let the current position be $c$\;
  \While{$\tilde{b}$ undefined}
  {
  \uIf(\tcc*[f]{undo separate}){$c$ is $-1$ on level $0$}
  {change $c$ and the position to the right into $0,0$\;}
  \ElseIf(\tcc*[f]{undo connect}){the number of positions strictly to the left is odd, both $c$ and the position to the right are $0$ on level $1$}
  {change $c$ and the position to the right into $-1,1$\;}
  \If(\tcc*[f]{extract $b$}){$c$ is $0$ in 3-row-position}
  {change $c$ into $-1$, let $\tilde{b}$ be $c$; let $b$ be the next $-1$ to the right, delete $b$\;
  } 
  go one position to the right\;}
  insert $a,b$ into the third row of $Q$\;
  }

\tcc{Extracting $2^{\text{nd}}$ row}

\While{ $V\neq(1,-1,1,-1,\dots, 1,-1)$}
{
let $a$ be the leftmost position that is neither a $1$ nor on level $0$\tcc*{extract $a$}
\lIf(\tcc*[f]{case 2}){$a$ is $-1$}{change the leftmost $1$ on level $1$ into $-1$}
\lElse(\tcc*[f]{case 1}){change $0$ to the right of $a$ into $-1$}
delete $a$; let $b$ be the next $-1$ to the right of the new $-1$, delete $b$\tcc*{extract $b$}
change $0,0$ on level $0$ into $1,-1$\;
insert $a,b$ into the second row of $Q$\;
}

 \tcc{Extracting $1^{\text{st}}$ row}
 insert labels still in the word into the first row of $Q$\;

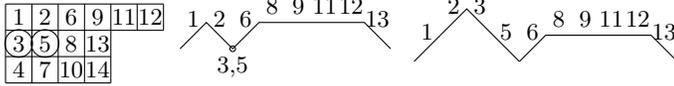
\Return $Q$\;

 \caption{vacillating tableau to standard Young tableau}
\end{algorithm}

Algorithm~\ref{algo3dim} maps standard Young tableaux with three rows of even length to vacillating tableaux of even length and shape $\emptyset$. As we saw in Section~\ref{bijection}, this is sufficient because for tableaux with rows of odd length we add three additional entries to the standard Young tableau.

We illustrate labeled words with labeled Riordan paths like we illustrate vacillating tableaux as Riordan paths. See Figure~\ref{FigScenarios} in the appendix for a set of possible insertion scenarios.

We give now a full iteration of Algorithm~\ref{algo3dim} on our running example.

\begin{example}
We start with inserting the first row by forming a word $1,-1,1,-1,\dots,1,-1$ labeled with the elements of the first row of the tableau:

\begin{tikzpicture}[scale=0.35]
  \draw (0,4) -- (6,4);
  \draw (0,3) -- (6,3);
  \draw (0,2) -- (4,2);
  \draw (0,1) -- (4,1);
  \draw (0,4) -- (0,1);
  \draw (1,4) -- (1,1);
  \draw (2,4) -- (2,1);
  \draw (3,4) -- (3,1);
  \draw (4,4) -- (4,1);
  \draw (5,4) -- (5,3);
  \draw (6,4) -- (6,3);
    \draw (0.5,3.5) node[draw, circle, scale=1.2] {};
  \draw (0.5,3.5) node {1};
    \draw (1.5,3.5) node[draw, circle, scale=1.2] {};  
  \draw (1.5,3.5) node {2};
  \draw (0.5,2.5) node {3};
  \draw (0.5,1.5) node {4};
  \draw (1.5,2.5) node {5};
    \draw (2.5,3.5) node[draw, circle, scale=1.2] {};
  \draw (2.5,3.5) node {6};
  \draw (1.5,1.5) node {7};
  \draw (2.5,2.5) node {8};
    \draw (3.5,3.5) node[draw, circle, scale=1.2] {};
  \draw (3.5,3.5) node {9};
  \draw (2.5,1.5) node {10};
    \draw (4.5,3.5) node[draw, circle, scale=1.2] {};
  \draw (4.5,3.5) node {11};
    \draw (5.5,3.5) node[draw, circle, scale=1.2] {};
  \draw (5.5,3.5) node {12};
  \draw (3.5,2.5) node {13};
  \draw (3.5,1.5) node {14};
\end{tikzpicture}
\begin{tikzpicture}[scale=0.35]
\draw (0,0)  -- (1,1) node [midway, above] {1} --
(2,0) node[midway, above] {2} --
(3,1) node[midway, above] {6} --
(4,0) node[midway, above] {9} --
(5,1) node[midway, above] {11} --
(6,0) node[midway, above] {12};
\end{tikzpicture}

Inserting the second row we begin with case 1 at inserting $a$ and changing the pair of $1,-1$ at $11,12$ into $0,0$:

\hspace*{-0.1cm}\begin{tikzpicture}[scale=0.35]
  \draw (0,4) -- (6,4);
  \draw (0,3) -- (6,3);
  \draw (0,2) -- (4,2);
  \draw (0,1) -- (4,1);
  \draw (0,4) -- (0,1);
  \draw (1,4) -- (1,1);
  \draw (2,4) -- (2,1);
  \draw (3,4) -- (3,1);
  \draw (4,4) -- (4,1);
  \draw (5,4) -- (5,3);
  \draw (6,4) -- (6,3);
  \draw (0.5,3.5) node {1};
  \draw (1.5,3.5) node {2};
  \draw (0.5,2.5) node {3};
  \draw (0.5,1.5) node {4};
  \draw (1.5,2.5) node {5};
  \draw (2.5,3.5) node {6};
  \draw (1.5,1.5) node {7};
      \draw (2.5,2.5) node[draw, circle, scale=1.2] {};
  \draw (2.5,2.5) node {8};
  \draw (3.5,3.5) node {9};
  \draw (2.5,1.5) node {10};
  \draw (4.5,3.5) node {11};
  \draw (5.5,3.5) node {12};
      \draw (3.5,2.5) node[draw, circle, scale=1.2] {};
  \draw (3.5,2.5) node {13};
  \draw (3.5,1.5) node {14};
\end{tikzpicture}
\begin{tikzpicture}[scale=0.35]
\draw (0,0)  -- (1,1) node [midway, above] {1} --
(2,0) node[midway, above] {2} --
(3,1) node[midway, above] {6};
\draw(3,1)  circle (3pt) node[above] {8};
\draw (3,1) -- (4,0) node[midway, above] {9} --
(5,1) node[midway, above] {11} --
(6,0) node[midway, above] {12};
\draw (6,0) circle (3pt) node[below] {13};
\draw (0,-0.6) node {};
\end{tikzpicture}
\begin{tikzpicture}[scale=0.35]
\draw (0,0)  -- (1,1) node [midway, above] {1} --
(2,0) node[midway, above] {2} --
(3,1) node[midway, above] {6};
\draw (3,1)--
(4,1) node[midway, above] {8} --
(5,1) node[midway, above] {9} --
(6,1) node[midway, above] {11} --
(7,1) node[midway, above] {12} --
(8,0) node[midway, above] {13};
\draw (0,-0.6) node {};
\end{tikzpicture}\\
And end with case 2 at inserting $a$:

\hspace*{-0.1cm}\begin{tikzpicture}[scale=0.35]
  \draw (0,4) -- (6,4);
  \draw (0,3) -- (6,3);
  \draw (0,2) -- (4,2);
  \draw (0,1) -- (4,1);
  \draw (0,4) -- (0,1);
  \draw (1,4) -- (1,1);
  \draw (2,4) -- (2,1);
  \draw (3,4) -- (3,1);
  \draw (4,4) -- (4,1);
  \draw (5,4) -- (5,3);
  \draw (6,4) -- (6,3);
  \draw (0.5,3.5) node {1};
  \draw (1.5,3.5) node {2};
        \draw (0.5,2.5) node[draw, circle, scale=1.2] {};
  \draw (0.5,2.5) node {3};
  \draw (0.5,1.5) node {4};
        \draw (1.5,2.5) node[draw, circle, scale=1.2] {};
  \draw (1.5,2.5) node {5};
  \draw (2.5,3.5) node {6};
  \draw (1.5,1.5) node {7};
  \draw (2.5,2.5) node {8};
  \draw (3.5,3.5) node {9};
  \draw (2.5,1.5) node {10};
  \draw (4.5,3.5) node {11};
  \draw (5.5,3.5) node {12};
  \draw (3.5,2.5) node {13};
  \draw (3.5,1.5) node {14};
\end{tikzpicture}
\begin{tikzpicture}[scale=0.35]
\draw (0,0) -- (1,1) node [midway, above] {1}
 -- (2,0) node[midway, above] {2};
\draw (2,0) circle (3pt) node[below] {3,5};
\draw (2,0) -- (3,1) node[midway, above] {6} --
(4,1) node[midway, above] {8} --
(5,1) node[midway, above] {9} --
(6,1) node[midway, above] {11} --
(7,1) node[midway, above] {12} --
(8,0) node[midway, above] {13};
\end{tikzpicture}
\begin{tikzpicture}[scale=0.35]
\draw (0,0)  -- (1,1) node [midway, above] {1};
\draw(1,1) --
(2,2) node[midway, above] {2} --
(3,1) node[midway, above] {3} --
(4,0) node[midway, above] {5};
\draw (4,0)--
(5,1) node[midway, above] {6} --
(6,1) node[midway, above] {8} --
(7,1) node[midway, above] {9} --
(8,1) node[midway, above] {11} --
(9,1) node[midway, above] {12} --
(10,0) node[midway, above] {13};
    \draw (0,-0.6) node {};
\end{tikzpicture}\\
We obtained a word, where the only $0$'s are in pairs on level one. We will see that this is always the case after any insertion that comes from the second row (see Lemma~\ref{Lemma0s}).

Inserting the third row we start with at \enquote{insert $a$ case 2}  and \enquote{separate}:

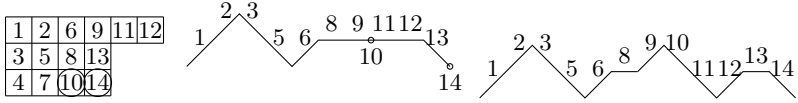
\hspace*{-0.1cm}\begin{tikzpicture}[scale=0.35]
  \draw (0,4) -- (6,4);
  \draw (0,3) -- (6,3);
  \draw (0,2) -- (4,2);
  \draw (0,1) -- (4,1);
  \draw (0,4) -- (0,1);
  \draw (1,4) -- (1,1);
  \draw (2,4) -- (2,1);
  \draw (3,4) -- (3,1);
  \draw (4,4) -- (4,1);
  \draw (5,4) -- (5,3);
  \draw (6,4) -- (6,3);
  \draw (0.5,3.5) node {1};
  \draw (1.5,3.5) node {2};
  \draw (0.5,2.5) node {3};
  \draw (0.5,1.5) node {4};
  \draw (1.5,2.5) node {5};
  \draw (2.5,3.5) node {6};
  \draw (1.5,1.5) node {7};
  \draw (2.5,2.5) node {8};
  \draw (3.5,3.5) node {9};
          \draw (2.5,1.5) node[draw, circle, scale=1.2] {};
  \draw (2.5,1.5) node {10};
  \draw (4.5,3.5) node {11};
  \draw (5.5,3.5) node {12};
  \draw (3.5,2.5) node {13};
          \draw (3.5,1.5) node[draw, circle, scale=1.2] {};
  \draw (3.5,1.5) node {14};
\end{tikzpicture}
\begin{tikzpicture}[scale=0.35]
\draw (0,0)  -- (1,1) node [midway, above] {1} --
(2,2) node[midway, above] {2} --
(3,1) node[midway, above] {3} --
(4,0) node[midway, above] {5} --
(5,1) node[midway, above] {6} --
(6,1) node[midway, above] {8} --
(7,1) node[midway, above] {9};
\draw (7,1) circle (3pt) node[below] {10};
\draw (7,1) -- (8,1) node[midway, above] {11} --
(9,1) node[midway, above] {12} --
(10,0) node[midway, above] {13};
\draw (10,0) circle (3pt) node[below] {14};
  \draw (0,-0.6) node {};
\end{tikzpicture}
\begin{tikzpicture}[scale=0.35]
\draw (0,0)  -- (1,1) node [midway, above] {1} --
(2,2) node[midway, above] {2} --
(3,1) node[midway, above] {3} --
(4,0) node[midway, above] {5} --
(5,1) node[midway, above] {6} --
(6,1) node[midway, above] {8};
\draw (6,1) --
(7,2) node[midway, above] {9} --
(8,1) node[midway, above] {10} --
(9,0) node[midway, above] {11} --
(10,1) node[midway, above] {12} --
(11,1) node[midway, above] {13} --
(12,0) node[midway, above] {14};
\end{tikzpicture}\\
We finalize with \enquote{inserting $a$ case 1}. We do not \enquote{connect} as $\tilde b$ and $b$ already connected the paths:

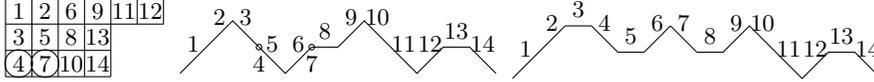
\hspace*{-0.1cm}\begin{tikzpicture}[scale=0.35]
  \draw (0,4) -- (6,4);
  \draw (0,3) -- (6,3);
  \draw (0,2) -- (4,2);
  \draw (0,1) -- (4,1);
  \draw (0,4) -- (0,1);
  \draw (1,4) -- (1,1);
  \draw (2,4) -- (2,1);
  \draw (3,4) -- (3,1);
  \draw (4,4) -- (4,1);
  \draw (5,4) -- (5,3);
  \draw (6,4) -- (6,3);
  \draw (0.5,3.5) node {1};
  \draw (1.5,3.5) node {2};
  \draw (0.5,2.5) node {3};
            \draw (0.5,1.5) node[draw, circle, scale=1.2] {};
  \draw (0.5,1.5) node {4};
  \draw (1.5,2.5) node {5};
  \draw (2.5,3.5) node {6};
            \draw (1.5,1.5) node[draw, circle, scale=1.2] {};
  \draw (1.5,1.5) node {7};
  \draw (2.5,2.5) node {8};
  \draw (3.5,3.5) node {9};
  \draw (2.5,1.5) node {10};
  \draw (4.5,3.5) node {11};
  \draw (5.5,3.5) node {12};
  \draw (3.5,2.5) node {13};
  \draw (3.5,1.5) node {14};
\end{tikzpicture}
\begin{tikzpicture}[scale=0.35]
\draw (0,0)  -- (1,1) node [midway, above] {1} --
(2,2) node[midway, above] {2} --
(3,1) node[midway, above] {3} ;
\draw (3,1) circle (3pt) node[below] {4};
\draw (3,1) -- (4,0) node[midway, above] {5} --
(5,1) node[midway, above] {6};
\draw (5,1) circle (3pt) node[below] {7};
\draw (5,1) -- (6,1) node[midway, above] {8} --
(7,2) node[midway, above] {9} --
(8,1) node[midway, above] {10} --
(9,0) node[midway, above] {11} --
(10,1) node[midway, above] {12} --
(11,1) node[midway, above] {13} --
(12,0) node[midway, above] {14};
\end{tikzpicture}
\begin{tikzpicture}[scale=0.35]
\draw (0,0)  -- (1,1) node [midway, above] {1} --
(2,2) node[midway, above] {2};
\draw(2,2)-- (3,2) node[midway, above] {3} --
(4,1) node[midway, above] {4} --
(5,1) node[midway, above] {5};
\draw (5,1)-- (6,2) node[midway, above] {6};
\draw (6,2) -- (7,1) node[midway, above] {7};
\draw (7,1) -- (8,1) node[midway, above] {8} --
(9,2) node[midway, above] {9} --
(10,1) node[midway, above] {10} --
(11,0) node[midway, above] {11} --
(12,1) node[midway, above] {12} --
(13,1) node[midway, above] {13} --
(14,0) node[midway, above] {14};
\end{tikzpicture}\\
We see that after each insertion from the third row we obtain a new $0$  in 3-row-position (that is either on level two or higher or the rightmost one in an odd sequence of $0$'s on level one). This is indeed always the case, see Lemma~\ref{Lemma0s}.

In the end we forget the labels. So what we obtain is the following path:

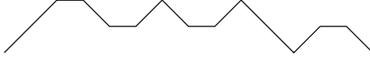
\begin{tikzpicture}[scale=0.35]
\draw (0,0)  -- (1,1) -- (2,2) -- (3,2) -- (4,1) -- (5,1) -- (6,2) -- (7,1) -- (8,1) -- (9,2) -- (10,1) -- (11,0) -- (12,1) -- (13,1) -- (14,0);
\end{tikzpicture}
\end{example}

As we will see, everything in Algorithm~\ref{algo3dim} is reversible, and Algorithm~\ref{Ualgo3dim} is its inverse. To illustrate the latter we use our running example and give a full iteration of it.
\begin{example}
We start searching for $0$ entries in 3-row-position from left to right. The $-1$ entries to the right are our first $a$ and $b$. This time we do not need to \enquote{undo connecting} as this is \enquote{done automatically}. We extract $a$ and $b$ and put them into the third row of our tableau.

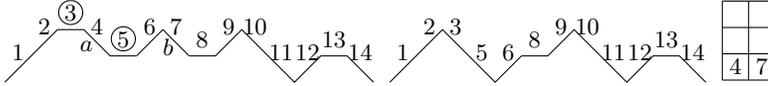
\begin{tikzpicture}[scale=0.35]
\draw (0,0)  -- (1,1) node [midway, above] {1} --
(2,2) node[midway, above] {2};
\draw(2,2)-- (3,2) node[midway, above] {3} --
(4,1) node[midway, above] {4} --
(5,1) node[midway, above] {5};
\draw (5,1)-- (6,2) node[midway, above] {6};
\draw (6,2) -- (7,1) node[midway, above] {7};
\draw (7,1) -- (8,1) node[midway, above] {8} --
(9,2) node[midway, above] {9} --
(10,1) node[midway, above] {10} --
(11,0) node[midway, above] {11} --
(12,1) node[midway, above] {12} --
(13,1) node[midway, above] {13} --
(14,0) node[midway, above] {14};
  \draw (2.5,2.65) circle (13pt) node{};
  \draw (3.1,1.9) node[below]{$a$};
  \draw (4.5,1.65) circle (13pt) node{};
  \draw (6.2,2) node[below]{$b$};
\end{tikzpicture}
\begin{tikzpicture}[scale=0.35]
\draw (0,0)  -- (1,1) node [midway, above] {1} --
(2,2) node[midway, above] {2} --
(3,1) node[midway, above] {3} ;
\draw (3,1) -- (4,0) node[midway, above] {5} --
(5,1) node[midway, above] {6};
\draw (5,1) -- (6,1) node[midway, above] {8} --
(7,2) node[midway, above] {9} --
(8,1) node[midway, above] {10} --
(9,0) node[midway, above] {11} --
(10,1) node[midway, above] {12} --
(11,1) node[midway, above] {13} --
(12,0) node[midway, above] {14};
\end{tikzpicture}
\begin{tikzpicture}[scale=0.35]
  \draw (0,4) -- (2,4);
  \draw (0,3) -- (2,3);
  \draw (0,2) -- (2,2);
  \draw (0,1) -- (2,1);
  \draw (0,4) -- (0,1);
  \draw (1,4) -- (1,1);
  \draw (2,4) -- (2,1);
  \draw (0.5,1.5) node {4};
  \draw (1.5,1.5) node {7};
\end{tikzpicture}\\
We do the same again. However, this time we change the $1$ right of the first $0$ into a $0$ and we find a separating point between $a$ and $b$ so we \enquote{undo separating}:

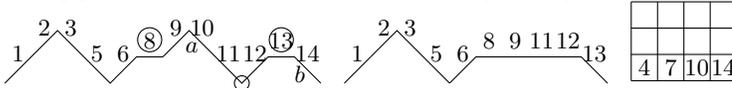
\begin{tikzpicture}[scale=0.35]
\draw (0,0)  -- (1,1) node [midway, above] {1} --
(2,2) node[midway, above] {2} --
(3,1) node[midway, above] {3} --
(4,0) node[midway, above] {5} --
(5,1) node[midway, above] {6} --
(6,1) node[midway, above] {8};
\draw (6,1) --
(7,2) node[midway, above] {9} --
(8,1) node[midway, above] {10} --
(9,0) node[midway, above] {11} --
(10,1) node[midway, above] {12} --
(11,1) node[midway, above] {13} --
(12,0) node[midway, above] {14};
  \draw (0,-0.2) node {};
  \draw (5.5,1.65) circle (13pt) node{};
  \draw (7.1,1.9) node[below]{$a$};
  \draw (10.5,1.65) circle (13pt) node{};
  \draw (11.2,1) node[below]{$b$};
  \draw (9,0) circle (8pt) node{};
\end{tikzpicture}
\begin{tikzpicture}[scale=0.35]
\draw (0,0)  -- (1,1) node [midway, above] {1};
\draw(1,1) --
(2,2) node[midway, above] {2} --
(3,1) node[midway, above] {3} --
(4,0) node[midway, above] {5};
\draw (4,0)--
(5,1) node[midway, above] {6} --
(6,1) node[midway, above] {8} --
(7,1) node[midway, above] {9} --
(8,1) node[midway, above] {11} --
(9,1) node[midway, above] {12} --
(10,0) node[midway, above] {13};
    \draw (0,-0.2) node {};
\end{tikzpicture}
\begin{tikzpicture}[scale=0.35]
  \draw (0,4) -- (4,4);
  \draw (0,3) -- (4,3);
  \draw (0,2) -- (4,2);
  \draw (0,1) -- (4,1);
  \draw (0,4) -- (0,1);
  \draw (1,4) -- (1,1);
  \draw (2,4) -- (2,1);
  \draw (3,4) -- (3,1);
  \draw (4,4) -- (4,1);
  \draw (0.5,1.5) node {4};
  \draw (1.5,1.5) node {7};
  \draw (2.5,1.5) node {10};
  \draw (3.5,1.5) node {14};
  \draw (0,0.7) node {};
\end{tikzpicture}\\
Now we do not find $0$'s in 3-row-position anymore. Thus we proceed with extracting the second row. Therefore we search for $-1$'s or $0$'s:

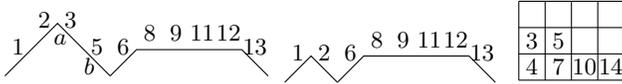
\begin{tikzpicture}[scale=0.35]
\draw (0,0)  -- (1,1) node [midway, above] {1};
\draw(1,1) --
(2,2) node[midway, above] {2} --
(3,1) node[midway, above] {3} --
(4,0) node[midway, above] {5};
\draw (4,0)--
(5,1) node[midway, above] {6} --
(6,1) node[midway, above] {8} --
(7,1) node[midway, above] {9} --
(8,1) node[midway, above] {11} --
(9,1) node[midway, above] {12} --
(10,0) node[midway, above] {13};
  \draw (2.1,1.9) node[below]{$a$};
  \draw (3.2,1) node[below]{$b$};
\end{tikzpicture}
\begin{tikzpicture}[scale=0.35]
\draw (0,0) -- (1,1) node [midway, above] {1}
 -- (2,0) node[midway, above] {2};
\draw (2,0) -- (3,1) node[midway, above] {6} --
(4,1) node[midway, above] {8} --
(5,1) node[midway, above] {9} --
(6,1) node[midway, above] {11} --
(7,1) node[midway, above] {12} --
(8,0) node[midway, above] {13};
\end{tikzpicture}
\begin{tikzpicture}[scale=0.35]
  \draw (0,4) -- (4,4);
  \draw (0,3) -- (4,3);
  \draw (0,2) -- (4,2);
  \draw (0,1) -- (4,1);
  \draw (0,4) -- (0,1);
  \draw (1,4) -- (1,1);
  \draw (2,4) -- (2,1);
  \draw (3,4) -- (3,1);
  \draw (4,4) -- (4,1);
  \draw (0.5,2.5) node {3};
  \draw (0.5,1.5) node {4};
  \draw (1.5,2.5) node {5};
  \draw (1.5,1.5) node {7};
  \draw (2.5,1.5) node {10};
  \draw (3.5,1.5) node {14};
  \draw (0,1.2) node {};
\end{tikzpicture}\\
We do it again, this time the other case occurs. Moreover, we change the $0,0$ at $11,12$ into $1,-1$:

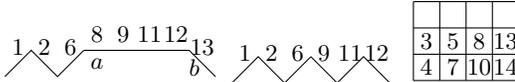
\begin{tikzpicture}[scale=0.35]
\draw (0,0)  -- (1,1) node [midway, above] {1} --
(2,0) node[midway, above] {2} --
(3,1) node[midway, above] {6};
\draw (3,1)--
(4,1) node[midway, above] {8} --
(5,1) node[midway, above] {9} --
(6,1) node[midway, above] {11} --
(7,1) node[midway, above] {12} --
(8,0) node[midway, above] {13};
  \draw (3.5,1) node[below]{$a$};
  \draw (7.2,1) node[below]{$b$};
\end{tikzpicture}
\begin{tikzpicture}[scale=0.35]
\draw (0,0)  -- (1,1) node [midway, above] {1} --
(2,0) node[midway, above] {2} --
(3,1) node[midway, above] {6};
\draw (3,1) -- (4,0) node[midway, above] {9} --
(5,1) node[midway, above] {11} --
(6,0) node[midway, above] {12};
\end{tikzpicture}
\begin{tikzpicture}[scale=0.35]
  \draw (0,4) -- (4,4);
  \draw (0,3) -- (4,3);
  \draw (0,2) -- (4,2);
  \draw (0,1) -- (4,1);
  \draw (0,4) -- (0,1);
  \draw (1,4) -- (1,1);
  \draw (2,4) -- (2,1);
  \draw (3,4) -- (3,1);
  \draw (4,4) -- (4,1);
  \draw (0.5,2.5) node {3};
  \draw (0.5,1.5) node {4};
  \draw (1.5,2.5) node {5};
  \draw (1.5,1.5) node {7};
  \draw (2.5,2.5) node {8};
  \draw (2.5,1.5) node {10};
  \draw (3.5,2.5) node {13};
  \draw (3.5,1.5) node {14};
  \draw (0,1.2) node {};
\end{tikzpicture}\\
To finish with Algorithm~\ref{Ualgo3dim} we fill the first row with the remaining numbers: 
\[\begin{tikzpicture}[scale=0.35]
  \draw (0,4) -- (6,4);
  \draw (0,3) -- (6,3);
  \draw (0,2) -- (4,2);
  \draw (0,1) -- (4,1);
  \draw (0,4) -- (0,1);
  \draw (1,4) -- (1,1);
  \draw (2,4) -- (2,1);
  \draw (3,4) -- (3,1);
  \draw (4,4) -- (4,1);
  \draw (5,4) -- (5,3);
  \draw (6,4) -- (6,3);
  \draw (0.5,3.5) node {1};
  \draw (1.5,3.5) node {2};
  \draw (0.5,2.5) node {3};
  \draw (0.5,1.5) node {4};
  \draw (1.5,2.5) node {5};
  \draw (2.5,3.5) node {6};
  \draw (1.5,1.5) node {7};
  \draw (2.5,2.5) node {8};
  \draw (3.5,3.5) node {9};
  \draw (2.5,1.5) node {10};
  \draw (4.5,3.5) node {11};
  \draw (5.5,3.5) node {12};
  \draw (3.5,2.5) node {13};
  \draw (3.5,1.5) node {14};
\end{tikzpicture}\]
We see that we obtained our original standard Young tableau.
\end{example}

We proceed by formulating and proving properties of Algorithm~\ref{algo3dim} and~\ref{Ualgo3dim}.

\begin{remark}
A (partial) standard Young tableau filled with elements of a totally ordered set produces the same vacillating tableau in Algorithm~\ref{algo3dim} as the standard Young tableau filled by $1,2,\dots,r$ increasingly according to this totally ordered set.
Thus the crucial information in the standard Young tableau is the relative order of the entries.

From this point of view, one can stop after any insertion of a pair $a,b$ and obtains the vacillating tableau corresponding to the tableau only consisting of already inserted entries.
\end{remark}

\begin{lemma}
Algorithms~\ref{algo3dim} and~\ref{Ualgo3dim} are well-defined.
\end{lemma}

\begin{proof}
To show that Algorithm~\ref{algo3dim} is well-defined we need to show that a $-1$ exists left of $a$ when inserting the second row in case 2 and that $\tilde{b}$ and $\tilde{a}$ exist when inserting the third row.

For inserting the second row consider the numbers directly above $a$ and $b$ in the tableau; call them $c$ and $d$. We will see inductively that $a$ and $b$ may change the positions labeled with $c$ and $d$ but nothing to the left, thus nothing with smaller labels. There are two cases: $c<a<d<b$ (this corresponds to \enquote{insert $a$ case 1}) and $c<d<a<b$ (this can happen in both cases). In both cases we know that $c$ and $d$ were not changed before. So they are still $1,-1$ and can be used.

Analogously, for inserting the third row we consider the numbers directly above $a$ and $b$. The entry above $b$ has been inserted as $-1$. The entry above $a$ either as $0$ or as $-1$.

To show that Algorithm~\ref{Ualgo3dim} is well-defined we need to show that $\tilde{b}$ exists in extracting the third row provided that there is a $0$ in 3-row-position. This holds because after deleting $a$ and changing $\tilde{a}$ there are an odd number of $0$'s left as the vacillating tableau needs to have even length.
\end{proof}

Next we show that our algorithms produce the desired output.

\begin{lemma}
Algorithm~\ref{algo3dim} produces a vacillating tableau of even length and shape $\emptyset$.
\end{lemma}

\begin{proof}
We show this inductively. After inserting the first row we obtain $(1,-1,1,-1,\dots,1,-1)$ which is a vacillating tableau. Moreover, after every insertion of a pair $(a,b)$, the sum over all entries is zero. Thus it remains to show that the algorithm neither produces any $0$ at level zero or lower nor any $\pm 1$ at level minus one or lower. We show that this is the case after every insertion of a pair $(a,b)$.

The only steps in the algorithm at which the level of a position decreases occur when pairs $(1,-1)$ between $a$ and $b$ are changed into $(0,0)$, or when \enquote{separating}. In the first situation, the level is increased between $a$ and $b$, so this does not produce a $0$ on level zero. In the second situation, pairs $(0,0)$ at level one are changed into $(-1,1)$ at level zero.
\end{proof}

\begin{lemma}
Algorithm~\ref{Ualgo3dim} produces a standard Young tableau with at most three rows, all of even length.
\end{lemma}

\begin{proof}
The row lengths are even because elements are extracted in pairs. It remains to show that the row lengths are weakly decreasing and the tableau is standard.

Suppose we extract a pair $(a,b)$ into the third (respectively second) row. We define two related elements $a'<b'$ with $a'<a$ and $b'<b$ which will be extracted into the second (respectively first) row and end up weakly right of $a$ and $b$.

We consider the second row first. We define $b'$ to be the $-1$ created just before deleting $a$. The element $a'$ is the $1$ directly left of $b'$. Thus $(a',b')$ is a pair $(1,-1)$ on level zero, such that everything to the left is also on level zero. Therefore it is extracted into the first row.

Now we consider the third row. Suppose $\tilde{a}$ is $-1$, then we define $a'$ to be $\tilde{a}$. Otherwise, if $\tilde{a}$ is $0$, we distinguish two cases. If there is no $-1$ extracted into the second row left of $\tilde{a}$ or the next $-1$ to the left is extracted into the second row as $b$ then $a'$ is the leftmost $0$ in the sequence of $0$'s of $\tilde{a}$. Otherwise it is the next $-1$ to the left. Similarly, if $\tilde{b}$ is a $-1$ on level one or higher, we define $b'$ to be $\tilde{b}$. Otherwise, if it is a $-1$ on level zero, we distinguish two cases. If there is a $-1$ directly to the left of $b$, that is extracted as $b$ into the third row, we define $b'$ to be this $-1$ and otherwise we define $b$ to be $\tilde{b}$.

Finally we show that $a'$ and $b'$ get extracted into the second row.
Therefore we describe the elements that are extracted into the second row. $-1$'s on level one or higher are always such elements. $-1$'s on level zero are extracted to the second row if and only if there is a $-1$ to the left that is extracted as $a$ or there is a $0$ directly to the left. $0$'s are extracted if and only if they are the leftmost in their sequence of $0$'s and the $-1$ to the left (if it exists) is extracted as $b$.
\end{proof}

The following lemma is one of the key observations to show that Algorithms~\ref{algo3dim} and~\ref{Ualgo3dim} are inverse.

\begin{lemma}
\label{Lemma0s}
In Algorithm~\ref{algo3dim}, after inserting a pair $(a,b)$ from the second row, any $0$'s come in pairs on level one. Any third row insertion produces $0$'s in 3-row-position.
\end{lemma}

\begin{proof}
We consider inserting the second row first. There are two possible ways to produce $0$'s when inserting the second row: at \enquote{inserting $a$ case 1}, and changing pairs $1,-1$ into $0,0$. Both produce them in pairs. Moreover, those are always right of the current $a$. As the next insertion happens strictly left of the next $b$, that is itself left of $a$, they cannot be isolated later by inserting the second row.

Now we consider inserting the third row.
Elements are inserted decreasingly, from right to left. Therefore left of the leftmost $\tilde{a}$ so far there are no $0$'s in 3-row-position. Moreover, \enquote{separate} and \enquote{connect} only changes or produces pairs such that an even number of positions and thus an even number of $0$'s is to the left. Now we consider $\tilde{a}$ and $\tilde{b}$ of a pair to insert. If those are on level two or higher, they are $0$s in 3-row-position. If $\tilde{b}$ is a $0$ on level one, it needs to be between $a$ and $b$. If $a$ is the $-1$ to end the sequence of $0$s of $b$ it could be the case that this sequence is even. In this case, the leftmost $0$ of this sequence is thus at level two. Moreover $\tilde{a}$ is either another $0$ on level two or a $1$ next to an odd number of $0$s on level one, as it deleted one. Now $\tilde{a}$ either produced a new $0$ or it deleted one. In the former case, if this is on level one, it needs to be either a single $0$ or part of a sequence that was there before. Thus, that sequence was even before adding a $0$. In the latter case this deleted a $0$ that was in an even sequence of $0$'s, hence is now a new odd sequence of $0$'s.
\end{proof}

\begin{lemma}
\label{LemsecondInverse}
Algorithms~\ref{algo3dim} and~\ref{Ualgo3dim} are inverse.
\end{lemma}

\begin{proof}
Due to Lemma~\ref{Lemma0s} and the fact that Algorithm~\ref{Ualgo3dim} extracts third row elements only as long as there are $0$'s in 3-row-position, it suffices to show that inserting and extracting the second (respectively third) row are inverse.

This is obvious for the second row. For the third row we point out that due to \enquote{connect} and \enquote{separate} between $a$ and $b$ in the third row a return to level zero between $a$ and $b$ can only be at odd points as we \enquote{connect} at even points and \enquote{separate} at odd ones.
\end{proof}

We continue proving the following important property.

\begin{lemma}
\label{LemDescP}
Algorithm~\ref{algo3dim} is descent-preserving.
\end{lemma}

\begin{proof}
We show that the algorithm preserves descents after every insertion of a pair $a,b$ in the sense that we consider the inserted numbers as a new total order.

In the first step we show that when we insert a pair $(a,b)$, $a$ and $b$ cause a descent except for the case that $a$ and $b$ are neighbors in the set of already inserted numbers. To see why we want this to hold, we consider the partial standard Young tableau consisting only of already inserted numbers. The number smaller than $a$ needs to be in a row below $a$ as numbers in the same row to the right are bigger. The same holds for $b$ except if $a$ and $b$ are neighbors in the current order.

In the case that $a$ and $b$ are neighbors, they are both inserted as $-1$'s and we have to show that only $a$ is a descent. Thus everything else is analogous to the general case.

If $a$ or $b$ is inserted as $-1$ the only way that this causes no descent is that the position to the left is a $-1$ or a $1$ on level zero. A $-1$ directly to the left of $a$ would change because it becomes an $\tilde{a}$ when inserting the second row. A $-1$ directly to the left of $b$ would change into $0$, when inserting the second row. When inserting the third row, this $-1$ would be $\tilde{a}$ or $\tilde{b}$ and thus changed into $0$. A $1$ on level zero left of $a$ or $b$ would cause $a$ to be inserted as $0$ when inserting the second row and would be not on level zero anymore in any other case.

In the second step we show that we do not lose descents when inserting a pair $(a,b)$. If an entry was a descent in the partial tableau before inserting $(a,b)$ it is still a descent in the new partial tableau, either with the same number above, or with $a$ or $b$. In the former case  neither $a$ nor $b$ are inserted inbetween those. In the latter case either $a$ or $b$ is inserted inbetween. This creates a descent in the vacillating tableau and removes the other descent as $(-1,x)$ can never be a descent. Choosing $\tilde{a}$ or $\tilde{b}$ changes a $-1$ into a $0$ or a $0$ into a $-1$. The latter cannot cancel a descent as the position left of this $0$ needs to be also a $0$. The former could only cancel a descent if there is a $0$ directly to the left. This $0$ needs to be on level one and therefore we \enquote{separate}. Thus the descent $(0,-1)$ is changed into the descent $(1,0)$.

What remains to be shows is that \enquote{separate}, \enquote{connect} and \enquote{change $(1,-1)$ into $(0,0)$} when inserting the second row preserves the descent set. For the latter this is clear as $(1,-1)$ on level zero is not a descent. For \enquote{separate} we consider a $0$ left or right of a position that was changed in \enquote{separate}. Those need to be either $\tilde{a}$ or $\tilde{b}$ or they were changed in \enquote{connect}, because otherwise they would have been \enquote{separated} also. In the former case we want a descent, in the latter too, as $(1,0)$ has changed into $(0,-1)$ or the other way around. The same holds for $1$'s or $-1$'s of \enquote{connect}.
\end{proof}

The following two lemmas about Algorithm~\ref{algo3dim} are essential for our bijection and hold as the algorithm is descent-preserving (see Lemma~\ref{LemDescP}) and defines a bijection (see Lemma~\ref{LemsecondInverse}).
\begin{lemma}
\label{Lem10-1}
Let $e_1<e_2<e_3$ be the three biggest entries in the standard Young tableau $\tilde{Q}$. The last three positions of the corresponding vacillating tableau are $1,0,-1$ if and only if $\tilde{Q}$ has three non-empty rows and $e_i$ is the last position of row $i$ for $1\leq i \leq 3$.
\end{lemma}

For the next lemma we point out that there are $1,-1$ on level zero in the end of a vacillating tableau if and only if the two largest numbers in the corresponding standard Young tableaux are in the first row.

\begin{lemma}
\label{LemHorizonal2}
The skew tableau formed by the cells containing the $m$ largest entries of a standard Young tableau are a horizontal strip with at most one box in the first row that is filled increasingly from left to right, if and only if the last $m$ positions of the corresponding vacillating tableau are a sequence of $-1$'s.
\end{lemma} 
 
We complete this section by stating further properties of Algorithm~\ref{algo3dim}. To do so we need the concept of concatenating tableaux.

\begin{definition}
We obtain the \emph{concatenation} $Q$ of two standard Young tableaux $Q_1$ and $Q_2$ in the following way. First we add the largest entry of $Q_1$ to each entry of $Q_2$ to obtain the tableau $\widetilde{Q_2}$. Then we append the $i^{\text{th}}$ row $\widetilde{Q_2}$ to the $i^{\text{th}}$ row of $Q_1$ to obtain $Q$.

The \emph{concatenation} of two vacillating tableaux is obtained by concatenating the words.
\end{definition}

To see if a standard Young tableau $Q$ can be written as concatenation of two standard Young tableaux $Q_1$ and $Q_2$, one has to check every possible $Q_1$. Thus there is no easy way to see this.

Algorithm~\ref{algo3dim} has a nice property concerning concatenation that we prove later, using an alternative formulation of Algorithm~\ref{Ualgo3dim}. The following theorem follows directly from Corollary~\ref{BogenCor2} and the fact that Algorithm~\ref{algo3dim} defines a bijection (see Lemma~\ref{LemsecondInverse}). 

\begin{theorem}
\label{TheoConcat}
Considering Algorithm~\ref{algo3dim}, concatenation of vacillating tableaux of empty shape corresponds to concatenation of standard Young tableaux where all row lengths have the same parity.

In particular, the following holds:
\begin{itemize}
\item If a vacillating tableau is composed of two concatenated paths with weight $0$, its corresponding standard Young tableau can be written as concatenation of two standard Young tableaux, such that each of those has row lengths of the same parity.
\item On the other hand if a standard Young tableau can be written as concatenation of two standard Young tableaux, such that each of those has row lengths of the same parity, its corresponding vacillating tableau is also composed of two concatenated paths with weight $0$.
\end{itemize}
\end{theorem}

\begin{example}
Concatenating the vacillating tableaux
\begin{tikzpicture}[scale=0.35]
\draw (0,0)  -- (1,1)  -- (2,1)  -- (3,0);
\end{tikzpicture}
and
\begin{tikzpicture}[scale=0.35]
\draw (0,0)  -- (1,1)  -- (2,2)  -- (3,1) -- (4,0);
\end{tikzpicture}
yields
\begin{tikzpicture}[scale=0.35]
\draw (0,0)  -- (1,1)  -- (2,1)  -- (3,0) -- (4,1) -- (5,2) -- (6,1) -- (7,0);
\end{tikzpicture}.

\raisebox{0.1cm}{Whereas concatenating the corresponding standard Young tableaux}
\begin{tikzpicture}[scale=0.35]
  \draw (0,4) -- (1,4);
  \draw (0,3) -- (1,3);
  \draw (0,2) -- (1,2);
  \draw (0,1) -- (1,1);
  \draw (0,4) -- (0,1);
  \draw (1,4) -- (1,1);
  \draw (0.5,3.5) node {1};
  \draw (0.5,2.5) node {2};
  \draw (0.5,1.5) node {3};
\end{tikzpicture}
\raisebox{0.1cm}{and}
\begin{tikzpicture}[scale=0.35]
  \draw (0,4) -- (2,4);
  \draw (0,3) -- (2,3);
  \draw (0,2) -- (2,2);
  \draw (0,4) -- (0,2);
  \draw (1,4) -- (1,2);
  \draw (2,4) -- (2,2);
  \draw (0.5,3.5) node {1};
  \draw (1.5,3.5) node {2};
  \draw (0.5,2.5) node {3};
  \draw (1.5,2.5) node {4};
\end{tikzpicture}
\raisebox{0.1cm}{yields}
\begin{tikzpicture}[scale=0.35]
  \draw (0,4) -- (3,4);
  \draw (0,3) -- (3,3);
  \draw (0,2) -- (3,2);
  \draw (0,1) -- (1,1);
  \draw (0,4) -- (0,1);
  \draw (1,4) -- (1,1);
  \draw (2,4) -- (2,2);
  \draw (3,4) -- (3,2);
  \draw (0.5,3.5) node {1};
  \draw (0.5,2.5) node {2};
  \draw (0.5,1.5) node {3};
  \draw (1.5,3.5) node {4};
  \draw (2.5,3.5) node {5};
  \draw (1.5,2.5) node {6};
  \draw (2.5,2.5) node {7};
\end{tikzpicture}\raisebox{0.1cm}.
\end{example}

\begin{remark}
This is another reason why we attach elements $e_1<e_2<e_3$ to a standard Young tableau $Q$ where $e_i$ is $i$ plus the largest entry of $Q$ to deal with tableaux with rows of odd length: \begin{tikzpicture}[scale=0.35]
  \draw (0,4) -- (1,4);
  \draw (0,3) -- (1,3);
  \draw (0,2) -- (1,2);
  \draw (0,1) -- (1,1);
  \draw (0,4) -- (0,1);
  \draw (1,4) -- (1,1);
  \draw (0.5,3.5) node {1};
  \draw (0.5,2.5) node {2};
  \draw (0.5,1.5) node {3};
\end{tikzpicture}
and
\begin{tikzpicture}[scale=0.35]
\draw (0,0)  -- (1,1)  -- (2,1)  -- (3,0);
\end{tikzpicture}
correspond to each other and attaching these elements is just concatenation of $Q$ with this column.
\end{remark}

Moreover, one might be interested in the pre-images of Dyck paths in Algorithm~\ref{algo3dim}. We can answer this question also using the concept of concatenation.

\begin{corollary}
Considering Algorithm~\ref{algo3dim}, Dyck paths, viewed as vacillating tableaux, correspond to those standard Young tableaux that have at most two rows, both of even length, which cannot be written as concatenation of two standard Young tableaux with two rows both of odd length.
\end{corollary}

\begin{proof}
As inserting the third row produces $0$'s (Lemma~\ref{Lemma0s}), the tableaux we are looking for have at most two rows.

First we consider paths of length $r$ with $0$'s that correspond to two-rowed tableaux. Suppose we insert two numbers left and right of the rightmost such $0$ at inserting the third row. Call the label of it $c$. This can be done by adding $c-1/2$ and $r+1/2$ into the third row of the corresponding two-rowed tableau. This would cause \enquote{separate}. Thus the path would be composed of two concatenated paths. Thus we can write the new standard Young tableau as concatenation of two tableaux standard Young tableau. As there is only one position in the third row each, both have two odd rows. Thus the original tableau can be written as concatenation of two standard Young tableaux with two rows, both of odd length.

Now we consider a two-rowed standard Young tableau that can be written as the concatenation of two two-rowed standard Young tableaux with rows of odd length. We add an entry in the third row of each. This corresponds to a concatenation of two odd length paths. Thus during the insertion process of this tableau we have \enquote{separating}, and for \enquote{separating} we need two $0$'s. Thus the corresponding path for our two-rowed tableau cannot be a Dyck path.
\end{proof}

Furthermore we suspect the following.

\begin{conjecture}[Rubey, personal conversation]
Let $T$ be a standard Young tableau with at most three rows all of even length or all of odd length. Suppose that $\varepsilon (T)$ denotes the evacuation of the $T$. Then the vacillating tableau (of shape $\emptyset$) corresponding to $\varepsilon(T)$ is the reversal of the vacillating tableau that corresponds to $T$.
\end{conjecture}

\begin{example}
Remember our running example:

\begin{tikzpicture}[scale=0.35]
  \draw (0,4) -- (6,4);
  \draw (0,3) -- (6,3);
  \draw (0,2) -- (4,2);
  \draw (0,1) -- (4,1);
  \draw (0,4) -- (0,1);
  \draw (1,4) -- (1,1);
  \draw (2,4) -- (2,1);
  \draw (3,4) -- (3,1);
  \draw (4,4) -- (4,1);
  \draw (5,4) -- (5,3);
  \draw (6,4) -- (6,3);
  \draw (0.5,3.5) node {1};
  \draw (1.5,3.5) node {2};
  \draw (0.5,2.5) node {3};
  \draw (0.5,1.5) node {4};
  \draw (1.5,2.5) node {5};
  \draw (2.5,3.5) node {6};
  \draw (1.5,1.5) node {7};
  \draw (2.5,2.5) node {8};
  \draw (3.5,3.5) node {9};
  \draw (2.5,1.5) node {10};
  \draw (4.5,3.5) node {11};
  \draw (5.5,3.5) node {12};
  \draw (3.5,2.5) node {13};
  \draw (3.5,1.5) node {14};
\end{tikzpicture}
\begin{tikzpicture}[scale=0.35]
\draw (0,0)  -- (1,1) -- (2,2) -- (3,2) -- (4,1) -- (5,1) -- (6,2) -- (7,1) -- (8,1) -- (9,2) -- (10,1) -- (11,0) -- (12,1) -- (13,1) -- (14,0);
\end{tikzpicture}

Now applying evacuation on the standard Young tableau yields the following, where the corresponding vacillating tableau is the reversal of the one above.

 \begin{tikzpicture}[scale=0.35]
  \draw (0,4) -- (6,4);
  \draw (0,3) -- (6,3);
  \draw (0,2) -- (4,2);
  \draw (0,1) -- (4,1);
  \draw (0,4) -- (0,1);
  \draw (1,4) -- (1,1);
  \draw (2,4) -- (2,1);
  \draw (3,4) -- (3,1);
  \draw (4,4) -- (4,1);
  \draw (5,4) -- (5,3);
  \draw (6,4) -- (6,3);
  \draw (0.5,3.5) node {1};
  \draw (0.5,2.5) node {2};
  \draw (0.5,1.5) node {3};
  \draw (1.5,3.5) node {4};
  \draw (2.5,3.5) node {5};  
  \draw (1.5,2.5) node {6};
  \draw (3.5,3.5) node {7};
  \draw (4.5,3.5) node {8};
  \draw (1.5,1.5) node {9};
  \draw (2.5,2.5) node {10};
  \draw (5.5,3.5) node {11};
  \draw (3.5,2.5) node {12};
  \draw (2.5,1.5) node {13};
  \draw (3.5,1.5) node {14};
\end{tikzpicture}
\begin{tikzpicture}[scale=0.35]
\draw (0,0)--(1,1)--(2,1)--(3,0)--(4,1)--(5,2)--(6,1)--(7,1)--(8,2)--(9,1)--(10,1)--(11,2)--(12,2)--(13,1)--(14,0);
\end{tikzpicture}
\end{example}

\subsection{Properties of the Bijection}
\label{SubsectionProperties}

\begin{theorem}
The map defined in Section~\ref{bijection} is
\begin{enumerate}
\item well-defined,
\item has a reversed mapping and thus is a bijection between vacillating tableaux and pairs consisting of a standard Young tableau and an orthogonal Littlewood-Richardson tableau,
\item descent-preserving.
\end{enumerate}
\end{theorem}

\begin{proof}
All we need is stated in the Sections~\ref{Sub1stAlgo} and~\ref{Sub2ndAlgo}.

We point out that the second statement of Lemma~\ref{LemHorizonal1}, Lemma~\ref{Lem10-1} and Lemma~\ref{LemHorizonal2} play a crucial role when proving this for the parts of the bijection that are not described by Algorithm~\ref{algoLR3} or~\ref{algo3dim}.
\end{proof}

\subsection{Alternative Formulation of the Second Bijection}
\label{SubsectionAlternative}

In this section we present Algorithm~\ref{AlgoBogen}, an alternative formulation for the inverse algorithm of the second bijection (Algorithm~\ref{Ualgo3dim}) which gives us a better understanding of Algorithms~\ref{Ualgo3dim} and~\ref{algo3dim} and proves some further properties.

\begin{algorithm}
\label{AlgoBogen}
\SetKwInOut{Input}{input}\SetKwInOut{Output}{output}
\Input{vacillating tableau $V$ in terms of a Riordan path}
\Output{standard Young tableau $Q$}
we build a graph, whose vertices are the steps of our Riordan path\;
we allow at most one incoming and at most one outgoing edge for each step; moreover, it is not allowed for an edge, that starts (respectively ends) at a step that has an incoming and an outgoing edge, to end (respectively start) at such a step also; we call those edges \enquote{double edges}\;
moreover, some (non double) edges might be marked\;
\For{entries $i$ in $V$ starting with the leftmost going to right}{
insert (half)edges for entry $i$ between the entries of $V$ as described in Table~\ref{TabAlgoBogen}\;
connect as many (half)edges as possible with respect to the following rules (any such connecting is admissible)\;
\Indp always connect outgoing edges with incoming edges and vice versa, moreover outgoing edges can only be connected to incoming ones to the right\;
outgoing edges from a $-1$ cannot be connected with incoming edges unless there are no open double or marked edges\;
incoming edges of $0$'s are not allowed to be connected with open double or marked edges\;
\Indm put $i$ in the row of $Q$ according to Table~\ref{TabAlgoBogen}\;
}
\Return $Q$\;
 \caption{vacillating tableaux to standard Young tableaux: alternative algorithm}
\end{algorithm}

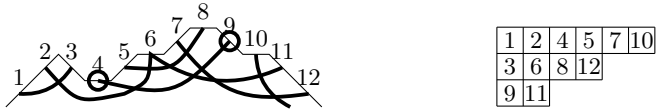
\begin{table}
\centering
\caption{Conditions for Algorithm~\ref{AlgoBogen}}
\label{TabAlgoBogen}
\begin{tabular}{|p{0.7cm}|p{3.6cm}|p{3.9cm}|p{3cm}|p{2.25cm}|p{0.5cm}|}
\hline
entry & \multicolumn{3}{|p{9cm}|}{conditions} & kind of edge(s) & row\\
\hline \hline
1 & \multicolumn{3}{|p{9cm}|}{always (no $0$ is active any more)} & outgoing & 1 \\
\hline
\multirow{6}{*}{0} & \multicolumn{3}{|p{9cm}|}{ $\#$ open double/marked edges $+1<$ level} & incoming and outgoing & 2\\
\cline{2-6}
& \multirow{4}{3.7cm}{$\#$ open double/marked edges $+1=$ level\\ (this $0$ is active, as well as $0$'s right from here until the next $1$ / end of word)} & \multirow{2}{4cm}{level and  $\#$ closed edges have different parity} &  $\#$ active $0$'s odd & incoming and outgoing & 2\\
\cline{4-6}
  &   &   & $\#$ active $0$'s even & incoming & 2\\
 \cline{3-6}
 & &\multirow{2}{4cm}{level and  $\#$ closed edges have same parity} &  $\#$ active $0$'s odd &  marked outgoing & 1\\
 \cline{4-6}
   &   &  & $\#$ active $0$'s even & outgoing & 1\\
    \cline{2-6}
& \multicolumn{3}{|p{10cm}|}{ $\#$ open double/marked edges $+1>$ level} & outgoing & 1\\
\hline
\multirow{4}{*}{$-1$} & \multicolumn{3}{|p{9cm}|}{ $\#$ open double/marked edges $>0$} & incoming & 3\\
\cline{2-6}
& \multirow{3}{4cm}{$\#$ open double/marked edges $=0$} & \multicolumn{2}{|p{6cm}|}{level $>0$} &incoming & 2\\
\cline{3-6}
& & \multirow{2}{4cm}{level $=0$} & $\#$ closed edges odd & incoming &2\\
\cline{4-6}
    &  &  & $\#$ closed edges even  &  outgoing &1 \\
\hline
\end{tabular}
\end{table}

\begin{theorem}
\label{TheoAlgoBogen}
Algorithm~\ref{AlgoBogen} is well-defined and equivalent to Algorithm~\ref{Ualgo3dim} for vacillating tableaux of shape $\emptyset$ and even length.
\end{theorem}

Before we prove this we give some examples.
\begin{example}
In our running example we can obtain the standard Young tableau directly. (Even though we do not need to, we label the edges for better readability.)

\begin{tikzpicture}[scale=0.35]
\draw (0,0)  -- (1,1) node[midway, above]{1} -- (2,2) node[midway, above]{2} -- (3,2) node[midway, above]{3} -- (4,1) node[midway, above]{4} -- (5,1) node[midway, above]{5} -- (6,2) node[midway, above]{6} -- (7,1) node[midway, above]{7} -- (8,1) node[midway, above]{8} -- (9,2)  node[midway, above]{9};
\draw[line width=0.5mm] (0.5,0.5)to[bend right, looseness = 1](2.5,2)to[bend right, looseness = 1](3.5,1.5);
\draw[line width=0.5mm] (1.5,1.5)to[bend right, looseness = 1.2](4.5,1)to[bend right, looseness = 1.5](6.5,1.5);
\draw[line width=0.5mm] (5.5,1.5)to[bend right, looseness = 1.5](7.5,1)to[bend right, looseness = 1](9.5,1);
\draw[line width=0.5mm] (8.5,1.5)to[bend right, looseness = 1](9.5,1.5);
\end{tikzpicture}
\hspace{3.1cm}
\begin{tikzpicture}[scale=0.35]
  \draw (0,4) -- (4,4);
  \draw (0,3) -- (4,3);
  \draw (0,2) -- (3,2);
  \draw (0,1) -- (2,1);
  \draw (0,4) -- (0,1);
  \draw (1,4) -- (1,1);
  \draw (2,4) -- (2,1);
  \draw (3,4) -- (3,2);
  \draw (4,4) -- (4,3);
  \draw (0.5,3.5) node {1};
  \draw (1.5,3.5) node {2};
  \draw (0.5,2.5) node {3};
  \draw (0.5,1.5) node {4};
  \draw (1.5,2.5) node {5};
  \draw (2.5,3.5) node {6};
  \draw (1.5,1.5) node {7};
  \draw (2.5,2.5) node {8};
  \draw (3.5,3.5) node {9};
\end{tikzpicture}
\end{example}

\begin{example}
We give another example where a marked edge occurs and the vacillating tableau has weight $\emptyset$. We point out that while ordinary edges can be open when the path returns to level zero, double or marked edges are never open at such a point.

\begin{tikzpicture}[scale=0.35]
\draw (0,0)  -- (1,1) node[midway, above]{1} -- (2,2) node[midway, above]{2} -- (3,1) node[midway, above]{3} -- (4,1) node[midway, above]{4} -- (5,2) node[midway, above]{5} -- (6,2) node[midway, above]{6} -- (7,3) node[midway, above]{7} -- (8,3) node[midway, above]{8} -- (9,2) node[midway, above]{9} -- (10,2) node[midway, above]{10} -- (11,1) node[midway, above]{11} -- (12,0) node[midway, above]{12};
\draw[line width=0.5mm] (0.5,0.5)to[bend right, looseness = 1](2.5,1.5);
\draw[line width=0.5mm] (1.5,1.5)to[bend right, looseness = 1.5](4.5,0.5)to[bend right, looseness = 1.5](5.5,2)to[bend right, looseness = 1](10.5,1.5);
\draw[line width=0.5mm] (3.5,1) node[draw, circle, scale=0.8] {} to[bend right, looseness = 1.4](8.5,2.5) node[draw, circle, scale=0.8] {};
\draw[line width=0.5mm] (4.5,1.5)to[bend right, looseness = 1.5](7.5,3);
\draw[line width=0.5mm] (6.5,2.5)to[bend right, looseness = 1.2](11.5,0.5);
\draw[line width=0.5mm] (9.5,2)to[bend right, looseness = 1](10.8,0);
\end{tikzpicture}
\hspace{2cm}
\begin{tikzpicture}[scale=0.35]
  \draw (0,4) -- (6,4);
  \draw (0,3) -- (6,3);
  \draw (0,2) -- (4,2);
  \draw (0,1) -- (2,1);
  \draw (0,4) -- (0,1);
  \draw (1,4) -- (1,1);
  \draw (2,4) -- (2,1);
  \draw (3,4) -- (3,2);
  \draw (4,4) -- (4,2);
  \draw (5,4) -- (5,3);
  \draw (6,4) -- (6,3);
  \draw (0.5,3.5) node {1};
  \draw (1.5,3.5) node {2};
  \draw (0.5,2.5) node {3};
  \draw (2.5,3.5) node {4};
  \draw (3.5,3.5) node {5};
  \draw (1.5,2.5) node {6};
  \draw (4.5,3.5) node {7};
  \draw (2.5,2.5) node {8};
  \draw (0.5,1.5) node {9};
  \draw (5.5,3.5) node{10};
  \draw (1.5,1.5) node{11};
  \draw (3.5,2.5) node{12};
\end{tikzpicture}
\end{example}

\begin{example}
This example shows the crucial point in counting open double or marked edges. Note that before inserting $9$, the tableaux, the height, and the next position to consider are the same.

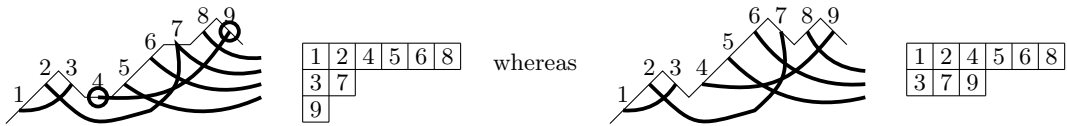
\begin{tikzpicture}[scale=0.35]
\draw (0,0)  -- (1,1) node[midway, above]{1} -- (2,2) node[midway, above]{2} -- (3,1) node[midway, above]{3} -- (4,1) node[midway, above]{4} -- (5,2) node[midway, above]{5} -- (6,3) node[midway, above]{6} -- (7,3) node[midway, above]{7} -- (8,4) node[midway, above]{8} -- (9,3) node[midway, above]{9};
\draw[line width=0.5mm] (0.5,0.5)to[bend right, looseness = 1](2.5,1.5);
\draw[line width=0.5mm] (1.5,1.5)to[bend right, looseness = 1.5](5.5,0.5)to[bend right, looseness = 1.3](6.5,3)to[bend right, looseness = 1](9.7,2);
\draw[line width=0.5mm] (3.5,1) node[draw, circle, scale=0.8] {} to[bend right, looseness = 1.2](8.5,3.5) node[draw, circle, scale=0.8] {};
\draw[line width=0.5mm] (4.5,1.5)to[bend right, looseness = 1](9.7,1);
\draw[line width=0.5mm] (5.5,2.5)to[bend right, looseness = 1](9.7,1.5);
\draw[line width=0.5mm] (7.5,3.5)to[bend right, looseness = 1](9.7,2.5);
\end{tikzpicture}
\hspace{0.2cm}
\begin{tikzpicture}[scale=0.35]
  \draw (0,4) -- (6,4);
  \draw (0,3) -- (6,3);
  \draw (0,2) -- (2,2);
  \draw (0,1) -- (1,1);
  \draw (0,4) -- (0,1);
  \draw (1,4) -- (1,1);
  \draw (2,4) -- (2,2);
  \draw (3,4) -- (3,3);
  \draw (4,4) -- (4,3);
  \draw (5,4) -- (5,3);
  \draw (6,4) -- (6,3);
  \draw (0.5,3.5) node {1};
  \draw (1.5,3.5) node {2};
  \draw (0.5,2.5) node {3};
  \draw (2.5,3.5) node {4};
  \draw (3.5,3.5) node {5};
  \draw (4.5,3.5) node {6};
  \draw (1.5,2.5) node {7};
  \draw (5.5,3.5) node {8};
  \draw (0.5,1.5) node {9};
  \draw (0,1) node{};
\end{tikzpicture}
\hspace{0.2cm}
\raisebox{0.8cm}{whereas}
\hspace{0.2cm}
\begin{tikzpicture}[scale=0.35]
\draw (0,0)  -- (1,1) node[midway, above]{1} -- (2,2) node[midway, above]{2} -- (3,1) node[midway, above]{3} -- (4,2) node[midway, above]{4} -- (5,3) node[midway, above]{5} -- (6,4) node[midway, above]{6} -- (7,3) node[midway, above]{7} -- (8,4) node[midway, above]{8} -- (9,3) node[midway, above]{9};
\draw[line width=0.5mm] (0.5,0.5)to[bend right, looseness = 1](2.5,1.5);
\draw[line width=0.5mm] (1.5,1.5)to[bend right, looseness = 1.5](5.5,0.5)to[bend right, looseness = 1.3](6.5,3.5);
\draw[line width=0.5mm] (3.5,1.5) node {} to[bend right, looseness = 1.2](8.5,3.5) node {};
\draw[line width=0.5mm] (4.5,2.5)to[bend right, looseness = 1](9.7,1);
\draw[line width=0.5mm] (5.5,3.5)to[bend right, looseness = 1](9.7,1.5);
\draw[line width=0.5mm] (7.5,3.5)to[bend right, looseness = 1](9.7,2);
\end{tikzpicture}
\hspace{0.2cm}
\begin{tikzpicture}[scale=0.35]
  \draw (0,4) -- (6,4);
  \draw (0,3) -- (6,3);
  \draw (0,2) -- (3,2);
  \draw (0,4) -- (0,2);
  \draw (1,4) -- (1,2);
  \draw (2,4) -- (2,2);
  \draw (3,4) -- (3,2);
  \draw (4,4) -- (4,3);
  \draw (5,4) -- (5,3);
  \draw (6,4) -- (6,3);
  \draw (0.5,3.5) node {1};
  \draw (1.5,3.5) node {2};
  \draw (0.5,2.5) node {3};
  \draw (2.5,3.5) node {4};
  \draw (3.5,3.5) node {5};
  \draw (4.5,3.5) node {6};
  \draw (1.5,2.5) node {7};
  \draw (5.5,3.5) node {8};
  \draw (2.5,2.5) node {9};
  \draw (0,1) node{};
\end{tikzpicture}
\end{example}

\begin{proof}[Proof of Theorem~\ref{TheoAlgoBogen}]
We start this proof with providing properties of Algorithm~\ref{AlgoBogen}, that prove among others that the algorithm is well-defined:
\begin{itemize}
\item Only the first active $0$ of a sequence of active $0$'s can introduce an open double or marked edge or an incoming-only edge.

To see this we show that once an active $0$ occurs the next active $0$ in the same sequence will be outgoing-only, as well as once an active $0$ has an outgoing-only edge, until the next $1$, all active $0$'s have outgoing-only edges.

Once there is an outgoing double or marked edge the height condition together with the rule that such edges can only be closed at $-1$'s ensure that for all following active $0$'s the the third case for the level will be chosen.

Once there is an incoming-only edge the number of closed edges increases by one. Thus the parity changes. A height decrease also increases the number of closed edges such that we stay in the case of same parity. Moreover, because the number of active $0$'s is even, which is necessary for the incoming-only case, we stay in the case of outgoing-only edges.

Once there is an outgoing-only case the number of closed edges only increases at a loss of height. Now this is analogous to the former argument.

\item Exactly in those steps where there is an incoming (or of course incoming and outgoing) edge introduced, there is exactly one connection of (half)edges. This is between a former outgoing edge and this new incoming one. In any other case it is not possible to connect (half)edges.

To see this we show that there are always enough outgoing open edges. Every $1$ produces an outgoing edge. Every $-1$ (except for the last one maybe) has an incoming one. Therefore without $0$'s this is regulated automatically with the level. The only $0$ case where the number of outgoing edges is reduced is the one where there is only an incoming edge. However in this case the number of active $0$'s is even and due to the former point we know that an outgoing-only edge will follow before reaching the bottom level.

\item There are never open double or marked edges when the vacillating tableau reaches a bottom point, thus level zero.
This holds because there cannot be such new edges when there are as many open ones as the height and those edges need to be closed before an ordinary edge can be closed by a $-1$.

\item The order in which the edges are closed makes no difference. Thus we can assume some first-come-first-served principle to connect (half)edges, and do not consider open (half)edges left of the last $1$ to the left on level zero. Moreover, this ensures that the map defined by the algorithm is well-defined.
\end{itemize}

Next we show that the theorem holds for vacillating tableaux without $0$'s in 3-row-position. Those are mapped to two-rowed standard Young tableaux in Algorithm~\ref{Ualgo3dim}.

In this case we have no $0$'s on a level higher than one and those at level one are always together in even numbers. Therefore we never have any double or marked edges. Thus nothing is put into the third row.

If we go through the path from left to right and find a $-1$ not on level zero, left of the leftmost $0$, if a $0$ exists, we will connect this with some $1$. If we find a $0$ left of the leftmost $-1$ not on level zero this will connect with the $1$ to the left. In both cases it will get to the second row and be our $a$, exactly as it would be in Algorithm~\ref{Ualgo3dim}.

As the number of closed edges is now odd the next closed edge will end in a $-1$ which will also get to the second row and therefore be our $b$, exactly as it would be in Algorithm~\ref{Ualgo3dim}.

Now it makes no difference if we extract those or not, because the number of closed edges is even now, there is a new area of active $0$'s and every other statistic is the same.

Moreover, $-1$'s that belong to the first line are treated correctly as those are at level zero, and will not be considered as new $a$ when the number of closed edges are even.\par\smallskip

Next we show that the algorithm finds the $0$'s in 3-row-position and their associated $-1$'s. It thus follows that the right numbers are put into the third row. We do this inductively.

$0$'s that are at level two or higher when executing Algorithm~\ref{Ualgo3dim} are recognized automatically. In Algorithm~\ref{Ualgo3dim} there would be a level decrease of $1$ between $\tilde{a}$ (respectively $\tilde{b}$) and $a$ (respectively $b$). Therefore open double or marked edges increase the needed height. The next $-1$ not already used by other $0$'s in 3-row-position is used for $a$ (respectively $b$).

$0$'s that are on level one, part of a sequence of $0$'s of odd length, are also recognized. Those are the first positions of active $0$'s as the level condition shows: $1=$level$-\#$ of open double or marked edges.\par\smallskip

Finally we show that it makes no difference for the two-row-case of Algorithm~\ref{AlgoBogen} if we do not carry out changes done by extracting the third row.

When extracting $a$ there are three cases:
\begin{itemize}
\item $\tilde{a}$ is a non active $0$. It becomes a $-1$ on level one or higher, thus still is on level one or higher.
\item $\tilde{a}$ is an active $0$. It becomes a $-1$ on level zero. Whether this is placed in the second or the first row depends in both situations on the parity of closed edges.
\item $\tilde{a}$ is a $1$. Thus to the right there is an active $0$ and the number of active $0$'s is odd. Again whether this is placed in the second or the first row depends in both situations on the parity of closed edges.
\end{itemize}
Extracting $a$ changes the parity of closed edges. Thus in the situations where there is a dependency on the parity of closed edges, this change is considered.

When \enquote{undoing connect} this happens between two active $0$'s. While the right one is always put into the first row, there are two possible ways for the left ones (as well as there are for $-1$'s on level zero), depending on the parity of closed edges. Again everything holds due to the change of the parity when extracting $a$. The same argument holds for \enquote{separate}.

When extracting $b$ there are two cases that are analogous to the first two cases when extracting $a$. Note that if $\tilde{a}<\tilde{b}<a<b$ the level and the number of open double or marked edges is reduced by one.

After extracting $a$ and $b$ the parity of open double or marked edges is the same as before as well as the parity of closed edges. Moreover, there is no conflict with active $0$'s as in an area of active $0$'s there can be at most one $\tilde{a}$ or $\tilde{b}$.
\end{proof}

Algorithm~\ref{AlgoBogen} proves the property about concatenation of tableaux we have stated before in Theorem~\ref{TheoConcat}. As there cannot be an open double or marked edge when the path reaches the bottom (thus has steps on level zero) due to level conditions, concatenated parts of paths are independent of each other.

\begin{corollary}
\label{BogenCor2}
Concatenation of paths with weight $0$ corresponds to concatenation of standard Young tableaux, all of whose row length have the same parity, when applying Algorithm~\ref{AlgoBogen}. Therefore this also holds for Algorithm~\ref{algo3dim} and Algorithm~\ref{Ualgo3dim}.
\end{corollary}

We consider now another alternative formulation.

\begin{corollary}
\label{BogenCor1}
Algorithm~\ref{AlgoBogenAlt} is an alternative formulation of Algorithm~\ref{AlgoBogen}. For an entry to insert, it only needs the current entry and four statistics we can calculate using only entries that are already inserted.
\end{corollary}

\begin{algorithm}
\label{AlgoBogenAlt}
\SetKwInOut{Input}{input}\SetKwInOut{Output}{output}
\Input{vacillating tableau $V$ in terms of a Riordan path}
\Output{standard Young tableau $Q$}
$a_0=odm=h=ce=0$\;
\For{entries $i$ in $V$ starting with the leftmost going right}{
\Switch{entry}{
\Case{$1$}{\lIf{$a_0$ odd}{odm=odm+1} $h=h+1$; $a_0=0$; put $i$ in the first row of $Q$\;}
\Case{$0$}{\Switch{$odm$}{
\lCase{$<h-1$}{$odm=odm+1$; $ce=ce+1$; put $i$ in the second row of $Q$} 
\Case{$=h-1$}{$a_0=a_0+1$\; \lIf{$h \not\equiv ce \mod 2$}{$ce=ce+1$; put $i$ in the second row of $Q$} \lElse{put $i$ in the first row of $Q$}}
\lCase{$>h-1$}{$a_0=a_0+1$; put $i$ in the first row of $Q$}
}}
\Case{$-1$}{
$h=h-1$\;
\lIf{$a_0$ odd \textbf{\emph{and}} $h=0$}{$odm=odm+1$; $a_0=0$}
\Switch{$odm$}{
\lCase{$>0$}{$odm=odm-1$; $ce=ce+1$; put $i$ in the third row of $Q$}
\Case{$=0$}{\Switch{h}{
\lCase{$>0$}{$ce=ce+1$; put $i$ in the second row of $Q$}
\Case{$=0$}{\lIf{$ce$ odd}{ $ce=ce+1$; put $i$ in the second row of $Q$}
\lElse{put $i$ in the first row of $Q$}}
}}
}}
}
}
\Return $Q$;
 \caption{vacillating tableaux to standard Young tableaux: second alternative algorithm}
\end{algorithm}

\begin{proof}
Algorithm~\ref{AlgoBogenAlt} satisfies the mentioned properties due to its formulation.

We see that Algorithm~\ref{AlgoBogenAlt} is equivalent to Algorithm~\ref{AlgoBogen} when pointing out that:
\begin{itemize}
\item $a_0$ in Algorithm~\ref{AlgoBogenAlt} is the number of active $0$'s in Algorithm~\ref{AlgoBogen}.
\item $odm$ in Algorithm~\ref{AlgoBogenAlt} is the number of open double or marked edges from Algorithm~\ref{AlgoBogen}.
\item $h$ in Algorithm~\ref{AlgoBogenAlt} is the height before some part of the path. Thus for a $0$ or a $-1$ it is the level and for a $1$ it is the level minus one. Although we use the level in Algorithm~\ref{AlgoBogen}, this does not matter as the neither the height nor the level is considered for a $1$ in any of the two algorithms.
\item $ce$ is the number of closed edges in Algorithm~\ref{AlgoBogen}.
\end{itemize}

The main difference between Algorithm~\ref{AlgoBogen} and~\ref{AlgoBogenAlt} is that in the latter the decision if an active $0$ has an outgoing edge such that this is part of a double edge or a marked edge, is done at the next $1$ or the next $-1$ where the path touches the bottom. However, this makes no difference because for active $0$'s it holds that for every $-1$ that is between them, there is another open double or marked edge that can connect with the $-1$. Besides, only $-1$'s on level zero interrupt a sequence of active $0$'s. (This explains the $h=0$ condition for $odm=odm+1$.)

Moreover, only one active $0$'s can have a double or marked edge. This holds because once this occurs in the original algorithm, the number of those edges increases and so the last case for $0$'s occurs until the next $1$. 
\end{proof}

\begin{remark}
One might think that this could be again rewritten in terms of an automaton. However as we need $h$ and $odm$ to be any positive natural number this could not be realized with a finite one. (As we need only the parity of the other statistics those could be implemented as boolean.)
\end{remark}

\begin{remark}
Furthermore, due to Corollary~\ref{BogenCor1} we get:
\begin{itemize}
\item For a vacillating tableau of any shape Algorithm~\ref{AlgoBogenAlt} produces the right standard Young tableau. For the information on the orthogonal Littlewood-Richardson tableau, we need to add $-1$'s until shape $\emptyset$ is reached as we do in the reversed bijection.
\item For vacillating tableaux of even length and shape $\emptyset$ we do not need to add $1,0,-1$ and delete the corresponding entries of the standard Young tableau afterwards.
\end{itemize}
\end{remark}

\begin{corollary}
\label{InsertionCoro}
Algorithm~\ref{AlgoBogenAlt} provides an insertion algorithm in the sense of~\cite[Theorem 6.2 (2)]{MR899903}: if the vacillating tableau $(v_1,v_2,\dots,v_{m-1},v_m)$ is mapped to the standard Young tableau $Q$ of size $m$, the vacillating tableau $(v_1,v_2,\dots,v_{m-1})$ is mapped to the standard Young tableau $\bar{Q}$ of size $m-1$ which is obtained by deleting the biggest entry $m$.
\end{corollary}

\section{Outlook}

\subsection{Sketch of the extension to $\SO(2k+1)$}

For higher dimensions we consider vacillating tableaux as tuples of paths with certain dependencies. We use an inductive approach to obtain an insertion algorithm for these paths. We start with the first row as in the $\SO(3)$ case. Then we insert rows in pairs using a variant of the insertion algorithm for the second and third row presented here.
However we need to change the \enquote{separating} and \enquote{connecting} parts of the algorithm due to the dependencies between the paths.

As this algorithm is far more complicated it will be presented in another paper.

\subsection{Open problems}

It would be interesting to do the same for the vector representation of other groups:
\begin{itemize}

\item $\SOn$, $n$ even, vector representation - there is already a descent set conjectured by Rubey.
\item $\mathrm{G}_2$, vector representation - there is already a descent set conjectured by Rubey.
\item $\mathrm{Sp}(n)$, $n$ even, vector representation using  Kwon's symplectic Littlewood-Richardson tableaux instead of Sundaram's symplectic Littlewood-Richardson tableaux.
\end{itemize}

\section*{Acknowledgements}

The author would like to thank Martin Rubey for valuable discussions and helpful comments on preliminary versions of this article.
The author also thanks Sheila Sundaram for her ongoing interest on this work and for her helpul comments and also Cristian Lenart for his helpful comments.
Moreover the author thanks an anonymous referee for a very careful reading of the paper and corrections and comments.
We used the computer algebra system SageMath to implement the algorithms and test examples.



\clearpage
\appendix
\section{Appendix}

\begin{table}[h]
\centering
\caption{List of all tableaux with $n=3$, $r\leq 3$.
Note that there is not necessarily an orthogonal Littlewood-Richardson tableau for every combination of $\mu\leq\lambda$. Moreover, if we have an orthogonal Littlewood-Richardson tableau consisting of only one column, it is important if this is the tail or the single column tableau. 
}
\label{TabListExample}
\begin{tabular}{|p{1.8cm}|p{1.8cm}|p{1.8cm}|p{1.8cm}|p{1.8cm}|p{1.8cm}|}
\hline
$\lambda$ & $\mu$ & $L$ &  $Q$  & $\tilde{Q}$ & $V$\\
\hline
\multicolumn{6}{p{10.8cm}}{$r=1$}\\
\hline
$(1)$& $(1)$ &
\begin{tikzpicture}[scale=0.35]
  \draw (0,0)--(0,1);
  \draw (1,0)--(1,1);
  \draw (0,0)--(1,0);
  \draw (0,1)--(1,1);
  \draw (0.5,0.5) node {1};
\end{tikzpicture}
&
\begin{tikzpicture}[scale=0.35]
  \draw (0,0)--(0,1);
  \draw (1,0)--(1,1);
  \draw (0,0)--(1,0);
  \draw (0,1)--(1,1);
  \draw (0.5,0.5) node {1};
\end{tikzpicture}
&
\begin{tikzpicture}[scale=0.35]
  \draw (0,0)--(0,1);
  \draw (1,0)--(1,1);
  \draw (2,0)--(2,1);
  \draw (0,0)--(2,0);
  \draw (0,1)--(2,1);
  \draw (0.5,0.5) node {1};
  \draw (1.5,0.5) node {2};
\end{tikzpicture}&
\begin{tikzpicture}[scale=0.35]
  \draw (0,0)--(1,1);
\end{tikzpicture}\\
\hline
\multicolumn{6}{p{10.8cm}}{$r=2$:}\\
\hline
$(1,1)$ & $(1)$ &
\begin{tikzpicture}[scale=0.35]
  \draw (0,0)--(0,1);
  \draw (1,0)--(1,1);
  \draw (0,0)--(1,0);
  \draw (0,1)--(1,1);
  \draw (0.5,0.5) node {1};
  \draw (2.5,1)--(2.5,2);
  \draw (3.5,1)--(3.5,2);
  \draw (2.5,1)--(3.5,1);
  \draw (2.5,2)--(3.5,2);
  \draw (3,1.5) node {1};
\end{tikzpicture} 
 &
\begin{tikzpicture}[scale=0.35]
  \draw (0,0)--(0,2);
  \draw (1,0)--(1,2);
  \draw (0,0)--(1,0);
  \draw (0,1)--(1,1);
  \draw (0,2)--(1,2);
  \draw (0.5,1.5) node {1};
  \draw (0.5,0.5) node {2};
\end{tikzpicture}
 &
 \begin{tikzpicture}[scale=0.35]
  \draw (0,0)--(0,3);
  \draw (1,0)--(1,3);
  \draw (0,0)--(1,0);
  \draw (0,1)--(1,1);
  \draw (0,2)--(1,2);
  \draw (0,3)--(1,3);
  \draw (0.5,2.5) node {1};
  \draw (0.5,1.5) node {2};
  \draw (0.5,0.5) node {3};
\end{tikzpicture}
 &\begin{tikzpicture}[scale=0.35]
  \draw (0,0)--(1,1)--(2,1);
\end{tikzpicture}\\
\hline
\end{tabular}
\end{table}

\begin{figure}
$r=2$:\\
\begin{tabular}{l}
$\{\}$: \begin{tikzpicture}[scale=0.35]
\draw(0,0)--(2,0);
\draw(0,1)--(2,1);
\draw(0,0)--(0,1);
\draw(1,0)--(1,1);
\draw(2,0)--(2,1);
\draw(0.5,0.5) node{1};
\draw(1.5,0.5) node{2};
\end{tikzpicture}
\begin{tikzpicture}[scale=0.35]
\draw (0,0)--(1,1)--(2,0);
\end{tikzpicture}
\end{tabular}\\

\vspace{0.25cm}
$r=3$:\\
\begin{tabular}{l}
$\{1,2\}$: \begin{tikzpicture}[scale=0.35]
\draw(0,0)--(1,0);
\draw(0,1)--(1,1);
\draw(0,2)--(1,2);
\draw(0,3)--(1,3);
\draw(0,0)--(0,3);
\draw(1,0)--(1,3);
\draw(0.5,2.5) node{1};
\draw(0.5,1.5) node{2};
\draw(0.5,0.5) node{3};
\end{tikzpicture}
\begin{tikzpicture}[scale=0.35]
\draw (0,0)--(1,1)--(2,1)--(3,0);
\end{tikzpicture}
\end{tabular}\\

\vspace{0.25cm}
$r=4$:\\
\begin{tabular}{p{5cm}p{5cm}p{5cm}}
\noindent $\{\}$: \begin{tikzpicture}[scale=0.35]
\draw(0,0)--(4,0);
\draw(0,1)--(4,1);
\draw(0,0)--(0,1);
\draw(1,0)--(1,1);
\draw(2,0)--(2,1);
\draw(3,0)--(3,1);
\draw(4,0)--(4,1);
\draw(0.5,0.5) node{1};
\draw(1.5,0.5) node{2};
\draw(2.5,0.5) node{3};
\draw(3.5,0.5) node{4};
\end{tikzpicture}
\begin{tikzpicture}[scale=0.35]
\draw (0,0)--(1,1)--(2,0)--(3,1)--(4,0);
\end{tikzpicture}&
$\{2\}$: \begin{tikzpicture}[scale=0.35]
\draw(0,0)--(2,0);
\draw(0,1)--(2,1);
\draw(0,2)--(2,2);
\draw(0,0)--(0,2);
\draw(1,0)--(1,2);
\draw(2,0)--(2,2);
\draw(0.5,1.5) node{1};
\draw(1.5,1.5) node{2};
\draw(0.5,0.5) node{3};
\draw(1.5,0.5) node{4};
\end{tikzpicture}
\begin{tikzpicture}[scale=0.35]
\draw (0,0)--(1,1)--(2,2)--(3,1)--(4,0);
\end{tikzpicture}&
$\{1,3\}$: \begin{tikzpicture}[scale=0.35]
\draw(0,0)--(2,0);
\draw(0,1)--(2,1);
\draw(0,2)--(2,2);
\draw(0,0)--(0,2);
\draw(1,0)--(1,2);
\draw(2,0)--(2,2);
\draw(0.5,1.5) node{1};
\draw(1.5,1.5) node{3};
\draw(0.5,0.5) node{2};
\draw(1.5,0.5) node{4};
\end{tikzpicture}
\begin{tikzpicture}[scale=0.35]
\draw (0,0)--(1,1)--(2,1)--(3,1)--(4,0);
\end{tikzpicture}
\end{tabular}\\

\vspace{0.25cm}
$r=5$:\\
\begin{tabular}{p{5cm}p{5cm}p{5cm}}
$\{3,4\}$: \begin{tikzpicture}[scale=0.35]
\draw(0,0)--(1,0);
\draw(0,1)--(1,1);
\draw(0,2)--(3,2);
\draw(0,3)--(3,3);
\draw(0,0)--(0,3);
\draw(1,0)--(1,3);
\draw(2,2)--(2,3);
\draw(3,2)--(3,3);
\draw(0.5,2.5) node{1};
\draw(1.5,2.5) node{2};
\draw(2.5,2.5) node{3};
\draw(0.5,1.5) node{4};
\draw(0.5,0.5) node{5};
\end{tikzpicture}
\begin{tikzpicture}[scale=0.35]
\draw (0,0)--(1,1)--(2,0)--(3,1)--(4,1)--(5,0);
\end{tikzpicture}&
$\{2,4\}$: \begin{tikzpicture}[scale=0.35]
\draw(0,0)--(1,0);
\draw(0,1)--(1,1);
\draw(0,2)--(3,2);
\draw(0,3)--(3,3);
\draw(0,0)--(0,3);
\draw(1,0)--(1,3);
\draw(2,2)--(2,3);
\draw(3,2)--(3,3);
\draw(0.5,2.5) node{1};
\draw(1.5,2.5) node{2};
\draw(2.5,2.5) node{4};
\draw(0.5,1.5) node{3};
\draw(0.5,0.5) node{5};
\end{tikzpicture}
\begin{tikzpicture}[scale=0.35]
\draw (0,0)--(1,1)--(2,2)--(3,1)--(4,1)--(5,0);
\end{tikzpicture}&
$\{1,4\}$: \begin{tikzpicture}[scale=0.35]
\draw(0,0)--(1,0);
\draw(0,1)--(1,1);
\draw(0,2)--(3,2);
\draw(0,3)--(3,3);
\draw(0,0)--(0,3);
\draw(1,0)--(1,3);
\draw(2,2)--(2,3);
\draw(3,2)--(3,3);
\draw(0.5,2.5) node{1};
\draw(1.5,2.5) node{3};
\draw(2.5,2.5) node{4};
\draw(0.5,1.5) node{2};
\draw(0.5,0.5) node{5};
\end{tikzpicture}
\begin{tikzpicture}[scale=0.35]
\draw (0,0)--(1,1)--(2,1)--(3,1)--(4,1)--(5,0);
\end{tikzpicture}\\
\end{tabular}\\

\vspace{0.25cm}
$r=6$:\\
\begin{tabular}{p{5cm}p{5cm}p{5cm}}
$\{\}$: \begin{tikzpicture}[scale=0.35]
\draw(0,0)--(6,0);
\draw(0,1)--(6,1);
\draw(0,0)--(0,1);
\draw(1,0)--(1,1);
\draw(2,0)--(2,1);
\draw(3,0)--(3,1);
\draw(4,0)--(4,1);
\draw(5,0)--(5,1);
\draw(6,0)--(6,1);
\draw(0.5,0.5) node{1};
\draw(1.5,0.5) node{2};
\draw(2.5,0.5) node{3};
\draw(3.5,0.5) node{4};
\draw(4.5,0.5) node{5};
\draw(5.5,0.5) node{6};
\end{tikzpicture}
\begin{tikzpicture}[scale=0.35]
\draw (0,0)--(1,1)--(2,0)--(3,1)--(4,0)--(5,1)--(6,0);
\end{tikzpicture}&
$\{4\}$: \begin{tikzpicture}[scale=0.35]
\draw(0,0)--(2,0);
\draw(0,1)--(4,1);
\draw(0,2)--(4,2);
\draw(0,0)--(0,2);
\draw(1,0)--(1,2);
\draw(2,0)--(2,2);
\draw(3,1)--(3,2);
\draw(4,1)--(4,2);
\draw(0.5,1.5) node{1};
\draw(1.5,1.5) node{2};
\draw(2.5,1.5) node{3};
\draw(3.5,1.5) node{4};
\draw(0.5,0.5) node{5};
\draw(1.5,0.5) node{6};
\end{tikzpicture}
\begin{tikzpicture}[scale=0.35]
\draw (0,0)--(1,1)--(2,0)--(3,1)--(4,2)--(5,1)--(6,0);
\end{tikzpicture}&
$\{3,5\}$: \begin{tikzpicture}[scale=0.35]
\draw(0,0)--(2,0);
\draw(0,1)--(4,1);
\draw(0,2)--(4,2);
\draw(0,0)--(0,2);
\draw(1,0)--(1,2);
\draw(2,0)--(2,2);
\draw(3,1)--(3,2);
\draw(4,1)--(4,2);
\draw(0.5,1.5) node{1};
\draw(1.5,1.5) node{2};
\draw(2.5,1.5) node{3};
\draw(3.5,1.5) node{5};
\draw(0.5,0.5) node{4};
\draw(1.5,0.5) node{6};
\end{tikzpicture}
\begin{tikzpicture}[scale=0.35]
\draw (0,0)--(1,1)--(2,0)--(3,1)--(4,1)--(5,1)--(6,0);
\end{tikzpicture} \\
\end{tabular}
\caption{Standard Young tableaux with at most 3 rows, all even / odd length and vacillating tableaux up to $r=6$ sorted by their descents. The smallest case with a non unique descent set appears at $r=6$. The assignment we have here corresponds to our mapping.}
\label{FigAllTabs}
\end{figure}
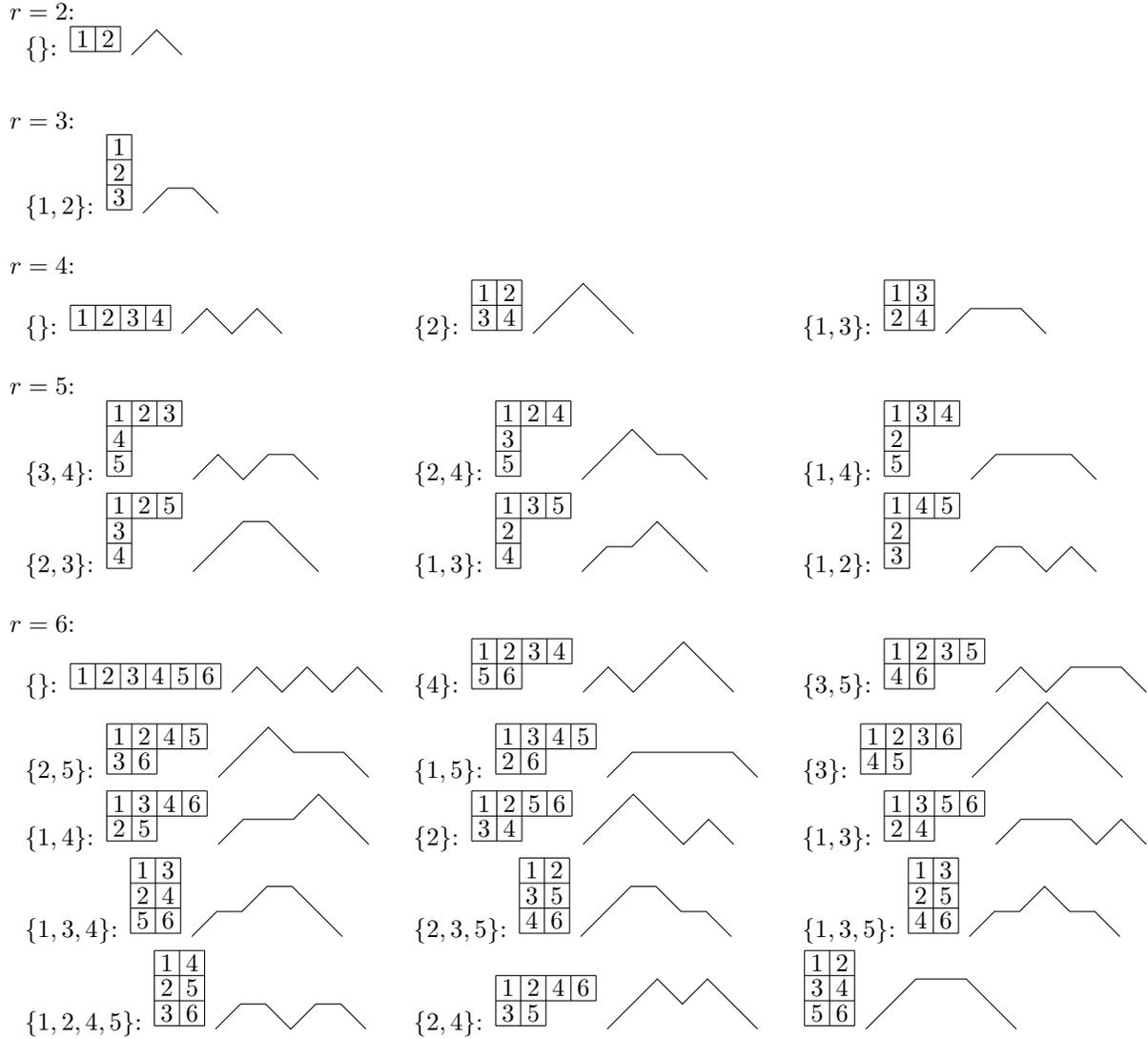

\begin{figure}
\begin{tabular}{|p{4cm}|p{5cm}|p{6cm}|}
\multicolumn{3}{p{15cm}}{Insert row 2:}\\
\hline
Insert $b$, insert $a$ case 1:&
\begin{tikzpicture}[scale=0.35]
\draw (0,0)--(1,1)--(2,0);
\draw (1,1) circle (3pt) node[above] {$a$};
\draw [dotted] (2,0)--(3,1)--(4,0)--(5,1)--(6,2);
\draw (6,2) circle (3pt) node[above] {$b$};
\end{tikzpicture}&
\begin{tikzpicture}[scale=0.35]
\draw (0,0)--(1,1)--(2,1) node[midway, above] {$a$} --(3,1);
\draw [dotted] (3,1)--(4,1)--(5,1)--(6,2)--(7,3);
\draw (7,3)--(8,2) node[midway, above] {$b$};
\end{tikzpicture}\\
\hline
\multirow{2}{*}{Insert $b$, insert $a$ case 2:}&
\begin{tikzpicture}[scale=0.35]
\draw (0,0)--(1,1)--(2,0);
\draw (2,0) circle (3pt) node[above] {$a$};
\draw [dotted] (2,0)--(3,1)--(4,0)--(5,1)--(6,2);
\draw (6,2) circle (3pt) node[above] {$b$};
\end{tikzpicture}&
\begin{tikzpicture}[scale=0.35]
\draw (0,0)--(1,1)--(2,2) --(3,1)node[midway, above] {$a$};
\draw [dotted] (3,1)--(4,1)--(5,1)--(6,2)--(7,3);
\draw (7,3)--(8,2) node[midway, above] {$b$};
\end{tikzpicture}\\
\cline{2-3} &
\begin{tikzpicture}[scale=0.35]
\draw (0,0)--(1,1)--(2,0);
\draw [dotted] (2,0)--(3,1)--(4,2)--(5,3);
\draw (4,2) circle (3pt) node[above] {$a$};
\draw (5,3) circle (3pt) node[above] {$b$};
\end{tikzpicture}&
\begin{tikzpicture}[scale=0.35]
\draw (0,0)--(1,1)--(2,2);
\draw [dotted] (2,2)--(3,3)--(4,4);
\draw (4,4)--(5,3) node[midway, above] {$a$};
\draw [dotted] (5,3)--(6,4);
\draw (6,4)--(7,3) node[midway, above] {$b$};
\end{tikzpicture}\\
\hline
\multicolumn{3}{p{15cm}}{\vspace{0.01cm}Insert row 3:}\\
\multicolumn{3}{p{15cm}}{$a<\tilde{b}$:}\\
\hline
\end{tabular}
\caption{Possible insertion scenarios}
\label{FigScenarios}
\end{figure}


\begin{thebibliography}{99}

\bibitem{MR1881971}
Jin Hong and Seok-Jin Kang.
\newblock {\em Introduction to quantum groups and crystal bases}, volume~42 of
  {\em Graduate Studies in Mathematics}.
\newblock American Mathematical Society, Providence, RI, 2002.

\bibitem{ONCRYSTALBASES}
Masaki Kashiwara.
\newblock On crystal bases.
\newblock In {\em Canadian Mathematical Society Conference Proceedings}, 1995.

\bibitem{MR3534070}
C.~Krattenthaler.
\newblock Bijections between oscillating tableaux and (semi)standard tableaux
  via growth diagrams.
\newblock {\em J. Combin. Theory Ser. A}, 144:277--291, 2016.

\bibitem{MR3482440}
Jae-Hoon Kwon.
\newblock Super duality and crystal bases for quantum ortho-symplectic
  superalgebras {II}.
\newblock {\em J. Algebraic Combin.}, 43(3):553--588, 2016.

\bibitem{2015arXiv151201877K}
Jae-Hoon Kwon.
\newblock Combinatorial extension of stable branching rules for classical
  groups.
\newblock {\em Trans. Amer. Math. Soc.}, 2018.

\bibitem{MR3604801}
Soichi Okada.
\newblock Pieri rules for classical groups and equinumeration between
  generalized oscillating tableaux and semistandard tableaux.
\newblock {\em Electron. J. Combin.}, 23(4):Paper 4.43, 27, 2016.

\bibitem{MR1043509}
Robert~A. Proctor.
\newblock A {S}chensted algorithm which models tensor representations of the
  orthogonal group.
\newblock {\em Canad. J. Math.}, 42(1):28--49, 1990.

\bibitem{MR2716353}
Thomas~Walton Roby, V.
\newblock {\em Applications and extensions of {F}omin's generalization of the
  {R}obinson-{S}chensted correspondence to differential posets}.
\newblock ProQuest LLC, Ann Arbor, MI, 1991.
\newblock Thesis (Ph.D.)--Massachusetts Institute of Technology.

\bibitem{MR3226822}
Martin Rubey, Bruce~E. Sagan, and Bruce~W. Westbury.
\newblock Descent sets for symplectic groups.
\newblock {\em J. Algebraic Combin.}, 40(1):187--208, 2014.

\bibitem{MR1676282}
Richard~P. Stanley.
\newblock {\em Enumerative combinatorics. {V}ol. 2}, volume~62 of {\em
  Cambridge Studies in Advanced Mathematics}.
\newblock Cambridge University Press, Cambridge, 1999.

\bibitem{MR899903}
John~R. Stembridge.
\newblock Rational tableaux and the tensor algebra of {${\rm gl}_n$}.
\newblock {\em J. Combin. Theory Ser. A}, 46(1):79--120, 1987.

\bibitem{MR2941115}
Sheila Sundaram.
\newblock {\em O{N} {THE} {COMBINATORICS} {OF} {REPRESENTATIONS} {OF} {THE}
  {SYMPLECTIC} {GROUP}}.
\newblock ProQuest LLC, Ann Arbor, MI, 1986.
\newblock Thesis (Ph.D.)--Massachusetts Institute of Technology.

\bibitem{MR1041447}
Sheila Sundaram.
\newblock Orthogonal tableaux and an insertion algorithm for {${\rm
  SO}(2n+1)$}.
\newblock {\em J. Combin. Theory Ser. A}, 53(2):239--256, 1990.

\end{thebibliography}
\end{document}